\documentclass[12pt, oneside]{amsart}

\usepackage{amsmath,amssymb}
\usepackage{graphicx}
\usepackage{longtable}
\usepackage{booktabs,array}
\usepackage{color}
\usepackage{float}
\usepackage{microtype}
\usepackage{placeins}
\usepackage{etoolbox}
\usepackage[textwidth=6in, textheight=9in,marginpar=0.75in]{geometry}
\usepackage[colorlinks=true,linkcolor=blue,citecolor=blue,urlcolor=blue]{hyperref}

\newtheorem{theorem}{Theorem}[section]

\newtheorem{proposition}[theorem]{Proposition}
\newtheorem{lemma}[theorem]{Lemma}

\newcommand{\bbz}{{\mathbb Z}}
\newcommand{\bbq}{{\mathbb Q}}
\newcommand{\dbc}{\Sigma}
\newcommand{\KI}{\textsf{KnotInfo}}
\newcommand{\kn}[2]{\textup{#1#2}}

\newlength{\ctKnot}
\newlength{\ctRange}
\newlength{\ctWidth}

\newcommand{\Reff}{R}

\newcommand{\nExactUpper}{19}

\newcommand{\nLowerImproved}{390}
\newcommand{\nChildren}{11}
\newcommand{\nCyclicFive}{4}
\newcommand{\nCyclicFour}{15}
\newcommand{\nCyclicNew}{78}
\newcommand{\nCyclicThree}{59}
\newcommand{\nDataKnots}{5{,}546}
\newcommand{\nDichotomy}{959}
\newcommand{\nDichotomyHeuristic}{46}
\newcommand{\nExact}{2{,}734}
\newcommand{\nExactLower}{2{,}715}
\newcommand{\nImproved}{380}
\newcommand{\nImprovedC}{3}
\newcommand{\nImprovedL}{91}
\newcommand{\nMcCoyKnots}{27}
\newcommand{\nOpenTwoThreeAlt}{1{,}027}
\newcommand{\nRefKnots}{12{,}965}
\newcommand{\nSigFourImproved}{35}
\newcommand{\nSigFourObstructed}{219}
\newcommand{\nSigFourOpen}{605}

\newcommand{\nSweepExact}{1{,}498}
\newcommand{\nSweepImproved}{215}
\newcommand{\nSweepNewExact}{1{,}496}
\newcommand{\nSweepObstructed}{1{,}711}
\newcommand{\nSweepRankFour}{7}

\newcommand{\nSweepRankThree}{43}
\newcommand{\nSweepRankTwo}{140}
\newcommand{\nSweepTargets}{1{,}967}

\newcommand{\nSweepUone}{1{,}521}

\newcommand{\nUfive}{28}
\newcommand{\nUfour}{172}
\newcommand{\nUthree}{373}
\newcommand{\nUthreeC}{11}
\newcommand{\nUthreeOwens}{226}
\newcommand{\nUtwo}{2{,}161}
\newcommand{\nUtwoC}{64}

\newcommand{\nUtwoL}{720}
\newcommand{\nUtwoM}{44}
\newcommand{\nWithCandidates}{22}

\newcommand{\nGPExact}{194}
\newcommand{\nGPImproved}{30}

\newcommand{\nLatestExactAgreements}{194}

\newcommand{\nLatestRangeAgreements}{57}

\newcommand{\nLatestRangeImprovements}{323}

\newcommand{\nLatestUnresolvedExact}{2{,}540}

\newcommand{\nMcCoyExcluded}{2{,}917}
\newcommand{\nMcCoyUnknotting}{710}
\newcommand{\nMcCoyVerified}{3{,}627}

\makeatletter
\patchcmd{\@settitle}{\bfseries}{\normalfont\scshape}{}%
  {\ClassError{amsart}{Could not set the title in small caps}{}}
\patchcmd{\@settitle}{\uppercasenonmath\@title}{}{}%
  {\ClassError{amsart}{Could not preserve title capitalization}{}}
\patchcmd{\@setauthors}{\MakeUppercase{\authors}}{{\scshape\authors}}{}%
  {\ClassError{amsart}{Could not set the author name in small caps}{}}
\g@addto@macro\@setauthors{%
  \vspace{-1.5ex}
  \begin{center}
    \normalfont\footnotesize
    Center for Geometry and Physics, Institute for Basic Science,\\
    Pohang 37673, South Korea
  \end{center}
}
\makeatother

\hypersetup{pdftitle={Computation of Unknotting Numbers: Which Knot Breaks the Bernhard–Jablan Conjecture?},pdfauthor={Seong-Jin Lee}}

\begin{document}

\title[Computation of Unknotting Numbers]{Computation of Unknotting Numbers:\\
  \mbox{Which Knot Breaks the Bernhard--Jablan Conjecture?}}
\author[Seong-Jin Lee]{Seong-Jin Lee}

\begin{abstract}
We determine the unknotting numbers of \nLatestUnresolvedExact{} prime knots
with at most $13$ crossings whose values are unresolved in the \KI{} snapshot
of 9 September 2026. The lower-bound calculations use Heegaard Floer
correction-term obstructions and Greene's spanning-tree model, with further
constraints supplied by the Casson--Walker invariant. We also give explicit
crossing-change constructions for upper bounds. For \nDichotomy{} alternating knots with range $[2,3]$, we verify
that no crossing change in a minimal diagram gives a knot of unknotting number
one. We determine the unknotting numbers of
all four knots in Brittenham and Hermiller's construction and thereby identify
\kn{13n}{3370} as an explicit counterexample to the original Bernhard--Jablan
conjecture.
\end{abstract}

\maketitle

\section{Introduction}\label{sec:intro}

A \emph{knot} is a smooth embedding of a circle $S^1$ in $S^3$, considered up to
ambient isotopy. A planar diagram records a generic planar projection together with the
overcrossing and undercrossing at each double point. By Reidemeister's theorem,
two diagrams represent the same knot if and only if they are related by planar
isotopy and a finite sequence of the three Reidemeister moves
\cite{LackenbySurvey}.
A \emph{crossing change} interchanges the overpassing and underpassing strands at
one crossing and it may change the knot type. Every knot can be transformed into the unknot by a finite number of crossing changes. 

The \emph{unknotting number} $u(K)$ is the least number of changes required to do so. Equivalently, if $u(D)$ is the least number
of crossings of a diagram $D$ that need to be changed to obtain a diagram of the unknot, then the unknotting number is defined by
\[
    u(K)=\min\{u(D)\mid D\text{ is a diagram of }K\}.
\]
This invariant goes back to Wendt's work on the Gordian number \cite{Wendt}. Since the minimum
is taken over all diagrams of $K$, a computation on one diagram might not determine
the unknotting number of the knot.

In this paper, we use an \emph{unknotting diagram}: a diagram of $K$
with a set of marked crossings whose changes give the unknot. If $r$ crossings are
marked, the diagram proves $u(K)\le r$.

Nakanishi and Bleiler showed that a minimal crossing diagram need not realize the unknotting number
\cite{Nakanishi,Bleiler}: the standard minimal diagram of $10_8$ requires three crossing
changes, whereas $u(10_8)=2$. There is no known
algorithm that computes the unknotting number of every knot
\cite{LackenbySurvey}; the same difficulty is emphasized in recent computational work
\cite{Applebaum}.

The Bernhard--Jablan conjecture relates this difficulty to minimal diagrams. Bernhard and Jablan
proposed that every nontrivial knot has a minimal diagram and a crossing whose
change lowers its unknotting number by one \cite{Bernhard,Jablan}. Let
$N_{\min}(K)$ be a set of the knots obtained by changing one crossing in any minimal
diagram of $K$. For nontrivial $K$, define
\[
    u^s_{\mathrm{BJ}}(K)=1+\min\{u(K')\mid K'\in N_{\min}(K)\}.
\]
This is called the \emph{strong Bernhard--Jablan unknotting number}. The \emph{weak}
number $u^w_{\mathrm{BJ}}(K)$ requires each crossing change in the sequence to
be made in a minimal diagram of the knot reached at that stage. Both numbers
are zero for the unknot. Thus
$u(K)\le u^s_{\mathrm{BJ}}(K)\le u^w_{\mathrm{BJ}}(K)$
\cite[\S1 and Lem.~2.1]{BH1705}. The original conjecture asserts
$u=u^s_{\mathrm{BJ}}$ for every knot. 
A failure of the weak version $u=u^w_{\mathrm{BJ}}$ alone does not determine whether the same knot fails the strong equality
\cite[Cor.~2.2]{BH1705}.

Brittenham and Hermiller proved
$u(\kn{13n}{3370})\le2<3=u^w_{\mathrm{BJ}}(\kn{13n}{3370})$ and showed
that at least one of
\[
    \kn{12n}{288},\quad \kn{12n}{491},\quad \kn{12n}{501},\quad \kn{13n}{3370}
\]
fails the strong equality \cite[Thm.~1.3]{BH1705}. By determining the
unknotting numbers of all four knots to be two, we prove that
\kn{13n}{3370} is an explicit counterexample to the original conjecture.
The Montesinos obstructions determine the values for \kn{12n}{288} and
\kn{12n}{501}, while Greene's spanning-tree model supplies the correction-term
data needed for \kn{12n}{491} and \kn{13n}{3370} \cite{Greene}.
Brittenham and Hermiller's minimal-diagram calculation then gives
$u^s_{\mathrm{BJ}}(\kn{13n}{3370})=3>u(\kn{13n}{3370})=2$
(Theorem~\ref{thm:bj-counterexample}).

We also determine $u(\kn{13n}{1587})=2$ in their second example
(Section~\ref{sec:bj-counterexample}). Brittenham and Hermiller had already shown that
$u(\kn{13n}{1587})\le2<3=u^w_{\mathrm{BJ}}(\kn{13n}{1587})$
\cite[\S3]{BH1705}. The new calculation gives its exact unknotting number, but does not decide
whether this particular knot also fails the strong equality.
For alternating knots with $u=1$, McCoy's theorem gives a positive result:
every alternating diagram contains an unknotting crossing \cite{McCoy}.

Brittenham and Hermiller also disproved additivity under connected sum by
constructing an unknotting sequence of length five for
$7_1\mathbin{\#}\overline{7_1}$, where $7_1=T(2,7)$ and
$\overline{7_1}$ denotes its mirror image. Both summands have unknotting
number three, so this gives
$u(7_1\mathbin{\#}\overline{7_1})\le5<6=u(7_1)+u(\overline{7_1})$
\cite[Thm.~1.2]{BHadd}.

To obtain lower bounds, we use the signature, the homology and linking
pairings of branched covers, and Floer-theoretic invariants. Applying these
obstructions together is useful when individual tests leave more than one
possible surgery description or sign pattern.

Greene's spanning-tree model makes correction terms accessible for further
branched covers, beyond the alternating and Montesinos cases \cite{Greene}.
When Khovanov homology establishes the $L$-space property, the model gives
finite sets of possible correction terms \cite{OSbranched,Greene}.
Testing all candidates against the surgery constraints can prove a lower bound
even when some correction terms remain undetermined \cite{Owens,NiWu}.
The Casson--Walker invariant also fixes the sum of the correction terms of an
$L$-space, which can resolve ambiguities left by the diagrammatic model
\cite{Rustamov,GreeneWatson}.
The necessary theory is reviewed in Section~\ref{sec:greene-background}.

On the upper-bound side, braid simplification, reinforcement learning, and geometric
simplification have been used to search for unknotting sequences in diagrams that
need not be minimal \cite{GukovUnknot,Applebaum,DKT,ReAPR}. The aim is to find a
diagram that can be transformed into the unknot by fewer crossing changes than are
required in a given minimal diagram. Such a diagram may have more crossings.

Our aim is to investigate how theoretical obstructions and unknotting diagrams
can be used together to compute the unknotting number.
Applied to knots with at most $13$ crossings, these methods determine \nExact{}
unknotting numbers and narrow \nImproved{} further ranges relative to the
archived comparison ranges specified in Section~\ref{sec:lower}
(Theorem~\ref{thm:main}). Lower bounds account for
\nExactLower{} exact values, including \nUfour{} values equal to four and
\nUfive{} equal to five from the rank-three and rank-four obstructions.
Explicit crossing-change constructions give the remaining \nExactUpper{}
exact values; ten of these also require the lower bounds obtained here. The computations using Greene's model contribute
\nSweepNewExact{} additional exact values (Section~\ref{sec:sweep}).

The count \nExact{} includes independent recoveries for 194 knots in recent concurrent
results\cite{GP}. In the \KI{} snapshot retrieved on 9 September 2026,
\nLatestExactAgreements{} of these knots have the same exact values and
\nLatestUnresolvedExact{} remain undetermined.
Section~\ref{sec:lower} explains the comparison with the database.

Alongside the numerical results, we ask how much can be learned from minimal diagrams.
For prime alternating knots, the flyping theorem makes it possible to compare crossing
changes across all minimal diagrams \cite{MT,BH1705}. For \nDichotomy{} alternating
knots with range $[2,3]$, we verify that none of these changes gives a knot of
unknotting number one. Any such knot with $u=2$ would be an alternating
counterexample; determining whether one has that value remains open.

The paper is organized as follows. Section~\ref{sec:background} provides a brief review of the relevant theory. Section~\ref{sec:lower} presents
the lower bounds, including applications of Owens's obstruction to three and four crossing
changes. Section~\ref{sec:upper} reviews upper-bound methods and gives our unknotting
constructions. In Section~\ref{sec:bj}, we revisit the examples of Brittenham and Hermiller
and study crossing changes in minimal alternating diagrams.
Section~\ref{sec:conclusion} places the results alongside concurrent work and
returns to the question of a counterexample among alternating knots.
The appendices contain the complete result tables and the diagram data.

\section{Background}\label{sec:background}

We denote $\det K=|\Delta_K(-1)|$ for the determinant where $\Delta_K(z)$ is the Alexander polynomial, $\sigma(K)$ for the signature, and
$\dbc(K)=\dbc_2(K)$ for the double cover of $S^3$ branched over $K$. Mirroring preserves
$u(K)$ and changes the sign of the signature. We use the convention in which changing a
negative crossing can decrease the signature by two.

\subsection{Signature and four-ball genus}

The signature gives one of the simplest lower bounds. Algebraically, it counts positive
eigenvalues minus negative eigenvalues of a symmetric form associated with a Seifert surface of the knot;
this integer is independent of the diagram used to define the form
\cite[\S1, Thm.~3.1]{Murasugi}. If $K'$ is obtained from $K$ by changing a positive crossing, then
$\sigma(K')\in\{\sigma(K),\sigma(K)+2\}$; for a negative crossing,
$\sigma(K')\in\{\sigma(K),\sigma(K)-2\}$
\cite{Murasugi,CL}. Since the unknot has signature zero,
\begin{equation}\label{eq:signature-bound}
    \frac{|\sigma(K)|}{2}\le u(K).
\end{equation}
The signs contain more information than this inequality alone. If $\sigma(K)=2r$ and
$u(K)=r$, every change in an unknotting sequence of length $r$ must be negative and must
lower the signature by two. We call this the case in which the signature bound is sharp.

Crossing changes also relate unknotting to four-dimensional topology. The smooth
\emph{four-ball genus} $g_4(K)$ is the least genus of a smooth, connected, oriented surface
properly embedded in $B^4$ with boundary $K$. An unknotting sequence of length $r$ gives an
immersed disk with $r$ double points; resolving these points gives an embedded surface of
genus $r$. Hence $g_4(K)\le u(K)$ \cite{CL}. In particular,
\begin{equation}\label{eq:concordance}
    \max\left\{\frac{|\sigma(K)|}{2},\,|\tau(K)|,\,
                      \frac{|s(K)|}{2}\right\}\le g_4(K)\le u(K),
\end{equation}
where $\tau$ is the Ozsv\'ath--Szab\'o invariant and $s$ is Rasmussen's invariant
\cite{OStau,Rasmussen}. These bounds form part of the information already recorded in knot
tables, but need not be sharp. A knot with $g_4(K)=0$ is called
\emph{slice}, and Nakanishi constructs prime slice knots with arbitrarily large
unknotting number \cite[\S3.1]{Nakanishi81}. For such knots, a useful lower
bound must detect restrictions on crossing changes beyond the four-ball genus.

\subsection{Branched covers and linking pairings}

The Montesinos trick turns crossing changes into surgery descriptions of branched covers
\cite{Montesinos}. For one crossing change, the ball containing the crossing lifts to a
solid torus. Replacing the crossing changes the Dehn filling on its complement.
Here Dehn filling means gluing in a solid torus with a specified meridian curve on the
boundary. The meridians for the two crossing choices intersect twice, which accounts
for the denominator two in the surgery coefficient \cite[Lem.~3.2]{Owens}.
If the crossing change unknots $K$, this describes $\dbc(K)$ as half-integral surgery
on a knot in $S^3$. Iterating the construction gives a surgery description on a link. The following
formulation keeps track of the signs \cite[Thm.~3]{Owens}.

\begin{theorem}[Montesinos--Owens]\label{thm:montesinos}
Suppose that $K$ can be unknotted by changing $p$ positive and $n$ negative crossings.
Then $\dbc(K)$ is obtained by surgery on a $(p+n)$-component link in $S^3$ with linking
matrix $\tfrac12Q$, where $Q$ is integral and symmetric and $Q\equiv I\pmod2$.
If $n=\sigma(K)/2$, then $Q$ is positive definite and exactly $n$ of its diagonal entries
are congruent to $3$ modulo $4$.
\end{theorem}

The definiteness condition allows us to apply the Heegaard Floer
correction-term inequality to the resulting surgery form.
Owens and Strle establish an analogous intersection-form theorem for normally
immersed disks in $B^4$ \cite[Thm.~1]{OwensStrle}.

The homology and linking pairing of $\dbc(K)$ can be read directly from a diagram.
Checkerboard-colour the regions of a connected knot diagram and put one vertex in each
shaded region, with an edge at each crossing. Give each edge its Goeritz sign and retain the
cyclic order of the edges at each vertex. This is the \emph{embedded signed Tait graph}.
Its signed medial graph recovers the knot diagram, and changing a crossing reverses the
corresponding edge sign. Figure~\ref{fig:tait} illustrates the construction for $6_2$:
the shaded regions of the knot diagram become vertices, and each crossing gives an
edge joining its two incident shaded regions. The Goeritz form of this graph describes
the homology and linking pairing of the double branched cover \cite{GL,Lickorish}.
The graph will also give the crossing-change formula in Section~\ref{sec:electrical}.

\begin{figure}[htbp]
\centering
\includegraphics[width=0.55\textwidth]{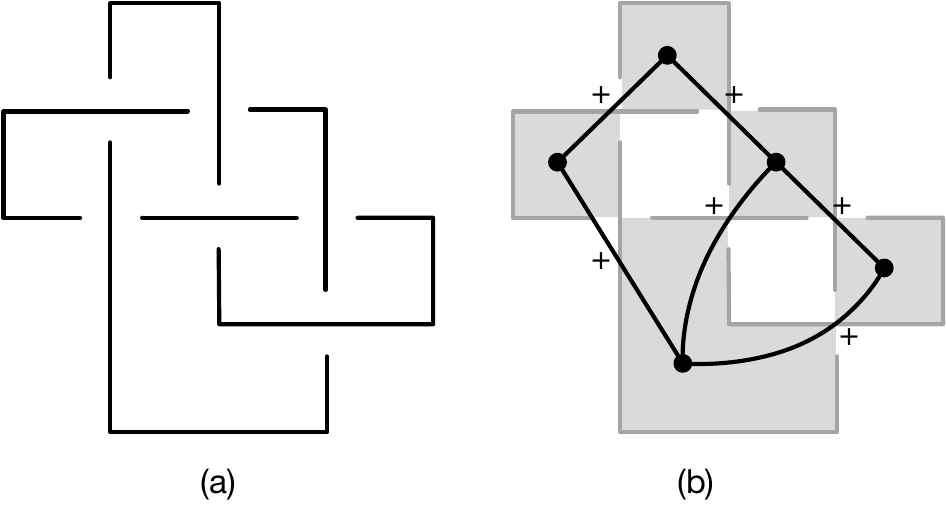}
\caption{(a) An alternating diagram of $6_2$ and (b) its Tait graph. Vertices correspond to shaded
regions and edges to crossings. For an alternating diagram all Goeritz signs agree, for
either fixed checkerboard colouring.}\label{fig:tait}
\end{figure}

The \emph{Goeritz matrix} $G$ is obtained by deleting one row and column of the signed
Laplacian: before deletion the off-diagonal entry is minus the sum of the edge signs
between the two vertices, and the diagonal entry is the sum of the incident non-loop edge
signs. It presents $H_1(\dbc(K))$, and
\[
    |\det G|=\det K,\qquad
    \lambda_G([x],[y])=x^TG^{-1}y\pmod{\bbz}
\]
gives the linking pairing, up to the orientation convention for the cover
\cite{GL,Lickorish}. For a reduced alternating diagram, the two checkerboard forms are
definite with opposite signs. The determinant gives the order of the homology group,
whereas the pairing also records a value in $\bbq/\bbz$ for each pair of classes.
It therefore supplies a finer necessary condition.

If $u(K)=1$ and $D=\det K$, the surgery description implies
$H_1(\dbc(K))\cong\bbz/D$ with a generator of self-linking $\pm2/D$. This is
Lickorish's obstruction \cite{Lickorish}. If a chosen generator has self-linking $a/D$,
the condition becomes
\begin{equation}\label{eq:lickorish}
    a\equiv\pm2k^2\pmod D\quad\text{for some }k\text{ coprime to }D.
\end{equation}
Changing the generator multiplies its self-linking by the square of a unit modulo $D$.
The congruence therefore asks whether any generator has the self-linking required by
surgery, independently of the generator first chosen.
Failure of cyclicity or of this congruence proves $u(K)\ge2$. The sign ambiguity makes
this test independent of the orientation chosen for the double cover.

\subsection{Cyclic covers}\label{sec:generator-background}

A surgery description on $r$ components presents $H_1(\dbc(K))$ with at most $r$
generators. Thus $u(K)$ is at least the number of generators needed for this group.
There is a corresponding bound for higher covers \cite{Wendt,Nakanishi81}. Let $g_m(K)$ be
the least number of generators of $H_1(\dbc_m(K);\bbz)$, including its free part. Then
\begin{equation}\label{eq:cyclic}
    u(K)\ge\left\lceil\frac{g_m(K)}{m-1}\right\rceil,\qquad m\ge2.
\end{equation}
These inequalities can detect information invisible in the double cover. The
\emph{Alexander module} is the first homology of the infinite cyclic cover of the knot
exterior. It is a module over $\bbz[t,t^{-1}]$, where $t$ acts by a generator of the
deck transformations \cite[\S1]{Nakanishi81}. The Nakanishi index is the least size of a
square presentation matrix for this module. It also gives a lower bound on $u(K)$
\cite[Thm.~3]{Nakanishi81}.

If an abelian group is written in invariant-factor form as
$\bbz^b\oplus\bbz/d_1\oplus\cdots\oplus\bbz/d_t$, with
$1<d_1\mid\cdots\mid d_t$, its least number of generators is $b+t$. Equivalently it is
the maximum, over primes $q$, of the dimension after tensoring with $\mathbb F_q$.
Each free summand contributes one generator, just as each nontrivial invariant factor does.
This distinguishes the number of generators from the order of the group: a large
determinant alone need not give a large unknotting number.

\subsection{Definite fillings and correction terms}

The linking pairing records a necessary part of a surgery description. The correction
terms of Ozsv\'ath and Szab\'o impose further restrictions on the definite four-manifold
bounded by the cover \cite{OS03}. We use $\mathbb F_2$ coefficients for Heegaard Floer
homology. For a rational homology sphere $Y$ and a
$\mathrm{Spin}^c$ structure $\mathfrak s$, the correction term
$d(Y,\mathfrak s)\in\bbq$ is the least grading of a nonzero homogeneous element in the
image of $HF^\infty$ in $HF^+(Y,\mathfrak s)$. It records the grading at the bottom of
the infinite $U$-tower in this homology group and depends only on the oriented boundary
and its $\mathrm{Spin}^c$ structure \cite[\S4]{OS03}. Its relation with intersection
forms gives a finite obstruction once the possible surgery forms and the correction
terms of $Y$ are known.

For a positive-definite integral matrix $A$ of rank $r$ and odd determinant, write
$\operatorname{Char}(A)=\{\xi\in\bbz^r:\xi_i\equiv A_{ii}\pmod2\}$ and define
\begin{equation}\label{eq:m-form}
    m_A(g)=\min_{\substack{\xi\in\operatorname{Char}(A)\\
                          [\xi]=g\in\bbz^r/A\bbz^r}}
             \frac{\xi^TA^{-1}\xi-r}{4}.
\end{equation}
Since $\det A$ is odd, the map $\xi\mapsto[\xi]$ identifies characteristic covectors
modulo $2A\bbz^r$ with the discriminant group and places the spin structure at zero,
as in \cite[\S4]{Owens}. When $A$ is an intersection matrix, these covectors represent
first Chern classes of $\mathrm{Spin}^c$ structures on the filling. The expression
$\xi^TA^{-1}\xi$ is their square, and the minimum compares all extensions of a fixed
boundary structure \cite[\S4]{Owens}.

\begin{theorem}[Ozsv\'ath--Szab\'o; Owens]\label{thm:dmatch}
Let $X$ be a smooth, simply connected positive-definite four-manifold with boundary $Y$,
where $|H_1(Y)|$ is odd, and intersection matrix $A$. There is an isomorphism
$\phi:\bbz^r/A\bbz^r\longrightarrow\mathrm{Spin}^c(Y)$, with the spin structure as
origin, such that for every $g$,
\begin{equation}\label{eq:dmatch}
    m_A(g)\ge d(Y,\phi(g)),\qquad
    m_A(g)\equiv d(Y,\phi(g))\pmod2.
\end{equation}
\end{theorem}

This is the positive-definite form of the correction-term inequality
\cite{OS03,Owens}. Both conditions must hold under one group isomorphism for all
classes simultaneously; an arbitrary reordering of the correction terms is not
sufficient \cite[Thm.~4.1]{Owens}. A filling is called \emph{sharp} if equality can be
achieved for every boundary structure. For an alternating knot, the positive-definite
checkerboard Goeritz form gives a sharp filling. Consequently, its function $m_G$
determines the correction terms exactly, so the same diagram supplies the boundary
invariants against which other definite fillings can be tested \cite[Prop.~3.2]{OS05}.
We use its boundary orientation convention to compare $m_G$ with the correction terms
of $\dbc(K)$ \cite{OS05,Owens}. Reversing the orientation changes every correction term
to its negative, which must also be taken into account when passing to the mirror
\cite[Prop.~4.2]{OS03}. For alternating $K$, the spin correction term is
$-\sigma(K)/4$ \cite{MO}, which also provides a check on the convention.

It is sometimes enough to exhibit a single characteristic covector in a
given class. For any such $\xi$ representing $g$, set
$E_A(\xi)=(\xi^TA^{-1}\xi-r)/4$. By \eqref{eq:m-form},
$m_A(g)\le E_A(\xi)$, and the two quantities agree modulo two
\cite[\S4]{Owens}. Thus an identification $\phi$ fails
\eqref{eq:dmatch} if $E_A(\xi)<d(Y,\phi(g))$ or if their residues modulo two
differ. Such a covector certifies the obstruction without requiring the
minimum in \eqref{eq:m-form} to be evaluated in every class.

To apply the inequality to Theorem~\ref{thm:montesinos}, write
$Q_{ii}=2m_i-1$ and $Q_{ij}=2a_{ij}$ for $i\ne j$. Replacing the half-integral surgery description by
integral surgery produces a smooth, simply connected $2$-handlebody $X$ with
boundary $\dbc(K)$ \cite[Lem.~2.2 and proof of Thm.~5]{Owens}. Its intersection
form is
\begin{equation}\label{eq:qtilde-n}
    \widetilde Q=\begin{pmatrix}M&I\\ I&2I\end{pmatrix},
    \qquad M_{ii}=m_i,\quad M_{ij}=a_{ij},\quad 2M-I=Q.
\end{equation}
Its determinant is $\det Q=\det K$. If $\sigma(K)=2n$ and $u(K)=n$, then every $m_i$
is even. Some such form must satisfy \eqref{eq:dmatch} for the correction terms of
$\dbc(K)$ \cite[Thm.~5]{Owens}. If every form fails, the signature bound improves from
$n$ to $n+1$. The applications to $n=3$ and $n=4$ below use this existing theorem in
higher ranks.

For a normally immersed disk with $n=\sigma(K)/2$ negative double points,
Owens and Strle obtain the same type of definite form, with determinant dividing
$\det K$ and square quotient \cite[Thm.~1]{OwensStrle}.
If $\det K$ is square-free and $|\sigma(K)|=2n$, the obstruction above therefore
also rules out a normally immersed disk in $B^4$ bounded by $K$ with $n$ double
points. It follows that the resulting lower bound applies to the
\emph{four-dimensional clasp number}, the minimum number of transverse double
points of a smooth, properly immersed disk in $B^4$ with boundary $K$ and no other
singularities, and to the \emph{slicing number}, the minimum number of crossing changes
needed to obtain a slice knot. These invariants satisfy
$c_4(K)\le u_s(K)\le u(K)$ \cite[\S1]{OwensStrle}.

For an unknotting sequence of two crossing changes with $n=\sigma(K)/2$ negative
crossing changes, a convenient reduced form is
\begin{equation}\label{eq:qtilde}
    \begin{gathered}
    Q=\begin{pmatrix}2m_1-1&2a\\2a&2m_2-1\end{pmatrix},\qquad 0\le a<m_1\le m_2,\\
    (2m_1-1)(2m_2-1)-4a^2=\det K.
    \end{gathered}
\end{equation}
Exactly $n$ of $m_1,m_2$ must be even \cite[Thm.~1]{Owens}.
The determinant equation and parities may already exclude all forms; otherwise the
correction terms decide which forms remain possible.

\subsection{Floer homology from a knot diagram}\label{sec:greene-background}

A rational homology sphere $Y$ is an \emph{$L$-space} if
$\dim_{\mathbb F_2}\widehat{HF}(Y)=|H_1(Y;\bbz)|$. Ozsv\'ath and Szab\'o
constructed a spectral sequence from the reduced Khovanov homology of the mirror
of $K$ to $\widehat{HF}(\dbc(K))$ \cite[Thm.~1.1 and Cor.~1.2]{OSbranched}. In particular,
\begin{equation}\label{eq:kh-lspace}
    \det K\le \dim_{\mathbb F_2}\widehat{HF}(\dbc(K))
       \le \dim_{\mathbb F_2}\widetilde{Kh}(\overline K;\mathbb F_2).
\end{equation}
Thus reduced Khovanov homology of dimension $\det K$ proves that the cover is an
$L$-space. The coefficient field matters: a calculation over $\bbq$ alone does not
give this conclusion over $\mathbb F_2$.

The $L$-space property also determines the spin correction term. Let
$Y=\dbc(K)$ have the orientation induced from $S^3$, and let $\mathfrak s_0$
be its unique spin structure. If $Y$ is an $L$-space over $\mathbb F_2$,
the graded universal coefficient theorem shows that it is an $L$-space over
$\bbq$ as well, with the same correction terms: the unique free generator
in each $\mathrm{Spin}^c$ summand of integral $\widehat{HF}$ has the same
absolute grading over both fields. Since the reduced Floer homology
vanishes, Lin, Ruberman, and Saveliev's formula for the covering involution
\cite[Thm.~A and Rem.~1.1]{LRS} gives
\begin{equation}\label{eq:lspace-spin}
    d(Y,\mathfrak s_0)=-\frac{\sigma(K)}4.
\end{equation}
This supplies the spin value used in the
rank comparisons of Section~\ref{sec:sweep}.

Quasi-alternating links provide another way to establish the $L$-space property. They
form the smallest class containing the unknot and closed under the following
rule: a link belongs to the class if it has a crossing whose two resolutions are
in the class and have positive determinants adding to the determinant of the
link. The double branched cover of a quasi-alternating link is an $L$-space
\cite[Def.~3.1 and Prop.~3.3]{OSbranched}. A resolution tree ending in unknots or
non-split alternating links therefore supplies a finite certificate.

Greene's model makes further information accessible directly from a diagram.
Greene used it to compute the Floer homology of the branched covers of the
$20$ non-alternating, non-Montesinos knots with at most $10$ crossings
\cite[\S7]{Greene}. Here we use the model for a larger collection of knots
and combine the resulting correction-term data with surgery obstructions.
Mark an edge and its two adjacent regions. A Kauffman state chooses
one corner at each crossing so that every unmarked region receives exactly one
choice. Let $G$ be the reduced white Goeritz matrix and $m$ its size. We use
Greene's incidence numbers $\mu$: an off-diagonal entry is the sum of $\mu$ over
the edges joining the two regions, and the unreduced matrix has row sums zero.
Orient each edge of the white graph towards the white region on the same side
of the overstrand as the chosen corner. The vector $v_x$ records outdegree minus
indegree at the unmarked white vertices, while $\delta(x)$ counts white-corner
choices at crossings with $\mu=-1$. Greene's absolute grading formula is
\cite[Thm.~4.1]{Greene}
\begin{equation}\label{eq:greene-grading}
    \operatorname{gr}(x)=\delta(x)
       +\frac{v_x^TG^{-1}v_x-2m-3\sigma(G)}{4}.
\end{equation}
The vector $v_x$ also gives the first Chern class of the associated boundary
$\mathrm{Spin}^c$ structure in $\operatorname{coker}G$
\cite[\S4.5]{Greene}.

A state is \emph{solitary} if no other Kauffman state induces the same orientation
of the white graph \cite[Lem.~6.1 and Def.~6.2]{Greene}. These states generate the $E_1$ page
of a spectral sequence over $\mathbb F_2$ converging to $\widehat{HF}(\dbc(K))$
\cite[Thm.~6.8]{Greene}. Although that theorem is stated over $\bbz$, the
acyclicity argument for the nonsolitary summands uses differentials with
coefficient $\pm1$ and therefore remains valid over $\mathbb F_2$
\cite[Lems.~6.6--6.7]{Greene}. The rank-one summands give the solitary
generators. Greene's filtration is defined separately on each
$\mathrm{Spin}^c$ summand of $\widehat{CF}$ and is generated by homogeneous
Kauffman states. The associated graded complex inherits the absolute Maslov
grading in \eqref{eq:greene-grading}; the Floer differential, and hence each induced
differential, lowers this grading by one \cite[Thm.~4.1 and \S6.1]{Greene}.
Passage to later pages therefore cannot introduce a grading absent from $E_1$.
If the cover is an $L$-space, each $\mathrm{Spin}^c$ summand of the limit is
one-dimensional and its surviving class has grading $d$. Its correction term
must consequently be one of the solitary-state gradings in that summand.
Changing the marking gives further candidate sets for the same invariant
\cite[\S7.1]{Greene}.

There is a further constraint on these candidates. Normalize the
Casson--Walker invariant $\lambda_{\mathrm{CW}}$ to have value $-1$ on the
boundary of the negative-definite $E_8$ plumbing. For an $L$-space $Y$,
Rustamov's formula for the sum of the renormalized Euler characteristics
gives $\sum_{\mathfrak s}d(Y,\mathfrak s)=-2|H_1(Y)|\lambda_{\mathrm{CW}}(Y)$
\cite[Thm.~3.3]{Rustamov}. If $Y=\dbc(K)$ and $D=\det K$, Mullins's formula
expresses this invariant in terms of the Jones polynomial and signature;
the normalization in \cite[Thm.~13]{GreeneWatson} is twice the one used here.
With $V_K$ normalized by $V_{\mathrm{unknot}}=1$ and the orientation
convention of \eqref{eq:lspace-spin}, the resulting identity is
\begin{equation}\label{eq:casson-sum}
    \sum_{\mathfrak s\in\mathrm{Spin}^c(Y)}d(Y,\mathfrak s)
    =-2D\lambda_{\mathrm{CW}}(Y)
    =-\frac{D\sigma(K)}4+\frac{D V'_K(-1)}{6V_K(-1)}.
\end{equation}
Every candidate vector must have this sum. In particular, if the sum of the
least candidate in each class already equals the right-hand side, each
correction term is determined to be that least candidate. The $L$-space
hypothesis is needed here to eliminate the reduced Floer homology terms
from Rustamov's formula.

The $L$-space condition places each correction term among the solitary-state
gradings. Owens's inequality \eqref{eq:dmatch} itself applies to rational
homology spheres without this hypothesis \cite[Thms.~4.1 and~5]{Owens}.
Even if the candidate sets contain more than one grading, they give an
obstruction when every possible correction-term vector fails the necessary
surgery conditions, since the actual vector is among those tested. This
allows us to apply the inequality using Greene's model when neither a sharp
Goeritz form nor a plumbing description is available
\cite[Thms.~4.1 and~5]{Owens}.

\subsection{Polynomial and torsion bounds}

The Jones polynomial gives a signed refinement of a generator bound. At
$\omega=e^{i\pi/3}$ one has
\begin{equation}\label{eq:lm}
    V_K(\omega)=\varepsilon_K(i\sqrt3)^{d_K},\qquad
    d_K=\dim_{\mathbb F_3}H_1(\dbc(K);\mathbb F_3),\quad
    \varepsilon_K\in\{1,-1\}
\end{equation}
\cite{LM,Jones}. The absolute value determines $d_K$, while $\varepsilon_K$ retains
information about the signs of crossing changes. A crossing change alters $d_K$ by
at most one. If an unknotting sequence has length $d_K$, every change lowers $d_K$,
and the Jones skein relation implies
\begin{equation}\label{eq:traczyk}
    \varepsilon_K=(-1)^{\text{number of negative crossing changes}}.
\end{equation}
This is the signed constraint used in Traczyk's criterion \cite{Traczyk}; Owens combines
it with correction terms in the example $9_{35}$ \cite{Owens}. It is useful when the
signature leaves more than one possible sign pattern: a pattern compatible with the
signature may still have the wrong parity in \eqref{eq:traczyk}.

Torsion in knot homology gives another kind of lower bound. Over $\mathbb F_2[U]$,
$\mathrm{HFK}^-(K)$ consists of a free summand and a torsion submodule
\cite[Def.~1.1]{JMZ}. Let $\operatorname{Ord}_U(K)$ be the least $N\ge0$ such that $U^N$
annihilates this torsion submodule. When torsion is present, it is the largest exponent $a$ for which
$\mathbb F_2[U]/(U^a)$ is a torsion summand, rather than the number of torsion generators.
This torsion order satisfies
\begin{equation}\label{eq:torsion}
    \operatorname{Ord}_U(K)\le u(K)
\end{equation}
\cite{AE,JMZ}. Crossing-change maps compose to multiplication by the
polynomial variable, so an unknotting sequence forces a power of $U$ to
annihilate all torsion, since the knot Floer homology of the unknot is torsion-free
\cite[\S4]{AE}.

For the corresponding bounds in deformed Khovanov theories, the coefficient conventions
must be specified: Alishahi--Dowlin use $X$-torsion in the Lee theory over
$\bbq[X,t]/(X^2-t)$, whereas Alishahi uses $h$-torsion in the characteristic-two
Bar-Natan theory \cite[Thm.~1.2]{AD}\cite[Thm.~1.2]{AlishahiBN}. These invariants need not
be determined by $\tau$, $s$, or the linking pairing. The torsion results reported in
Section~\ref{sec:lower} use the knot Floer bound.

\subsection{Alternating diagrams}\label{sec:alternating-background}

A reduced alternating diagram realizes the crossing number of its knot
\cite{Kauffman,MurasugiJones,ThisSpan}. The flyping theorem states that any two reduced
alternating diagrams of a prime alternating knot are related by flypes \cite{MT}.
A flype rotates a tangle and moves an adjacent crossing to its other side. Changing a
crossing before a flype gives the same knot type as changing the corresponding crossing
afterwards; at the active crossing one may need the opposite flype. Thus the crossing
neighbors of a knot can be compared across minimal diagrams
\cite[Lem.~2.4]{BH1705}.

\begin{lemma}\label{lem:flype}
For a prime alternating knot, the set of knot types obtained by changing one crossing in
a minimal diagram is independent of the minimal diagram.
\end{lemma}

Indeed, every minimal diagram of an alternating knot is alternating, and the flyping
theorem and the preceding observation apply; see also \cite[\S1]{McCoy}. In particular,
to determine whether a crossing change gives a specified knot type, it suffices to
examine a single reduced alternating diagram. The conclusion concerns the neighboring
knot types; after the crossing change, the displayed diagram need not be minimal or
alternating.

\begin{theorem}[McCoy {\cite[Thm.~1]{McCoy}}]\label{thm:mccoy}
A nontrivial alternating knot has unknotting number one if and only if every alternating
diagram of the knot contains an unknotting crossing.
\end{theorem}

McCoy also proves that these conditions are equivalent to $\dbc(K)$ being obtained by
half-integral surgery on a knot in $S^3$. For alternating knots this gives a converse to
the Montesinos trick: the existence of the surgery description forces an unknotting
crossing in the diagram \cite[Thm.~1 and \S7.4]{McCoy}.

This theorem turns a finite examination of a diagram into a lower-bound argument. If
$K$ is nontrivial and no crossing change in a reduced alternating diagram gives the
unknot, then $u(K)>1$. For larger unknotting numbers no such general conclusion follows:
the Bernhard--Jablan conjecture is an additional condition, even when all minimal diagrams
have been examined.

\section{Lower-bound computations and results}\label{sec:lower}

We now apply the obstructions of Section~\ref{sec:background}. A computed lower bound
determines the unknotting number when it reaches an independently known upper bound.
\subsection{Comparison with KnotInfo}\label{sec:reference}

We compare our computations for \nRefKnots{} prime knots with at most $13$
crossings with the \KI{} snapshot retrieved on 9 September 2026 \cite{KnotInfo}.
Of the \nExact{} exact values reported here, \nLatestUnresolvedExact{} are
undetermined in that snapshot and \nLatestExactAgreements{} are already
recorded with the same values. The downloaded data and the complete
comparison are deposited with the code.

Our computations independently recover the \nGPExact{} exact values in Gebel
and Prangley's Theorems~1 and~2, as well as their \nGPImproved{} improved
ranges \cite{GP}.

For the counts of improvements, let $[\ell(K),h(K)]$ denote the archived
starting range.\footnote{The aggregate totals and the starting ranges use the
archived research table, obtained from the August 2026 snapshot by retaining
the April 2026 ranges for the entries specified in the accompanying provenance
records \cite{KnotInfoSnapshots}.} Appendix~\ref{app:improve} records these starting ranges and the
resulting ranges for knots that remain undetermined.

Of the \nImproved{} improved ranges measured from the archived starting
ranges, \nLatestRangeAgreements{} are already recorded in the 9 September
snapshot; the remaining \nLatestRangeImprovements{} are narrower than its
ranges.

For \kn{13n}{3370}, we combine Brittenham and Hermiller's published upper
bound $u\le2$ \cite[Thm.~1.3(a)]{BH1705} with the lower bound in
Section~\ref{sec:bj}. The archived range $[1,3]$ remains the basis of
comparison, with the published upper bound recorded separately.

\subsection{Linking pairings}

For a knot with range $[1,b]$, the first question is whether its double branched cover
could arise from one crossing change. We apply the cyclicity and self-linking conditions
of \eqref{eq:lickorish} to the knots with these ranges. The obstruction excludes $u=1$
for $815$ non-alternating $13$-crossing knots. Of these, \nUtwoL{} have upper bound two,
so their unknotting number is two. A further \nImprovedL{} ranges improve to $[2,b]$;
four of the obstructed knots are settled by the stronger cyclic-cover bound below.

We compute the linking pairing from three presentations: a Goeritz matrix,
a Seifert matrix, and a triangulation of the double branched cover in Regina
\cite{Regina}. All three give the same obstruction list, and the test excludes
none of the $1{,}516$ knots with recorded unknotting number one.

\subsection{Torsion bounds}\label{sec:torsion}

The knot Floer torsion bound detects four knots among the
$3{,}002$ knots in the torsion scan with comparison range $[1,b]$:
\[
    \kn{13n}{689},\quad\kn{13n}{1166},\quad
    \kn{13n}{2504},\quad\kn{13n}{2807}.
\]
Each has torsion order two. The computations use the knot Floer program of \cite{HFK}
and the bound of \cite{AE,JMZ}.
Thus $u(\kn{13n}{1166})=u(\kn{13n}{2504})=2$, and the range of \kn{13n}{689}
improves from $[1,3]$ to $[2,3]$. For \kn{13n}{2807}, this confirms the obstruction from the linking pairing. The knots \kn{13n}{1166} and \kn{13n}{2504} show why it is useful to retain
both kinds of invariant: their linking pairings allow $u=1$, while their homological
torsion excludes it.

\subsection{Two crossing changes}\label{sec:owens}

Owens's obstruction compares possible half-integral surgery forms with the correction
terms of the double cover \cite{Owens}. For alternating knots the latter are obtained
from a Goeritz form. We distinguish the sign patterns permitted by the signature, since
excluding one pattern does not necessarily exclude unknotting number two.

If $\sigma(K)=4$ and two crossing changes unknot $K$, both crossings must be negative.
The forms in \eqref{eq:qtilde} must therefore have both $m_1$ and $m_2$ even. We enumerate
the possible forms and apply Theorem~\ref{thm:dmatch}, comparing their
correction-term bounds with $m_G$ over the isomorphisms of the discriminant
groups. Among the \nSigFourOpen{} knots with
$|\sigma|=4$ and comparison range $[2,3]$, the obstruction excludes $u=2$ for
\nSigFourObstructed{} knots. It also improves \nSigFourImproved{} ranges from $[2,4]$
to $[3,4]$. For \kn{13a}{2890}, the determinant and parity conditions already leave
no candidate form, so the correction-term comparison is unnecessary.

For $\sigma(K)=2$, by contrast, two sign patterns remain: $(p,n)=(1,1)$ and $(0,2)$.
The first satisfies $n=\sigma(K)/2$ and is covered by the definite obstruction, with
one of $m_1,m_2$ even. The second does not satisfy that hypothesis. When
$d_K=2$ and $\varepsilon_K=-1$ in \eqref{eq:lm}, Traczyk's constraint
\eqref{eq:traczyk} excludes two negative changes. Combining the two obstructions gives
\[
    u(\kn{12a}{634})=u(\kn{13a}{2008})
    =u(\kn{13a}{2745})=u(\kn{13a}{3607})=3.
\]
This is the same combination used by Owens for $9_{35}$ \cite{Owens}.

The generator bound accounts for three further values without the sign hypothesis.
The groups for \kn{13a}{2720} and \kn{13a}{2727} are
$\bbz/3\oplus\bbz/3\oplus\bbz/27$, and the group for \kn{13a}{1786} is
$\bbz/3\oplus\bbz/3\oplus\bbz/33$. Each needs three generators, and the recorded
upper bound is three. The same argument improves the range of \kn{13a}{4877} to
$[3,5]$. In total, we obtain \nUthreeOwens{} values equal to three and improve $36$ further
ranges using the obstruction to two crossing changes, Traczyk's criterion, and the bound from the homology of the double branched cover. The three values
obtained from this homological bound are marked separately in Appendix~\ref{app:u3}.

\subsection{Three and four crossing changes}\label{sec:higher}

Owens's general theorem is not restricted to two changes \cite[Thm.~5]{Owens}.
We apply it to alternating knots with $|\sigma|=6$ and $|\sigma|=8$.
After mirroring, write $\sigma(K)=2n$, where $n=3$ or $4$. If $u(K)=n$, all changes
in a shortest sequence are negative, so a positive-definite rank-$n$ form $Q$ must have
\[
    Q\equiv I\pmod2,\qquad Q_{ii}\equiv3\pmod4,\qquad
    \det Q=\det K.
\]
The associated rank-$2n$ form \eqref{eq:qtilde-n} must satisfy the correction-term
inequality. Consequently, excluding every candidate proves $u(K)\ge n+1$.
The enumeration makes the obstruction exhaustive: by Theorem~\ref{thm:montesinos},
any unknotting sequence of the prescribed length would yield an admissible surgery
form. Once all such forms are excluded, the conclusion applies to every diagram of
the knot.

The finiteness of the candidate list is part of this argument. For fixed rank and
determinant, positive-definite integral forms have only finitely many equivalence
classes under integral changes of basis. Owens's theorem uses the finer equivalence
under changes of basis congruent modulo two to permutation matrices; these preserve
the required parity conditions \cite[Thm.~5 and \S5]{Owens}. This subgroup has finite
index, so the finer set of classes is still finite. An exhaustive application must
include every admissible class; a reduced form may require a change of basis before
the congruences become visible. We test the corresponding
discriminant groups, linking pairings, and correction terms. The accompanying code and
data record the enumeration and the individual comparisons in ranks three and four.

For alternating knots with $|\sigma|=6$, the computation determines $129$
values equal to four and improves $12$ ranges from $[3,5]$ to $[4,5]$.
The exact values include \kn{12a}{107}, \kn{13a}{647}, \kn{13a}{648},
and \kn{13a}{660}. In each case the obstruction excludes three crossing
changes, and the retained upper bound is four.

For alternating knots with $|\sigma|=8$, the obstruction determines $21$
values equal to five by excluding $u=4$. The extension to non-alternating
knots in Section~\ref{sec:sweep} gives \nSweepRankThree{} additional values equal to four
and \nSweepRankFour{} additional values equal to five. Together, these computations
determine \nUfour{} values equal to four and \nUfive{} equal to five,
listed in Appendices~\ref{app:u4} and~\ref{app:u5}.

\subsection{Cyclic-cover bounds}\label{sec:cyclic}

We evaluate \eqref{eq:cyclic} using the homology of the covers $\dbc_m(K)$ for
$2\le m\le9$ recorded in \KI{}. Higher covers provide \nCyclicNew{} additional changes
to the comparison ranges: \nUtwoC{} knots have $u=2$, \nUthreeC{} have $u=3$, and
\nImprovedC{} further ranges improve. The effective bound comes from the threefold
cover in \nCyclicThree{} cases, the fourfold cover in \nCyclicFour{} cases, and the
fivefold cover in \nCyclicFive{} cases.

These computations use the full first homology, including free summands, and not just its
torsion subgroup. The Alexander-module calculations in \cite{DGKTdata} give closely
related bounds by specialization over finite fields. In the stored comparison, our
calculation using the tabulated covers recovers $77$ of the $78$ corresponding conclusions and
also detects \kn{12n}{873}. The remaining knot, \kn{13n}{4649}, is not detected by the
covers tabulated here. These overlapping values are independent recomputations of bounds
already present in that data set.

\subsection{Montesinos knots}\label{sec:montesinos}

A \emph{Montesinos knot} is a knot obtained as the numerator closure of a sum of rational
tangles: the tangles are joined in a row, and the two upper endpoints and the two lower
endpoints are joined outside the row. Each rational tangle is obtained from two trivial
arcs in a ball by successive twists of adjacent endpoints \cite[Def.~4.1 and Fig.~3]{MO}.

The surgery obstruction does not require an alternating diagram. For Montesinos knots,
the double branched covers are Seifert fibred spaces and, after reversing orientation
if necessary, bound negative-definite star-shaped plumbings \cite[\S4.1]{MO}.
Their correction terms can then be computed by the methods of
Ozsv\'ath--Szab\'o and N\'emethi \cite{OSplumbed,Nemethi}.

Correction-term obstructions to unknotting number one already appear in the work of
Ozsv\'ath and Szab\'o, including positivity, evenness, and symmetry conditions
\cite[Thm.~1.1 and \S8.4]{OS05}. Owens and Strle extend these conditions using the
formula of Ni and Wu \cite[Thm.~3 and its proof]{OwensStrle}.
If a knot of determinant $D$ has unknotting number one, its cover is
$\pm D/2$-surgery on a knot $C$ in $S^3$. For positive surgery the formula of Ni and Wu
\cite[Prop.~1.6 and Rem.~2.10]{NiWu} gives
\begin{equation}\label{eq:niwu}
 d(S^3_{D/2}(C),i)=d(L(D,2),i)
 -2\max\{V_{\lfloor i/2\rfloor},V_{\lfloor(D+1-i)/2\rfloor}\},
 \qquad 0\le i<D,
\end{equation}
where the nonnegative integers $V_j$ satisfy $V_j-V_{j+1}\in\{0,1\}$ and eventually
vanish \cite{NiWu}. Negative surgery is treated by reversing orientation. Thus, for an unknotting crossing to exist, this correction-term pattern must be
realized under an affine identification of the boundary $\mathrm{Spin}^c$ structures
with $\bbz/D$ that is compatible with the linking pairing.

Reduced Floer homology can distinguish covers whose correction terms alone allow such a
pattern. Lidman's proof that $u(\kn{11n}{102})=2$ illustrates this distinction
\cite{Lidman}. With the orientation used in that proof, the cover has the same correction
terms as $L(3,2)$, but its reduced Floer
homology has ranks $0,0,2$. The mapping cone formula forces a different distribution
among the boundary structures, excluding half-integral surgery. More generally, the
contribution from the reduced parts of the knot Floer mapping cone has the form
\begin{equation}\label{eq:cone}
    \operatorname{rank}HF_{\mathrm{red}}(S^3_{D/2}(C),i)-T_i
       =c_i+c_{i-1},\qquad c_i\in\bbz_{\ge0},\quad i\in\bbz/D,
\end{equation}
where $T_i$ is the finite-dimensional contribution from the tower part of the mapping
cone, after removing its infinite tower. It is determined by the $V_j$
\cite[Cor.~14, Prop.~15]{Gainullin}. Each reduced summand contributes at two adjacent
surgery labels. The ranks can be computed from the plumbing in the cases considered
here \cite{Nemethi}.

Among the $50$ Montesinos knots with comparison range $[1,b]$, these tests exclude
unknotting number one for $49$. They determine \nUtwoM{} values equal to two and improve
$5$ further ranges. The corresponding bounds are also recorded in \cite{DGKTdata}. Our independent
calculation illustrates the role of reduced homology for \kn{12n}{457}. For that knot the correction terms admit the
pattern $V_0=1$, $V_j=0$ for $j\ge1$. Under either surviving identification, however,
the only nonzero reduced ranks occur at the nonadjacent labels $4$ and $8$ in
$\bbz/11$, and $T_i=0$. This contradicts \eqref{eq:cone}, since a positive $c_i$
would make two adjacent ranks positive. Hence $u(\kn{12n}{457})=2$ follows from the
reduced homology test. The only target not excluded is \kn{12n}{309}, of determinant one.

For \kn{12n}{288} and \kn{12n}{501}, the correction terms already exclude unknotting
number one. Combining this obstruction with the recorded upper bounds gives
\[
    u(\kn{12n}{288})=u(\kn{12n}{501})=2.
\]
These are two of the four knots in Brittenham and Hermiller's counterexample
construction \cite[Thm.~1.3]{BH1705}. Together with the calculation of
$u(\kn{12n}{491})$ in Section~\ref{sec:bj}, they identify \kn{13n}{3370} as a
counterexample. All four knots in that construction now have unknotting number two.

\raggedbottom
\subsection{Knots whose branched cover is an \texorpdfstring{$L$}{L}-space}\label{sec:sweep}

The correction-term obstructions can also be applied when the cover has no
sharp Goeritz or plumbing description available. If the reduced Khovanov
homology over $\mathbb F_2$ has rank $\det K$, then \eqref{eq:kh-lspace} proves
that $\dbc(K)$ is an $L$-space. Greene's model gives finite sets containing its
correction terms, as explained in Section~\ref{sec:greene-background}. After
comparing the markings, we apply the surgery tests to every remaining vector.
An improved lower bound is recorded only if all these vectors are excluded.

We use the ranges obtained after the preceding lower-bound computations and the
two Greene calculations of Section~\ref{sec:bj-counterexample}; these input
ranges are frozen as the input to this calculation. They precede the
improvements reported in this subsection and can be reproduced from the
deposited sources. A knot enters
the list when it is non-alternating, satisfies the reduced mod-$2$ Khovanov rank
equality, and has cyclic $H_1(\dbc(K))$. The rank equality is a sufficient
condition for the $L$-space property. Cyclicity is a restriction of the
present implementation, which labels the classes by $\bbz/\det K$.
These two requirements select $4{,}264$ of the $6{,}236$ non-alternating prime
knots with at most $13$ crossings in the archived data; $31.6\%$ lie outside
this calculation on these grounds alone.

Within this group, the range and signature select the test. For a range $[1,b]$ with
$|\sigma|\le2$, we use the half-integral surgery pattern \eqref{eq:niwu} to
test $u=1$. For $[n,b]$ with $|\sigma|=2n$ and $n\in\{2,3,4\}$, we use the
rank-$n$ comparison. After mirroring if necessary, $\sigma=2n$ and any
unknotting sequence of length $n$ would consist entirely of negative crossing
changes. Mirroring reverses the orientation of the cover and negates its
correction terms. Theorem~\ref{thm:montesinos} and \eqref{eq:qtilde-n}
therefore apply. For this rank-$n$ comparison, we report an improved lower
bound only when the spin entry is determined to be $-n/2$, as required by
\eqref{eq:lspace-spin}, and every remaining candidate vector is incompatible
with every admissible form. Under these conditions the lower bound rises
from $n$ to $n+1$. For the test of $u=1$, we likewise require every candidate
vector, in both orientations and under every affine identification, to fail
the half-integral surgery condition.

We apply these tests to \nSweepTargets{} knots whose unknotting number is
undetermined in these input ranges. We also test knots with known values for comparison;
the completed calculations are consistent with those known values. The
selection and execution records are deposited with the code.

The sum constraint \eqref{eq:casson-sum} determines the correction terms
needed for two further applications. For $K=\kn{13n}{1619}$, the reduced
mod-$2$ Khovanov rank and determinant are both $33$, while $\sigma(K)=0$.
Here $V_K(-1)=33$ and $V'_K(-1)=-80$, so the correction terms sum to
$-40/3$. The sum of the least candidates obtained from Greene's model has
the same value. Each correction term is therefore fixed, and all
$2\varphi(33)\cdot33=1{,}320$ affine comparisons, including both
orientations, fail the surgery condition \eqref{eq:niwu}. An explicit
sequence of two crossing changes supplies the upper bound, giving
$u(\kn{13n}{1619})=2$.

The knots \kn{13n}{4876}, \kn{13n}{4973}, and \kn{13n}{5102} have determinants
$33$, $105$, and $117$, respectively, and signature $-8$. In each case the
reduced mod-$2$ Khovanov rank equals the determinant. Greene's model
determines the correction terms of the latter two covers directly; for
\kn{13n}{4876}, the spin value and \eqref{eq:casson-sum} select a unique
vector from the remaining candidates. After mirroring to obtain signature
$8$, we apply Owens's obstruction to four crossing changes
\cite[Thm.~5]{Owens}. Every admissible half-integral surgery form and group
identification is excluded by a characteristic covector as described after
Theorem~\ref{thm:dmatch}. Together with verified sequences of five crossing
changes, this proves
\[
    u(\kn{13n}{4876})=u(\kn{13n}{4973})=u(\kn{13n}{5102})=5.
\]
The correction-term vectors, surgery forms, and covectors are supplied with
the code. As a check on the normalization in \eqref{eq:casson-sum}, the
identity agrees with all $3{,}204$ complete correction-term vectors in the
comparison data.

The obstruction improves the lower bound for \nSweepObstructed{} knots:
\nSweepUone{} by the
surgery pattern, \nSweepRankTwo{} in rank two, \nSweepRankThree{} in rank three
and \nSweepRankFour{} in rank four. These bounds determine
\nSweepNewExact{} additional exact values and narrow \nSweepImproved{} ranges
without making them exact. Including the earlier values for \kn{12n}{491} and
\kn{13n}{3370}, Greene's model accounts for \nSweepExact{} exact values in
Table~\ref{tab:lower-summary}. The entries are marked $\mathrm{G}$ in the
appendices.

\subsection{Summary and verification}\label{sec:mccoy-results}

McCoy's theorem permits a broader verification of the alternating cases
(Theorem~\ref{thm:mccoy}). In the April 2026 snapshot,
\nMcCoyVerified{} alternating knots have lower bound one. We examine every
crossing in a reduced alternating diagram of each knot. All
\nMcCoyUnknotting{} knots recorded with $u=1$ have an unknotting crossing,
and none of the remaining \nMcCoyExcluded{} knots does. The latter set is
exactly the set whose lower bounds become two in the June snapshot. Thus the
complete scan agrees with the June classification in all
\nMcCoyVerified{} cases.

The \nMcCoyKnots{} entries marked $\mathrm{k}$ in
Appendix~\ref{app:improve} are a subset of these \nMcCoyExcluded{} knots.
Their lower bounds were restored to one when constructing our comparison
table, although they were already two in June. For each of these
\nMcCoyKnots{} diagrams, every crossing change has determinant different
from one, which alone excludes the unknot. They contribute to the count
relative to the reconstructed table, but are not new lower bounds relative
to either the June snapshot or the snapshot of 9 September.

Together, the lower bounds determine \nExactLower{} exact values and improve
\nLowerImproved{} further ranges before applying the constructions of
Section~\ref{sec:upper}.
Table~\ref{tab:lower-summary} assigns each change from the comparison table to one method; where several
obstructions apply, the knot is counted only once. The distinction between the
correction-term and generator bounds is retained in the appendix labels.

\begin{table}[htbp]
\centering
\small
\caption{Changes from the reconstructed comparison table supplied by lower
bounds, before applying the constructions of Section~\ref{sec:upper}.
Exact values occur when the computed lower bound reaches an
independently known upper bound. The McCoy row includes \nMcCoyKnots{} lower
bounds already recorded in the June snapshot and that of 9 September.}\label{tab:lower-summary}
\begin{tabular}{@{}p{9.3cm}rr@{}}
\toprule
Method & Exact values & Other ranges\\
\midrule
Linking pairing & 720 & 91 \\
Homological torsion & 2 & 1 \\
Correction terms, Traczyk's criterion, and homology & 226 & 36 \\
Obstruction to three crossing changes & 129 & 12 \\
Obstruction to four crossing changes (completed cases) & 21 & 0 \\
Higher cyclic covers & 75 & 3 \\
Branched covers of Montesinos knots & 44 & 5 \\
Greene's model for branched covers & 1498 & 215 \\
McCoy's criterion for alternating diagrams & 0 & 27 \\
\midrule
Total & \nExactLower{} & \nLowerImproved{}\\
\bottomrule
\end{tabular}
\end{table}

The constructions of Section~\ref{sec:upper} make ten of the
\nLowerImproved{} improved ranges exact, leaving \nImproved{} undetermined
ranges. Combining the lower bounds and these constructions gives the
following result.

\begin{samepage}
\begin{theorem}\label{thm:main}
For prime knots with at most $13$ crossings, our computations give the following
changes from the comparison table reconstructed from \KI{} in
Section~\ref{sec:lower}.
\begin{enumerate}
\item[(i)] We determine \nExact{} unknotting numbers: \nUtwo{} knots have $u=2$,
\nUthree{} have $u=3$, \nUfour{} have $u=4$, and \nUfive{} have $u=5$.
\item[(ii)] Lower bounds determine \nExactLower{} of these exact values by reaching
an independently known upper bound.\footnote{For \kn{13n}{3370} and
\kn{13n}{1587}, the upper bound two is supplied by \cite[Thm.~1.3(a) and
\S3]{BH1705}, rather than by the comparison ranges $[1,3]$.} Explicit
crossing-change constructions determine the remaining \nExactUpper{} values,
ten of which also use our improved lower bounds.
\item[(iii)] A further \nImproved{} ranges improve without becoming exact, all of
them by lower bounds.
\end{enumerate}
\end{theorem}
\end{samepage}

The exact values and improved ranges are listed in Appendices~\ref{app:u5}--\ref{app:improve}.

The calculations recover Owens's examples. We also compare independent presentations
of the linking pairings and correction terms. For Montesinos covers, we compare
the correction terms computed from star-shaped plumbings with lens-space
formulas and with Goeritz calculations for alternating examples, obtaining the
same values. The rank-three obstruction is consistent with the $346$ knots
recorded with unknotting number three in the August snapshot.

We also compared the correction-term conditions for $2{,}776$ knots with cyclic
branched-cover homology in the Greene and Montesinos data. Adding monotonicity
and boundedness to the positive even symmetric matching conditions gave no
further obstruction to unknotting number one. Replacing symmetry by monotonicity
likewise gave the same exclusions, in agreement with the observations of Owens
and Strle \cite[Thm.~3 and the subsequent discussion]{OwensStrle}.

\flushbottom
\section{Upper bounds: methods and results}\label{sec:upper}

If $j$ crossing changes transform $K$ into a knot $J$,
then an unknotting sequence for $J$ can be appended to provide
\begin{equation}\label{eq:upper}
    u(K)\le j+u(J).
\end{equation}
In particular, if changing $j$ marked crossings of a diagram gives the unknot,
then we obtain $u(K)\le j$. Reversing a single crossing change also shows that
$|u(K)-u(J)|\le1$ for crossing neighbors. The difficulty is to find useful diagrams and
crossings, while keeping track of the knot represented at each stage.

\subsection{Finding unknotting diagrams}

Minimal diagrams give a natural starting point, but the examples of Nakanishi and Bleiler
show that they need not contain a shortest unknotting sequence \cite{Nakanishi,Bleiler}.
Even when a minimal diagram realizes $u(K)$, the particular minimal diagram in a table
may fail to do so. An exploration can therefore benefit both from moving among minimal diagrams
and from increasing the crossing number.

The counterexamples of Brittenham and Hermiller make this point concrete. They find a $20$-crossing diagram of \kn{13n}{3370} with a
crossing change to \kn{11n}{21}, whose unknotting number is equal to one \cite{BH1705}.
Their work on connected sums likewise gives explicit sequences passing through other knot
types \cite{BHadd,BH26}. For example, an explicit sequence of three crossing changes establishes
$u(4_1\mathbin{\#}9_{10})\le3$ \cite[Thm.~1.1]{BH26}. Such constructions allow
Reidemeister moves between crossing changes; they need not be visible as marked crossings
of the diagram with which the search began.

Machine learning has been used to choose useful moves in these large spaces of knot diagrams.
Gukov, Halverson, Ruehle, and Su{\l}kowski use reinforcement learning to simplify braid
representatives by braid relations and Markov moves \cite{GukovUnknot}. Applebaum et al.
use reinforcement learning in the search for unknotting sequences and hard unknot
diagrams \cite{Applebaum}. The RL unknotter of Dranowski, Kabkov, and Tubbenhauer
explores diagram simplification through Reidemeister moves and applies the resulting
procedure to upper-bound searches \cite{DKT}. These methods address the choice of
presentations and moves, which is often the main practical obstacle to finding a sequence.

A complementary approach concerns the simplification of an unknot diagram once it has
been found. A pass move reroutes an overpass or an underpass through the complementary
regions of a diagram, preserving knot type. The ReAPR algorithm combines crossing-reducing
pass moves with changes of spatial embedding and planar projection. It simplifies the
published hard unknot diagrams tested by Cantarella, Schumacher, and Shonkwiler
\cite{ReAPR}. This result motivates the application of pass moves to the simplification of knot diagrams whenever examining the diagrams obtained after some crossing changes.

\subsection{The constructions}\label{sec:construction}

Starting from a tabulated diagram, we generate equivalent diagrams by random
sequences of Reidemeister moves and flypes, represented on the
embedded signed Tait graphs of Section~\ref{sec:background}. The randomized
exploration allows up to $50$ crossings and has a default limit of $10{,}000$
additional distinct diagrams per knot. In each diagram, we first screen
crossing-change subsets by the determinant. If more subsets survive than the
reduction budget permits, we select a random sample and simplify the changed
diagrams by Reidemeister moves and pass moves. If the result is the unknot,
it gives a direct upper bound. If it is an identified knot with a known upper
bound, \eqref{eq:upper} applies. This finite exploration does not exhaust all diagrams
or crossing-change subsets; failure to find an unknotting sequence gives no
lower bound.

For each successful construction the accompanying data specify the initial and marked
diagrams, the selected crossings, and the recorded moves. The presentation diagrams below
are checked separately from the exploration that found them. The starting knot is compared
with its tabulated representative by knot-complement isometries preserving
the meridian in SnapPy \cite{SnapPy}. For a claimed unknot, we compute knot
Floer homology, which detects Seifert genus and hence identifies the unknot
when the genus is zero \cite{OSgenus}. When a construction instead passes
through a named knot, its type is checked by a meridian-preserving isometry
and its known upper bound is used in \eqref{eq:upper}.
Identification up to mirror image suffices for the unknotting number. The
deposited diagrams and crossing data allow these identifications to be
repeated independently of the search.

No new unknotting diagram with one marked crossing was found in the examined $[1,2]$ class, and the
examined $[3,4]$ class yielded no improved upper bound. The new exact values on the upper-bound side come
from the following \nExactUpper{} constructions giving upper bound two.

\subsection{Explicit upper bounds}\label{sec:eight}

Figures~\ref{fig:marked1}--\ref{fig:marked5} give crossing-change constructions
establishing $u\le2$ for nineteen knots. Eighteen diagrams have two circled
crossings whose changes give the unknot. In the diagram of \kn{13n}{447},
changing the single circled crossing gives $10_{91}$ up to mirror image.
Since $u(10_{91})=1$ \cite{KnotInfo}, this also proves $u(\kn{13n}{447})\le2$.
The fifteen diagrams with $13$ crossings are minimal; three others have
$14$ crossings, and the diagram of \kn{13n}{4237} has $18$ crossings.

For the eight knots in Figures~\ref{fig:marked1}--\ref{fig:marked2} and for
\kn{13n}{3733}, the lower bound two is already recorded in the comparison
ranges. For the other ten knots, it follows from the Greene computations
of Section~\ref{sec:sweep}. The lower and upper bounds together determine
$u=2$ for all nineteen knots. These ten joint applications account for the
decrease from \nLowerImproved{} ranges after the lower-bound calculations
to \nImproved{} in the final table.

Thirteen of these upper bounds also appear in the recently updated work of Dranowski,
Kabkov, and Tubbenhauer \cite[Table~1]{DKT}. For \kn{13n}{221},
\kn{13n}{636}, \kn{13n}{1439}, \kn{13n}{2251}, \kn{13n}{2639}, and
\kn{13n}{4025}, they give the range $[1,2]$; our lower bounds determine
$u=2$ for these six knots.

\begin{figure}[htbp]
\centering
\includegraphics[width=0.24\textwidth]{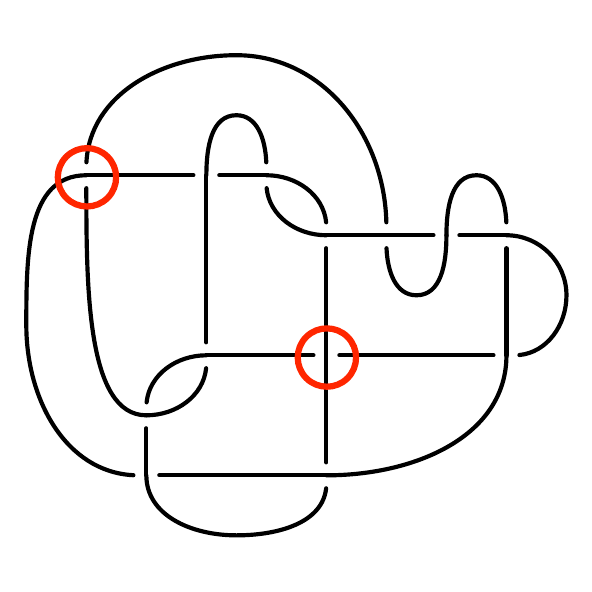}\hfill
\includegraphics[width=0.24\textwidth]{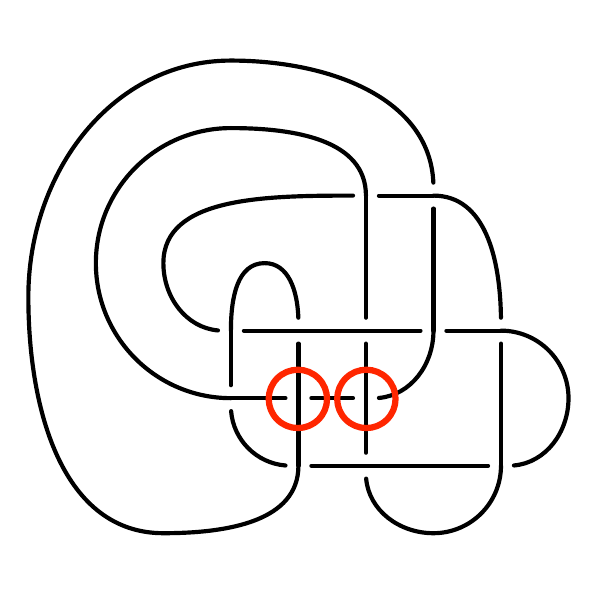}\hfill
\includegraphics[width=0.24\textwidth]{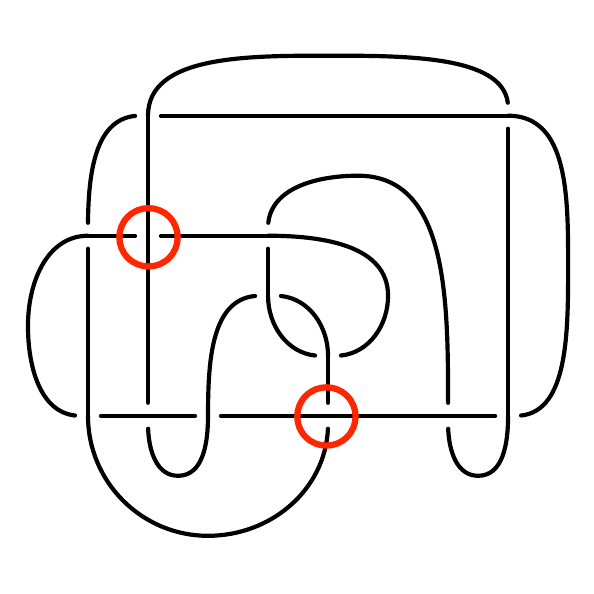}\hfill
\includegraphics[width=0.24\textwidth]{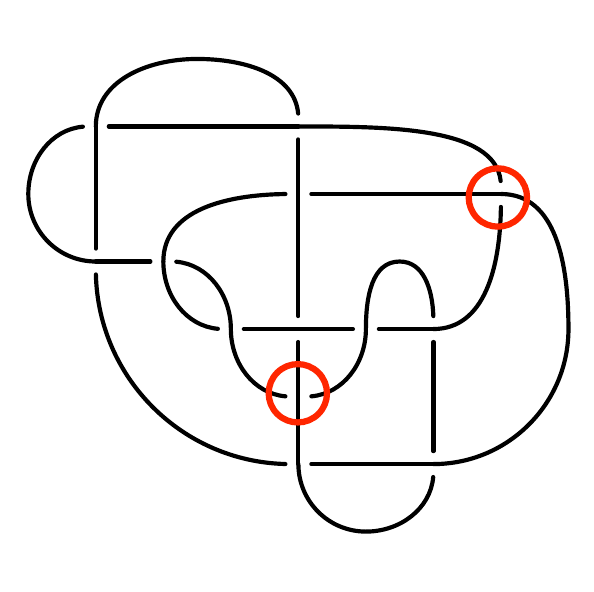}
\caption{Minimal marked diagrams of \kn{13n}{2379}, \kn{13n}{2809},
\kn{13n}{2907}, and \kn{13n}{3033}, from left to right. Changing the two circled
crossings in each diagram gives the unknot.}\label{fig:marked1}
\end{figure}

\begin{figure}[htbp]
\centering
\includegraphics[width=0.24\textwidth]{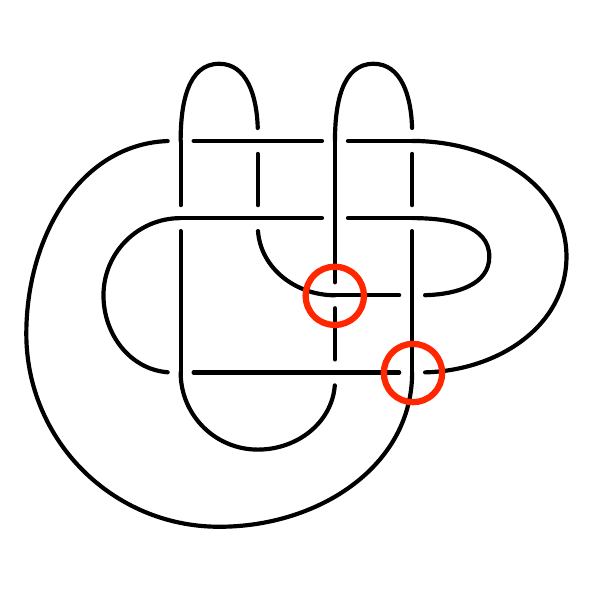}\hfill
\includegraphics[width=0.24\textwidth]{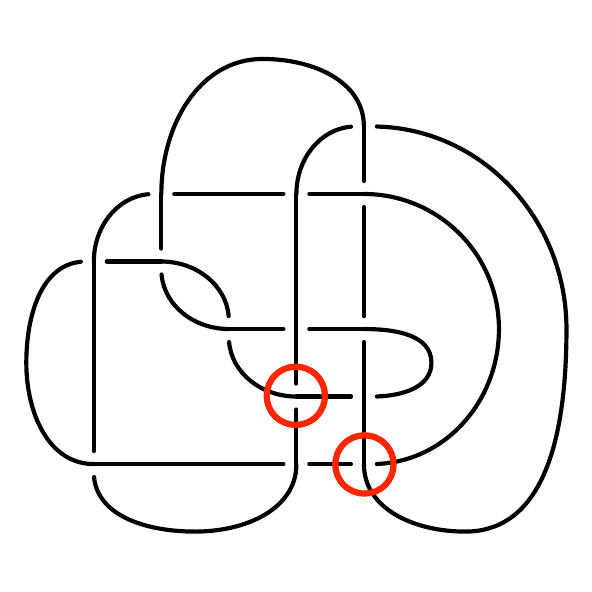}\hfill
\includegraphics[width=0.24\textwidth]{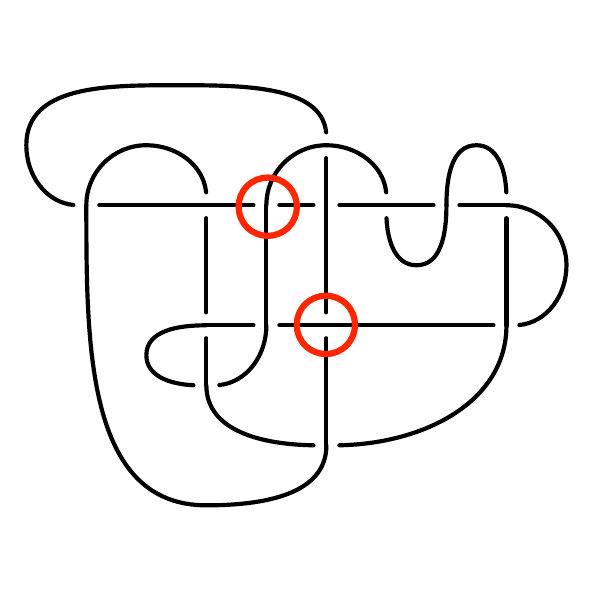}\hfill
\includegraphics[width=0.24\textwidth]{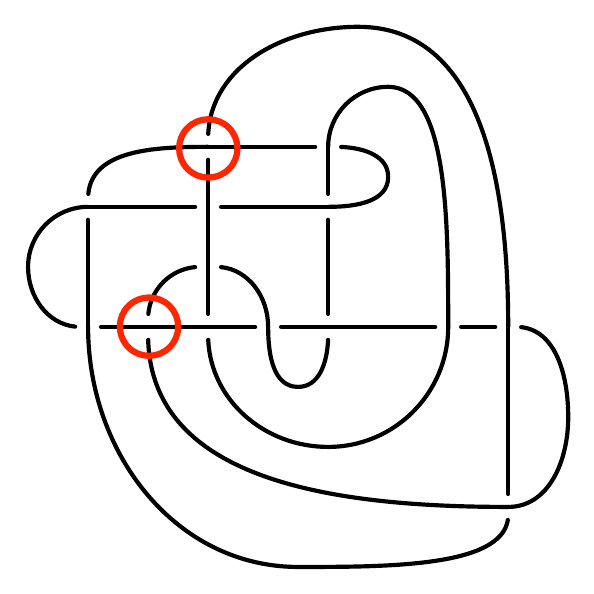}
\caption{Marked diagrams of \kn{13n}{3589}, \kn{13n}{30}, \kn{13n}{45}, and
\kn{13n}{80}, from left to right. The first has $13$ crossings and the others have $14$.
Changing the two circled crossings gives the unknot.}\label{fig:marked2}
\end{figure}

\begin{figure}[htbp]
\centering
\includegraphics[width=0.24\textwidth]{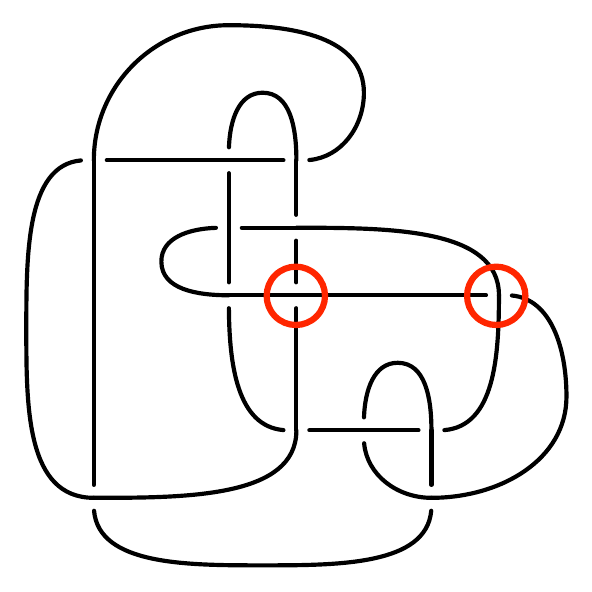}\hfill
\includegraphics[width=0.24\textwidth]{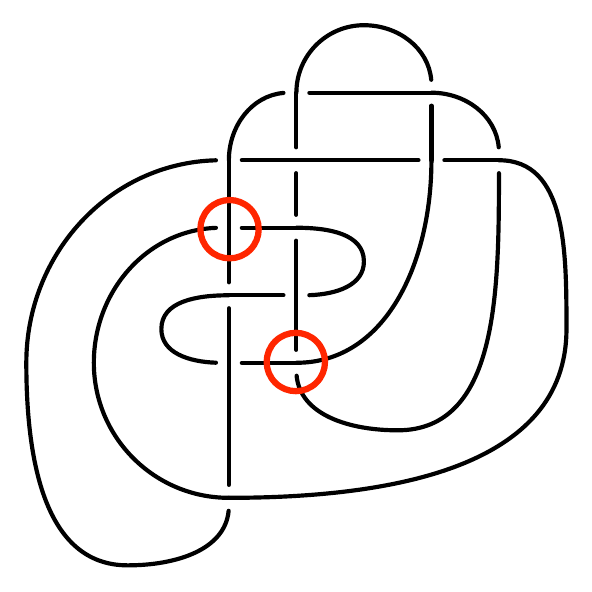}\hfill
\includegraphics[width=0.24\textwidth]{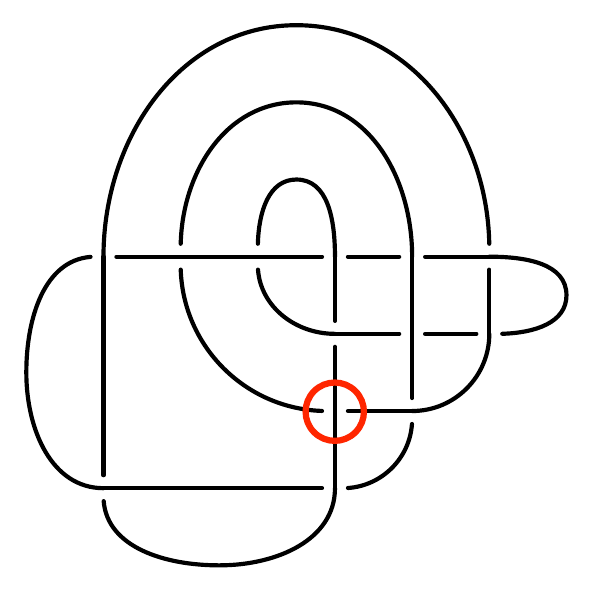}\hfill
\includegraphics[width=0.24\textwidth]{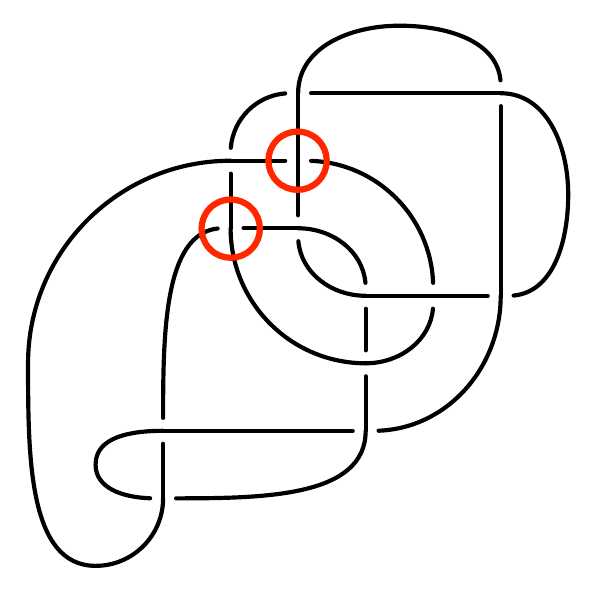}
\caption{Minimal marked diagrams of \kn{13n}{221}, \kn{13n}{436},
\kn{13n}{447}, and \kn{13n}{636}, from left to right. Changing the circled
crossing in \kn{13n}{447} gives $10_{91}$, whose unknotting number is one.
In each of the other diagrams, changing both circled crossings gives the unknot.}\label{fig:marked3}
\end{figure}

\begin{figure}[htbp]
\centering
\includegraphics[width=0.24\textwidth]{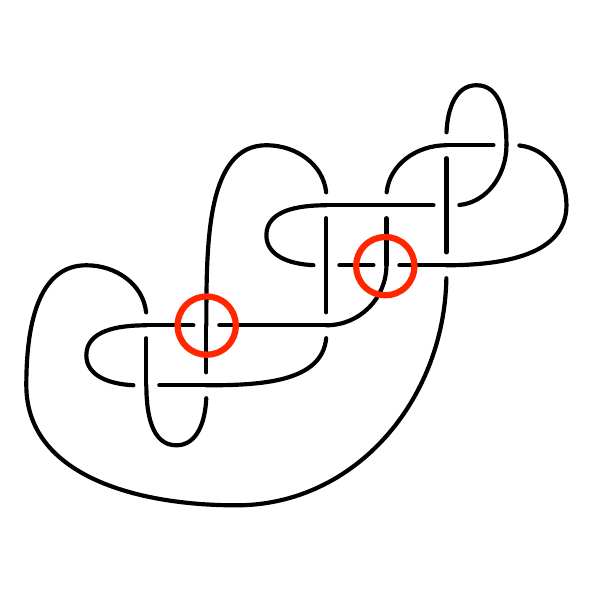}\hfill
\includegraphics[width=0.24\textwidth]{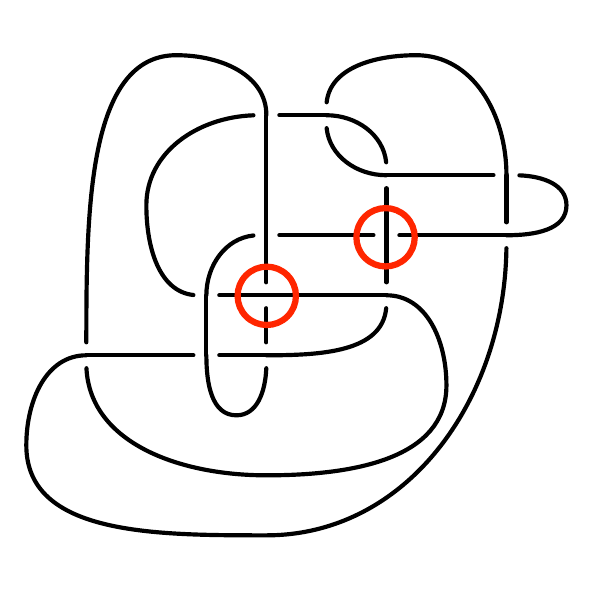}\hfill
\includegraphics[width=0.24\textwidth]{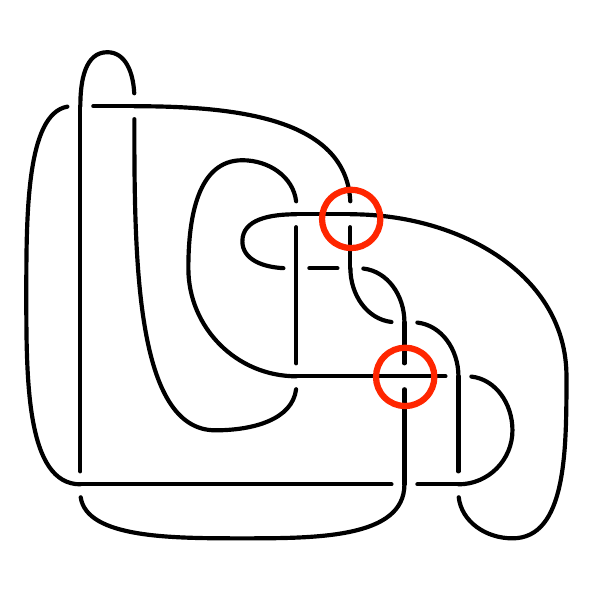}\hfill
\includegraphics[width=0.24\textwidth]{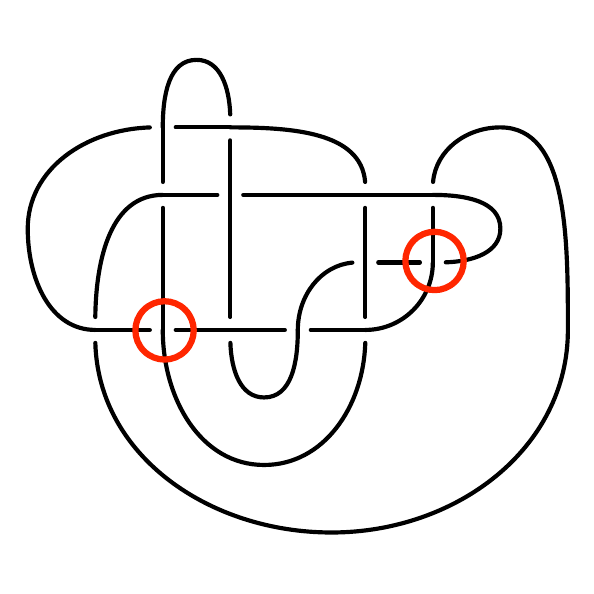}
\caption{Minimal marked diagrams of \kn{13n}{1439}, \kn{13n}{2251},
\kn{13n}{2639}, and \kn{13n}{3108}, from left to right. Changing the two
circled crossings in each diagram gives the unknot.}\label{fig:marked4}
\end{figure}

\begin{figure}[htbp]
\centering
\includegraphics[width=0.24\textwidth]{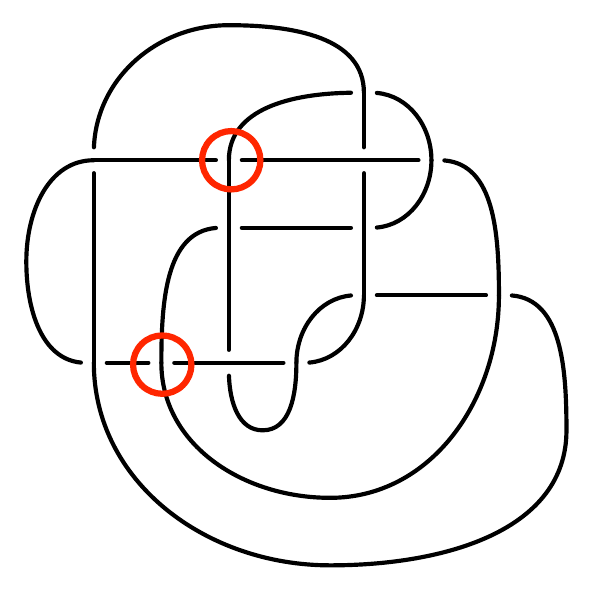}\hspace{0.04\textwidth}
\includegraphics[width=0.24\textwidth]{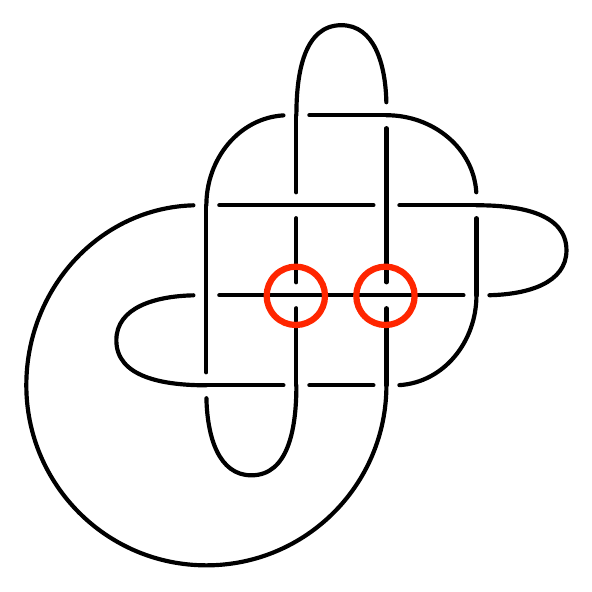}\hspace{0.04\textwidth}
\includegraphics[width=0.24\textwidth]{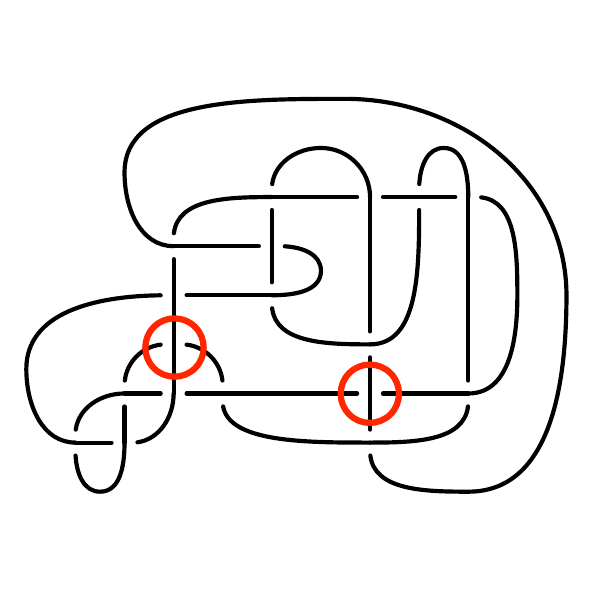}
\caption{Marked diagrams of \kn{13n}{3733}, \kn{13n}{4025}, and
\kn{13n}{4237}, from left to right. The first two are minimal diagrams
with $13$ crossings; the last has $18$ crossings. Changing the two circled
crossings in each diagram gives the unknot.}\label{fig:marked5}
\end{figure}

Appendix~\ref{app:diagrams} gives the planar diagram codes and marks the
crossings to be changed, with the resulting knot specified in each case.
These codes describe the constructions independently of the drawings.

The minimal diagrams also illustrate why the choice of diagram matters.
For \kn{13n}{2809}, no pair of changes in the reference diagram gives the
unknot, whereas the displayed minimal diagram contains such a pair. There
is a related distinction for \kn{13n}{2251} and \kn{13n}{4025}: all thirteen
crossing changes in each reference diagram give knots with unknotting
number two. In the displayed minimal diagrams, changing either member of
the marked pair gives a knot with unknotting number one. Thus even the set
of crossing-change neighbours can depend on the minimal diagram for a
non-alternating knot. The identified neighbours and their bounds are
recorded with the constructions.

For each of the ten knots in
Figures~\ref{fig:marked3}--\ref{fig:marked5} shown with $13$ crossings,
the construction gives a neighbour with $u=1$ in a minimal diagram.
Consequently $u^s_{\mathrm{BJ}}=u=2$ for these knots. In the case of
\kn{13n}{447}, that neighbour is $10_{91}$; the second crossing change is
made after replacing it by an unknotting diagram.

We also verify the upper bound $u(\kn{13n}{1587})\le2$ using a braid representative with
one crossing change to $10_{113}$, whose unknotting number is one. This crossing-change
construction was already given by Brittenham and Hermiller \cite[\S3]{BH1705}.
Our accompanying certificate verifies the diagrams and the intermediate knot.
This supplies the upper bound two independently of the Greene calculation in
Section~\ref{sec:sweep}; together the two bounds determine $u(\kn{13n}{1587})=2$.

\subsection{Crossing changes and Tait graphs}\label{sec:electrical}

The invariants used for lower bounds can also guide the choice of a crossing for an upper
bound. Let $D$ be a reduced alternating diagram, choose the checkerboard form $G$ to be
positive definite, and let $\Gamma$ be its Tait graph. All edge signs are then $+1$.
For an edge $e$, set
\[
    \Reff(e)=\frac{\tau_e(\Gamma)}{\tau(\Gamma)},
\]
where $\tau(\Gamma)$ is the number of spanning trees and $\tau_e(\Gamma)$ is the number
containing $e$. This ratio is the effective resistance between the endpoints of $e$ when
every edge has unit resistance \cite{Bollobas}. Since a reduced diagram gives neither
a loop nor a bridge at $e$, one has $0<\Reff(e)<1$.

For the checkerboard surface $F$ whose Gordon--Litherland form is $G$, let type I and type II denote the two local orientation
types of a crossing in the Gordon--Litherland convention \cite{GL}. Their signature
formula is
\begin{equation}\label{eq:gl}
    \sigma(K)=\operatorname{sign}(G)-\mu(F),\qquad
    \mu(F)=\sum_{c\text{ of type II}}\eta(c).
\end{equation}
Thus type II crossings contribute to the correction term, while type I crossings do not.

\begin{proposition}\label{prop:sigma}
Let $c$ be a crossing of $D$, let $e$ be its Tait edge, and let $K_c$ be the knot after
the crossing change. Then
\begin{equation}\label{eq:det-change}
    \det K_c=\det K\,|1-2\Reff(e)|.
\end{equation}
For a type I crossing, the signature decreases by two if $\Reff(e)>1/2$ and is unchanged
if $\Reff(e)<1/2$. For a type II crossing, it increases by two if $\Reff(e)<1/2$ and
is unchanged if $\Reff(e)>1/2$. The value $\Reff(e)=1/2$ cannot occur.
\end{proposition}

\begin{proof}
If $x$ is the difference of the coordinate vectors of the two regions incident to $e$,
with the deleted coordinate set to zero, the changed Goeritz matrix is
$G'=G-2xx^T$. The matrix determinant lemma gives
\[
    \det G'=\det G(1-2x^TG^{-1}x),\qquad x^TG^{-1}x=\Reff(e).
\]
This proves \eqref{eq:det-change}. The determinant of a knot is odd, so $G'$ is
nonsingular and $\Reff(e)\ne1/2$. As $G$ is positive definite, the rank-one change either
preserves its signature or decreases it by two, according to the sign of
$1-2\Reff(e)$. The type of the crossing is unchanged, while its Goeritz sign reverses.
Consequently $\mu$ decreases by two at a type II crossing and is unchanged at a type I
crossing. Substitution in \eqref{eq:gl} gives the signature statements.
\end{proof}

The same update yields the linking pairing of the changed knot. In the convention in
which it is represented by $G'^{-1}$, the class of $x$ has self-linking
\begin{equation}\label{eq:link-change}
    x^TG'^{-1}x=\frac{\Reff(e)}{1-2\Reff(e)}\pmod{\bbz},
\end{equation}
by the Sherman--Morrison formula. When $[x]$ generates the discriminant group,
\eqref{eq:lickorish} can therefore be tested directly. In general, one computes the full
linking pairing and applies the cyclicity and self-linking tests. This combines the signature and linking-pairing tests
without first identifying the changed knot.

These formulas explain why a crossing that lowers the absolute signature need not lower
the unknotting number. If $u(K)=|\sigma(K)|/2$, such a signature drop is necessary, but
the changed knot may still fail the conditions on the linking pairing or correction terms for the
desired smaller unknotting number.

Proposition~\ref{prop:sigma} also gives a diagrammatic way to evaluate the
signature condition used in Section~\ref{sec:minimal}: a crossing can lead
to a knot with $u=1$ only if the resulting signature has absolute value at
most two. In the implementation, we compute this signature directly from
the changed diagram using the Gordon--Litherland formula. Across the
$69{,}632$ recorded crossing changes, the stored signatures and determinants
agree with Proposition~\ref{prop:sigma}, with no exceptions. The stored
self-linking values also agree with \eqref{eq:link-change}.

\section{Bernhard--Jablan Conjecture: Revisited}\label{sec:bj}

\subsection{An Explicit Counterexample}\label{sec:bj-counterexample}

The examples \kn{13n}{2251} and \kn{13n}{4025} in
Section~\ref{sec:eight} show why an obstruction for the neighbours of a
single minimal diagram need not decide the strong equality. Brittenham
and Hermiller's construction accounts for all minimal diagrams of
\kn{13n}{3370}. They showed that
$u(\kn{13n}{3370})\le2$ and examined $24$ minimal diagrams representing all possibilities up to flypes \cite[\S2.4]{BH1705}.
Among the knots obtained by changing one crossing in these diagrams, only
\kn{12n}{288}, \kn{12n}{491}, and \kn{12n}{501} could have unknotting number one.
None of these three knots has a minimal diagram with an unknotting crossing
\cite[Thm.~1.3 and \S2.4]{BH1705}.

Our Montesinos computations in Section~\ref{sec:lower} give
$u(\kn{12n}{288})=u(\kn{12n}{501})=2$, in agreement with the calculations of
Dranowski et al.\ \cite{DGKTdata}. The remaining question in this construction is
whether $u(\kn{12n}{491})=1$ or $2$. We resolve it by computing correction terms
from a knot diagram, using the theory of Section~\ref{sec:greene-background}.
The same method also determines $u(\kn{13n}{3370})$.

The archived reduced Khovanov homology over $\mathbb F_2$ has dimension $69$ for
\kn{12n}{491} and $33$ for \kn{13n}{3370} \cite{KnotInfo}. These dimensions equal
the respective determinants. Equation~\eqref{eq:kh-lspace} therefore shows that
both double branched covers are $L$-spaces. We then apply Greene's model to
compute the correction-term data needed for the surgery obstruction.
For \kn{12n}{491}, we also obtain a quasi-alternating resolution tree, beginning
with $69=48+21$ and ending in unknots and non-split alternating links. The
archived tree gives a separate certificate of the $L$-space property.

For each marking and checkerboard colouring, we enumerate the Kauffman states
and their solitary-state gradings. Odd determinant identifies a boundary
$\mathrm{Spin}^c$ structure with its first Chern class, with the spin structure
at zero. The grading candidates from different markings can therefore be compared
using isomorphisms of the cyclic discriminant groups. We consider every isomorphism
preserving the linking pairing and retain those compatible with the candidate
gradings in all classes. If several remain, we take the union of their candidate
sets before intersecting with the information from other markings. We also impose
the conjugation symmetry $d(Y,\mathfrak s)=d(Y,\overline{\mathfrak s})$.
Every marking is included, and we repeat the intersections until the candidate
sets stop changing. In the deposited computations this takes two passes for
each knot; no marking is discarded because its identification is ambiguous.
The actual identification is always among the retained possibilities, so these
operations preserve the actual correction term in every class. This makes precise
the use of different markings to constrain correction terms
\cite[\S7.1]{Greene}.

For \kn{12n}{491}, the $48$ choices of marked edge and colouring determine all
$69$ correction terms. For \kn{13n}{3370}, the $52$ choices determine $29$ of the
$33$ values. The remaining four classes form two conjugate pairs, each with three
possible gradings. Thus there are nine candidate vectors, one of which is the
actual correction-term vector. Appendix~\ref{app:bj-correction-terms} gives the
labelled values and candidate sets, so the obstruction can be checked without
reconstructing the state enumeration.

If either knot had unknotting number one, the Montesinos trick would identify
one orientation of its cover with $D/2$ surgery on a knot in $S^3$, where
$D=69$ or $33$. Under some affine identification $i\mapsto ai+b$ of the cyclic
labels, the correction terms would therefore satisfy the Ni--Wu formula
\eqref{eq:niwu}. We test every unit $a$ modulo $D$, every translation $b$, and
both orientations. This includes all identifications allowed by the linking
pairing, without imposing that restriction to reduce the search. There are
$2\varphi(D)D$ affine comparisons per candidate vector, where $\varphi$ is
Euler's totient function. Thus the $D=69$ calculation makes
$2\cdot44\cdot69=6{,}072$ comparisons. For $D=33$, there are
$2\cdot20\cdot33=1{,}320$ per vector, or $11{,}880$ over the nine vectors.
In every case, either the differences from the lens-space correction
terms fail to be nonnegative even integers, or the labels corresponding to the
same $V_j$ give inconsistent values. For \kn{13n}{3370}, all nine candidate
vectors fail. Hence neither knot has unknotting number one.

Brittenham and Hermiller's upper bounds then give
$u(\kn{12n}{491})=u(\kn{13n}{3370})=2$
\cite[Thm.~1.3(a)--(b)]{BH1705}. Their minimal-diagram calculation now shows
that every crossing change in a minimal diagram of \kn{13n}{3370} gives a knot
of unknotting number at least two. Recall that the \emph{strong Bernhard--Jablan
unknotting number} $u^s_{\mathrm{BJ}}(K)$ is one plus the least unknotting number
among all such crossing neighbours \cite{BH1705}. Since \kn{12n}{491} is one of
these neighbours, that least value is two. We obtain the following result.

\begin{theorem}\label{thm:bj-counterexample}
The four knots \kn{12n}{288}, \kn{12n}{491}, \kn{12n}{501}, and
\kn{13n}{3370} all have unknotting number two. Moreover,
\[
    u(\kn{13n}{3370})=2<3=u^s_{\mathrm{BJ}}(\kn{13n}{3370}),
\]
so \kn{13n}{3370} fails the strong Bernhard--Jablan equality and is a
counterexample to the original conjecture.
\end{theorem}

The theorem combines Brittenham and Hermiller's enumeration
\cite[Thm.~1.3(c)]{BH1705} with the three twelve-crossing unknotting numbers
and the matching lower bound for \kn{13n}{3370} obtained here.

The state gradings agree with the Ozsv\'ath--Szab\'o formula
on alternating examples and with independent plumbing calculations for $8_{20}$,
$9_{43}$, and $9_{44}$. Repeating the calculation with another minimal diagram of
\kn{12n}{491} gives the same correction terms up to relabelling.
The obstruction admits the known examples
$3_1$, $4_1$, $5_2$, $8_{20}$, and $9_{44}$ of unknotting number one and excludes
that value for $5_1$, $7_4$, and $9_{43}$. All Goeritz signatures and grading
calculations for the two target knots use exact arithmetic.

We next consider $K=\kn{13n}{1587}$, the second example in \cite[\S3]{BH1705}.
Brittenham and Hermiller proved $u^w_{\mathrm{BJ}}(\kn{13n}{1587})=3$ and
constructed a crossing change from a nonminimal diagram to $10_{113}$, which
has unknotting number one. This gives the upper bound two. The deposited
certificate verifies their construction as described in Section~\ref{sec:eight}.

The calculation of Section~\ref{sec:sweep} gives the matching lower bound.
The reduced Khovanov homology over $\mathbb F_2$ has rank $75=\det K$, so the
double branched cover is an $L$-space. Comparing all $52$ choices of marked
edge and colouring determines its $75$ correction terms. The Ni--Wu test
considers both orientations, all $\varphi(75)=40$ units $a$ modulo $75$, and
all $75$ translations $b$ in the affine relabelling $i\mapsto ai+b$.
There are therefore $2\cdot40\cdot75=6{,}000$ comparisons, and every one
fails \eqref{eq:niwu}. Thus $u(K)\ne1$, and combining this obstruction
with the published upper bound gives $u(\kn{13n}{1587})=2$.

This determines the unknotting number of Brittenham and Hermiller's second
example, for which the weak equality was already known to fail
\cite[\S3]{BH1705}. Its status under the strong equality depends on whether
a crossing change in some minimal diagram gives a knot of unknotting number
one; this remains undetermined.

\subsection{Alternating Examples}\label{sec:minimal}

The alternating example \kn{14a}{2539} of Brittenham and Hermiller lies
outside our table of knots with at most $13$ crossings and has
$u\le3<4=u^w_{\mathrm{BJ}}$ \cite[\S3]{BH1705}. By Lemma~\ref{lem:flype},
its minimal-diagram crossing neighbours can be found from a single reduced
alternating diagram. The deposited diagram identifications give precisely
\[
    10_{11},\quad \kn{12a}{178},\quad \kn{12a}{183},\quad
    \kn{12a}{196},\quad \kn{12a}{684},\quad 8_6\mathbin{\#}4_1.
\]
All have unknotting number at least two: the five prime knots belong to
Appendix~\ref{app:dichotomy}, and Scharlemann's theorem excludes unknotting
number one for the composite knot \cite{Scharlemann}. In particular, this
reduced alternating diagram has no unknotting crossing. If the knot had
unknotting number one, McCoy's theorem would require such a crossing, so
$u(\kn{14a}{2539})\ge2$ \cite{McCoy}. The stronger bounds on its six
neighbours also give $u^s_{\mathrm{BJ}}(\kn{14a}{2539})\ge3$.
If $u(\kn{14a}{2539})=2$, this knot is therefore an alternating counterexample.
The present calculations leave
its value in $\{2,3\}$, and the five prime neighbours connect this example to
the larger collection considered below.

We next consider an archived collection of \nDataKnots{} alternating knots with at most $13$ crossings whose
unknotting number is known in our table and lies between one and four. For each knot,
we examine every crossing of the reference diagram. Each of these diagrams has at least one crossing whose change
lowers the unknotting number by one. The resulting knot at such a crossing is identified
by diagram comparison or knot-complement isometry. Hence the strong
Bernhard--Jablan equality holds for each knot in this collection.

After the lower-bound computations,
\nOpenTwoThreeAlt{} alternating knots retain range $[2,3]$. If one of them has $u=2$
and satisfies the Bernhard--Jablan conjecture, a crossing change in a minimal diagram must
give a knot with $u=1$. By Lemma~\ref{lem:flype}, the same possibility must occur in the
reference diagram. The changed knot must have $|\sigma|\le2$, must pass the known
obstructions to $u=1$, and must be prime by Scharlemann's theorem \cite{Scharlemann}.
The signature condition is the one discussed after
Proposition~\ref{prop:sigma}; it is evaluated directly on the changed diagram.
We call a crossing not excluded by these conditions a \emph{candidate crossing}.

For \nDichotomy{} knots, every crossing is excluded by the recorded invariant tests and
verified knot identifications. By Lemma~\ref{lem:flype}, no crossing change in
any minimal diagram of these knots gives a knot of unknotting number one.
Their names appear in Appendix~\ref{app:dichotomy}.
A further \nDichotomyHeuristic{} diagrams have no apparent candidate, but one or more exclusions rely on
an unverified composite-knot identification; these are excluded from that list.

For a knot $K$ in the list, the definition of $u^s_{\mathrm{BJ}}$ gives
$u^s_{\mathrm{BJ}}(K)\ge3$. Any such knot with $u(K)=2$ would therefore be
an alternating counterexample to the strong equality. Their unknotting
numbers remain in $[2,3]$, so they are not included in the \nExact{} exact
determinations.

The remaining \nWithCandidates{} knots have candidate crossings leading to
\nChildren{} distinct non-alternating knots with range $[1,2]$ or $[1,3]$.
For each crossing neighbor, it remains to determine whether its unknotting number
is one. A positive answer gives a matching upper bound for the parent knots; a
negative answer removes the corresponding candidates. The search for matching upper
bounds therefore focuses on these questions about unknotting number one.

\raggedbottom
\section{Conclusion}\label{sec:conclusion}

The methods applied here determine \nExact{} unknotting numbers and narrow
\nImproved{} further ranges relative to the archived comparison ranges
of Section~\ref{sec:lower}. Lower bounds supply \nExactLower{} exact values;
crossing-change constructions supply the other \nExactUpper{}, including
ten for which our lower and upper bounds are both needed. The results include
\nUfour{} values equal to four and \nUfive{} equal to five, with the methods
recorded in the appendices.

The signature restricts the signs of a possible sequence, while the homology
and linking pairing of the branched cover constrain its surgery description.
Correction terms restrict the definite filling. For \kn{12n}{457}, reduced
Floer homology gives a further obstruction even though the correction terms
admit the necessary surgery pattern.

Using Greene's model, we compute correction-term candidates for a larger
collection beyond alternating and Montesinos knots \cite{Greene}. Once the
$L$-space property is established,
it places each correction term among finitely many solitary-state gradings.
Comparing markings and imposing the Casson--Walker sum constraint can
determine the correction terms exactly, as for \kn{13n}{1619} and
\kn{13n}{4876}. The same obstructions can also succeed without a complete
calculation of the correction terms, provided every remaining candidate
fails the surgery conditions. Under the
cyclicity and signature hypotheses of Section~\ref{sec:sweep}, the completed
calculations give \nSweepNewExact{} additional exact values. This contribution
combines the existing diagrammatic model with the surgery obstructions
\cite{Owens,NiWu}, making them applicable to these further branched covers.

The nineteen upper-bound constructions use diagrams with $13$, $14$, or
$18$ crossings. Eighteen give the unknot directly by changing two marked
crossings, while the construction for \kn{13n}{447} passes through $10_{91}$.
For \kn{13n}{2251} and \kn{13n}{4025}, the displayed minimal diagrams admit a
crossing change to a knot with $u=1$, although every crossing change in the
reference diagram gives a knot with $u=2$. These examples make the role of
the chosen diagram explicit in the computation of upper bounds.

In recent work that appeared while this paper was in preparation, Gebel and
Prangley obtained \nGPExact{} values equal to three and \nGPImproved{} improved
lower bounds for alternating knots using Owens's obstruction and correction
terms \cite{GP}. Our independently obtained results agree with all of them.
These \nGPExact{} exact values also agree with the \KI{} snapshot retrieved
on 9 September 2026. The other \nLatestUnresolvedExact{} exact determinations
remain undetermined in that snapshot, and \nLatestRangeImprovements{} of our
remaining ranges are narrower than its recorded bounds.

The computations also determine $u(\kn{13n}{1587})=2$, below its weak
Bernhard--Jablan number three established by Brittenham and Hermiller
\cite[\S3]{BH1705}. This gives an exact value for their second example,
while its status under the strong equality remains undetermined.

In their original four-knot construction, we identify \kn{13n}{3370} as
an explicit counterexample to the Bernhard--Jablan conjecture. The knots
\kn{12n}{288}, \kn{12n}{491}, \kn{12n}{501}, and \kn{13n}{3370} all have
unknotting number two. Combining these values with Brittenham and Hermiller's
minimal-diagram relation \cite[Thm.~1.3(c)]{BH1705} gives
\[
    u(\kn{13n}{3370})=2<3=u^s_{\mathrm{BJ}}(\kn{13n}{3370}).
\]
Thus no crossing change in any minimal diagram of this knot reduces its
unknotting number, although two crossing changes suffice in a suitable diagram.

An alternating counterexample remains to be found. For the \nDichotomy{}
knots of Section~\ref{sec:minimal}, every crossing change in a minimal
diagram gives a knot of unknotting number at least two, so any member with
$u=2$ would provide one. Their unknotting numbers remain in $[2,3]$.
Future work will focus on determining these values by finding diagrams
that can be unknotted by two crossing changes or by proving lower bounds
equal to three.

\begin{samepage}
\subsection*{Data and code availability}
The code is available in the following repository:
\begin{center}
    \url{https://github.com/Pinocchio315/unknotting-tait-graphs}
\end{center}
The repository contains the source \KI{} snapshots, the reconstructed
comparison table and release comparisons, obstruction outputs, and data for
verifying the upper bounds. The aggregate counts and table entries are
generated from these deposited sources, including the Greene computations,
and are included directly in the manuscript source. 

The code was written with the assistance of Claude Fable 5.1 (Anthropic).
However, the author takes full responsibility for the code,
the mathematical content, and the results in this paper.
\end{samepage}

\subsection*{Acknowledgements}

S.-J. Lee is extremely grateful to Pavel Putrov for his extensive and insightful feedback, which has been invaluable in shaping this paper.
The author is grateful to Brendan Owens for helpful comments on the correction-term
obstructions and for drawing attention to their implications for the four-dimensional
clasp number and slicing number.
The author also thanks Charles Livingston and Allison H. Moore, the maintainers of KnotInfo, and Daniel Tubbenhauer for their helpful clarifications regarding the known values and bounds of the unknotting number for several knots.

He is also grateful to the Ulsan National Institute of Science and Technology (UNIST), the Shanghai Institute for Mathematics and Interdisciplinary Sciences (SIMIS), and the Beijing Institute of Mathematical Sciences and Applications (BIMSA) for their generous hospitality during various stages of this work. 

This work is supported by the Institute for Basic Science (IBS) under Project No. IBS-R003-D1.

\clearpage
\flushbottom
\appendix

\newsavebox{\appendixLegendBox}

\section{Knots with \texorpdfstring{$u=5$}{u=5}}\label{app:u5}
The following \nUfive{} knots have unknotting number $5$.  The entries are ordered down the
columns.

{\footnotesize
\setlength{\tabcolsep}{4pt}
\setlength{\LTleft}{\fill}
\setlength{\LTright}{\fill}
\sbox{\appendixLegendBox}{%
\begin{minipage}{\textwidth}
\hrule height\heavyrulewidth\vspace{3pt}
\raggedright\setlength{\parindent}{0pt}\setlength{\parskip}{0pt}
\strut For every knot listed, the correction-term obstruction to four crossing changes gives the matching lower bound (Section~\ref{sec:lower}).\strut\par
\strut $^{\mathrm{G}}$ correction-term candidates obtained from Greene's model (Section~\ref{sec:sweep}).\strut\par
\end{minipage}}
\begin{longtable}{*{7}{>{\raggedright\arraybackslash}p{\dimexpr\textwidth/7-2\tabcolsep\relax}}}
\toprule
\endfirsthead
\multicolumn{7}{@{}l@{}}{\emph{Knots with $u=5$, continued from the previous page.}} \\[2pt]
\toprule
\endhead
\bottomrule
\multicolumn{7}{@{}r@{}}{\emph{continued on the next page}} \\
\endfoot
\endlastfoot
13a3685 & 13a4085 & 13a4550 & 13a4824 & 13a4849 & 13a4875 & 13n4970$^{\mathrm{G}}$ \\
13a3712 & 13a4212 & 13a4562 & 13a4826 & 13a4858 & 13n4361$^{\mathrm{G}}$ & 13n4973$^{\mathrm{G}}$ \\*
13a3714 & 13a4368 & 13a4740 & 13a4830 & 13a4863 & 13n4789$^{\mathrm{G}}$ & 13n4991$^{\mathrm{G}}$ \\*
13a4068 & 13a4381 & 13a4763 & 13a4832 & 13a4864 & 13n4876$^{\mathrm{G}}$ & 13n5102$^{\mathrm{G}}$ \\*
\multicolumn{7}{@{}l@{}}{\usebox{\appendixLegendBox}} \\
\end{longtable}}

\section{Knots with \texorpdfstring{$u=4$}{u=4}}\label{app:u4}
The following \nUfour{} knots have unknotting number $4$.  The entries are ordered down the
columns.

{\footnotesize
\setlength{\tabcolsep}{4pt}
\setlength{\LTleft}{\fill}
\setlength{\LTright}{\fill}
\sbox{\appendixLegendBox}{%
\begin{minipage}{\textwidth}
\hrule height\heavyrulewidth\vspace{3pt}
\raggedright\setlength{\parindent}{0pt}\setlength{\parskip}{0pt}
\strut For every knot listed, the correction-term obstruction to three crossing changes gives the matching lower bound (Section~\ref{sec:lower}).\strut\par
\strut $^{\mathrm{G}}$ correction-term candidates obtained from Greene's model (Section~\ref{sec:sweep}).\strut\par
\end{minipage}}
\begin{longtable}{*{7}{>{\raggedright\arraybackslash}p{\dimexpr\textwidth/7-2\tabcolsep\relax}}}
\toprule
\endfirsthead
\multicolumn{7}{@{}l@{}}{\emph{Knots with $u=4$, continued from the previous page.}} \\[2pt]
\toprule
\endhead
\bottomrule
\multicolumn{7}{@{}r@{}}{\emph{continued on the next page}} \\
\endfoot
\endlastfoot
11a291 & 12a973 & 13a1306 & 13a2977 & 13a4215 & 13a4788 & 13n3671$^{\mathrm{G}}$ \\
11a298 & 12a995 & 13a1353 & 13a3060 & 13a4217 & 13a4789 & 13n3707$^{\mathrm{G}}$ \\
11a319 & 12a1004 & 13a1354 & 13a3095 & 13a4303 & 13a4800 & 13n3711$^{\mathrm{G}}$ \\
11a336 & 12a1035 & 13a1356 & 13a3098 & 13a4332 & 13a4802 & 13n4256$^{\mathrm{G}}$ \\
11a340 & 12a1112 & 13a1467 & 13a3103 & 13a4366 & 13a4805 & 13n4313$^{\mathrm{G}}$ \\
11a353 & 12n100$^{\mathrm{G}}$ & 13a1471 & 13a3117 & 13a4369 & 13a4825 & 13n4362$^{\mathrm{G}}$ \\
11a356 & 12n169$^{\mathrm{G}}$ & 13a1893 & 13a3120 & 13a4380 & 13a4827 & 13n4363$^{\mathrm{G}}$ \\
11a357 & 12n245$^{\mathrm{G}}$ & 13a1911 & 13a3130 & 13a4382 & 13a4831 & 13n4790$^{\mathrm{G}}$ \\
11n169$^{\mathrm{G}}$ & 12n259$^{\mathrm{G}}$ & 13a1941 & 13a3594 & 13a4383 & 13a4833 & 13n4794$^{\mathrm{G}}$ \\
11n180$^{\mathrm{G}}$ & 12n406$^{\mathrm{G}}$ & 13a2076 & 13a3686 & 13a4388 & 13a4850 & 13n4825$^{\mathrm{G}}$ \\
12a94 & 12n453$^{\mathrm{G}}$ & 13a2128 & 13a3687 & 13a4395 & 13a4852 & 13n4877$^{\mathrm{G}}$ \\
12a102 & 12n477$^{\mathrm{G}}$ & 13a2149 & 13a3713 & 13a4453 & 13a4854 & 13n4883$^{\mathrm{G}}$ \\
12a107 & 12n758$^{\mathrm{G}}$ & 13a2289 & 13a3715 & 13a4459 & 13a4867 & 13n4888$^{\mathrm{G}}$ \\
12a144 & 13a248 & 13a2333 & 13a3962 & 13a4464 & 13a4868 & 13n4900$^{\mathrm{G}}$ \\
12a145 & 13a357 & 13a2342 & 13a3971 & 13a4492 & 13n1921$^{\mathrm{G}}$ & 13n4971$^{\mathrm{G}}$ \\
12a319 & 13a507 & 13a2620 & 13a4054 & 13a4551 & 13n2318$^{\mathrm{G}}$ & 13n4972$^{\mathrm{G}}$ \\
12a368 & 13a517 & 13a2725 & 13a4069 & 13a4552 & 13n2709$^{\mathrm{G}}$ & 13n4992$^{\mathrm{G}}$ \\
12a391 & 13a526 & 13a2759 & 13a4086 & 13a4553 & 13n2903$^{\mathrm{G}}$ & 13n4994$^{\mathrm{G}}$ \\
12a431 & 13a543 & 13a2760 & 13a4087 & 13a4563 & 13n2962$^{\mathrm{G}}$ & 13n5037$^{\mathrm{G}}$ \\
12a586 & 13a647 & 13a2773 & 13a4095 & 13a4564 & 13n3097$^{\mathrm{G}}$ & 13n5063$^{\mathrm{G}}$ \\
12a659 & 13a648 & 13a2857 & 13a4099 & 13a4567 & 13n3210$^{\mathrm{G}}$ & 13n5094$^{\mathrm{G}}$ \\
12a814 & 13a660 & 13a2867 & 13a4105 & 13a4569 & 13n3335$^{\mathrm{G}}$ & 13n5103$^{\mathrm{G}}$ \\
12a828 & 13a665 & 13a2871 & 13a4190 & 13a4764 & 13n3495$^{\mathrm{G}}$ &  \\*
12a877 & 13a1061 & 13a2955 & 13a4195 & 13a4772 & 13n3500$^{\mathrm{G}}$ &  \\*
12a900 & 13a1270 & 13a2957 & 13a4213 & 13a4774 & 13n3552$^{\mathrm{G}}$ &  \\*
\multicolumn{7}{@{}l@{}}{\usebox{\appendixLegendBox}} \\
\end{longtable}}

\section{Knots with \texorpdfstring{$u=3$}{u=3}}\label{app:u3}
The following \nUthree{} knots have unknotting number $3$.  The entries are ordered down the
columns, and the superscripts are explained at the foot of the table.

{\footnotesize
\setlength{\tabcolsep}{4pt}
\setlength{\LTleft}{\fill}
\setlength{\LTright}{\fill}
\sbox{\appendixLegendBox}{%
\begin{minipage}{\textwidth}
\hrule height\heavyrulewidth\vspace{3pt}
\raggedright\setlength{\parindent}{0pt}\setlength{\parskip}{0pt}
\strut Unmarked: the correction-term obstruction to $u=2$ with $|\sigma|=4$ (Section~\ref{sec:lower}).\strut\par
\strut $^{\mathrm{a}}$ arithmetic case of that obstruction.\strut\par
\strut $^{\mathrm{G}}$ Greene's model and the correction-term obstruction (Section~\ref{sec:sweep}).\strut\par
\strut $^{\mathrm{t}}$ combined Owens--Traczyk criterion.\strut\par
\strut $^{\mathrm{g}}$ generator bound of $H_1(\dbc(K))$.\strut\par
\strut $^{\mathrm{c}}$ generator bound of a higher cyclic cover.\strut\par
\end{minipage}}
\begin{longtable}{*{7}{>{\raggedright\arraybackslash}p{\dimexpr\textwidth/7-2\tabcolsep\relax}}}
\toprule
\endfirsthead
\multicolumn{7}{@{}l@{}}{\emph{Knots with $u=3$, continued from the previous page.}} \\[2pt]
\toprule
\endhead
\bottomrule
\multicolumn{7}{@{}r@{}}{\emph{continued on the next page}} \\
\endfoot
\endlastfoot
11a63 & 13a150 & 13a1491 & 13a2784 & 13a4326 & 13n2061$^{\mathrm{G}}$ & 13n3913$^{\mathrm{G}}$ \\
11a64 & 13a154 & 13a1495 & 13a2785 & 13a4385 & 13n2089$^{\mathrm{G}}$ & 13n3945$^{\mathrm{G}}$ \\
11a144 & 13a170 & 13a1499 & 13a2805 & 13a4390 & 13n2092$^{\mathrm{G}}$ & 13n3994$^{\mathrm{G}}$ \\
11a299 & 13a172 & 13a1546 & 13a2809 & 13a4392 & 13n2252$^{\mathrm{G}}$ & 13n4022$^{\mathrm{G}}$ \\
11a320 & 13a209 & 13a1547 & 13a2821 & 13a4396 & 13n2287$^{\mathrm{G}}$ & 13n4088$^{\mathrm{G}}$ \\
11a329 & 13a217 & 13a1552 & 13a2832 & 13a4461 & 13n2309$^{\mathrm{G}}$ & 13n4108$^{\mathrm{G}}$ \\
11n171$^{\mathrm{G}}$ & 13a231 & 13a1554 & 13a2834 & 13a4465 & 13n2335$^{\mathrm{G}}$ & 13n4133$^{\mathrm{G}}$ \\
11n181$^{\mathrm{G}}$ & 13a251 & 13a1557 & 13a2843 & 13a4508 & 13n2345$^{\mathrm{G}}$ & 13n4220$^{\mathrm{G}}$ \\
12a35 & 13a419 & 13a1580 & 13a2859 & 13a4510 & 13n2346$^{\mathrm{G}}$ & 13n4266$^{\mathrm{G}}$ \\
12a37 & 13a422 & 13a1615 & 13a2865 & 13a4556 & 13n2362$^{\mathrm{G}}$ & 13n4314$^{\mathrm{G}}$ \\
12a47 & 13a434 & 13a1616 & 13a2888 & 13a4565 & 13n2407$^{\mathrm{c}}$ & 13n4352$^{\mathrm{G}}$ \\
12a75 & 13a438 & 13a1617 & 13a2890$^{\mathrm{a}}$ & 13a4571 & 13n2408$^{\mathrm{c}}$ & 13n4365$^{\mathrm{G}}$ \\
12a97 & 13a443 & 13a1618 & 13a2911 & 13a4579 & 13n2410$^{\mathrm{c}}$ & 13n4469$^{\mathrm{G}}$ \\
12a152 & 13a568 & 13a1638$^{\mathrm{c}}$ & 13a2965 & 13a4580 & 13n2411$^{\mathrm{c}}$ & 13n4503$^{\mathrm{G}}$ \\
12a231 & 13a571 & 13a1643 & 13a3065 & 13a4624 & 13n2412$^{\mathrm{c}}$ & 13n4563$^{\mathrm{G}}$ \\
12a254 & 13a576 & 13a1721 & 13a3072$^{\mathrm{c}}$ & 13a4655 & 13n2413$^{\mathrm{c}}$ & 13n4565$^{\mathrm{G}}$ \\
12a269 & 13a584 & 13a1786$^{\mathrm{g}}$ & 13a3089 & 13a4756 & 13n2414$^{\mathrm{c}}$ & 13n4568$^{\mathrm{G}}$ \\
12a289 & 13a586 & 13a1826 & 13a3100 & 13a4768 & 13n2423$^{\mathrm{G}}$ & 13n4615$^{\mathrm{G}}$ \\
12a320 & 13a589 & 13a1910 & 13a3121 & 13a4796 & 13n2474$^{\mathrm{G}}$ & 13n4631$^{\mathrm{G}}$ \\
12a331 & 13a644 & 13a1943 & 13a3310 & 13a4810 & 13n2481$^{\mathrm{G}}$ & 13n4679$^{\mathrm{G}}$ \\
12a421 & 13a645 & 13a1974 & 13a3498 & 13a4820 & 13n2518$^{\mathrm{G}}$ & 13n4700$^{\mathrm{G}}$ \\
12a443 & 13a653 & 13a1984 & 13a3531 & 13a4835 & 13n2644$^{\mathrm{G}}$ & 13n4712$^{\mathrm{G}}$ \\
12a610 & 13a661 & 13a2008$^{\mathrm{t}}$ & 13a3601 & 13a4836 & 13n2697$^{\mathrm{G}}$ & 13n4747$^{\mathrm{G}}$ \\
12a634$^{\mathrm{t}}$ & 13a666 & 13a2025 & 13a3607$^{\mathrm{t}}$ & 13n54$^{\mathrm{G}}$ & 13n2785$^{\mathrm{G}}$ & 13n4752$^{\mathrm{G}}$ \\
12a653 & 13a670 & 13a2037 & 13a3691 & 13n111$^{\mathrm{G}}$ & 13n2898$^{\mathrm{G}}$ & 13n4761$^{\mathrm{G}}$ \\
12a661 & 13a796 & 13a2125 & 13a3698 & 13n131$^{\mathrm{G}}$ & 13n2900$^{\mathrm{G}}$ & 13n4774$^{\mathrm{G}}$ \\
12a763 & 13a859 & 13a2185 & 13a3704 & 13n135$^{\mathrm{G}}$ & 13n3010$^{\mathrm{G}}$ & 13n4782$^{\mathrm{G}}$ \\
12a764 & 13a896 & 13a2269 & 13a3705 & 13n142$^{\mathrm{G}}$ & 13n3015$^{\mathrm{G}}$ & 13n4820$^{\mathrm{G}}$ \\
12a795 & 13a901 & 13a2290 & 13a3716 & 13n150$^{\mathrm{G}}$ & 13n3021$^{\mathrm{G}}$ & 13n4826$^{\mathrm{G}}$ \\
12a796 & 13a933 & 13a2291 & 13a3784 & 13n154$^{\mathrm{G}}$ & 13n3029$^{\mathrm{G}}$ & 13n4837$^{\mathrm{G}}$ \\
12a798 & 13a983 & 13a2295 & 13a3814 & 13n156$^{\mathrm{G}}$ & 13n3030$^{\mathrm{G}}$ & 13n4845$^{\mathrm{G}}$ \\
12a800 & 13a1043 & 13a2310 & 13a3902 & 13n183$^{\mathrm{G}}$ & 13n3038$^{\mathrm{G}}$ & 13n4867$^{\mathrm{G}}$ \\
12a812 & 13a1051 & 13a2317 & 13a3963 & 13n187$^{\mathrm{G}}$ & 13n3042$^{\mathrm{G}}$ & 13n4868$^{\mathrm{G}}$ \\
12a880 & 13a1062 & 13a2334 & 13a3975 & 13n190$^{\mathrm{G}}$ & 13n3048$^{\mathrm{G}}$ & 13n4879$^{\mathrm{G}}$ \\
12a938 & 13a1137 & 13a2359 & 13a3981 & 13n196$^{\mathrm{G}}$ & 13n3145$^{\mathrm{G}}$ & 13n4884$^{\mathrm{G}}$ \\
12a967 & 13a1140 & 13a2360 & 13a4057 & 13n200$^{\mathrm{G}}$ & 13n3191$^{\mathrm{G}}$ & 13n4890$^{\mathrm{G}}$ \\
12a974 & 13a1147 & 13a2365 & 13a4080 & 13n204$^{\mathrm{G}}$ & 13n3199$^{\mathrm{G}}$ & 13n4906$^{\mathrm{G}}$ \\
12a978 & 13a1182 & 13a2373 & 13a4091 & 13n280$^{\mathrm{G}}$ & 13n3200$^{\mathrm{G}}$ & 13n4921$^{\mathrm{G}}$ \\
12a983 & 13a1232$^{\mathrm{c}}$ & 13a2379 & 13a4093 & 13n368$^{\mathrm{G}}$ & 13n3203$^{\mathrm{G}}$ & 13n4931$^{\mathrm{G}}$ \\
12a996 & 13a1238$^{\mathrm{c}}$ & 13a2399 & 13a4101 & 13n495$^{\mathrm{G}}$ & 13n3260$^{\mathrm{G}}$ & 13n4941$^{\mathrm{G}}$ \\
12n441$^{\mathrm{G}}$ & 13a1273 & 13a2412 & 13a4104 & 13n744$^{\mathrm{G}}$ & 13n3274$^{\mathrm{G}}$ & 13n4952$^{\mathrm{G}}$ \\
12n698$^{\mathrm{G}}$ & 13a1295 & 13a2434 & 13a4106 & 13n841$^{\mathrm{G}}$ & 13n3276$^{\mathrm{G}}$ & 13n4975$^{\mathrm{G}}$ \\
12n700$^{\mathrm{G}}$ & 13a1317 & 13a2478 & 13a4113 & 13n900$^{\mathrm{G}}$ & 13n3294$^{\mathrm{G}}$ & 13n4985$^{\mathrm{G}}$ \\
12n734$^{\mathrm{G}}$ & 13a1329 & 13a2502 & 13a4156 & 13n947$^{\mathrm{G}}$ & 13n3426$^{\mathrm{G}}$ & 13n4993$^{\mathrm{G}}$ \\
12n796$^{\mathrm{G}}$ & 13a1331 & 13a2510 & 13a4185 & 13n974$^{\mathrm{G}}$ & 13n3454$^{\mathrm{G}}$ & 13n4996$^{\mathrm{G}}$ \\
12n863$^{\mathrm{G}}$ & 13a1339 & 13a2520 & 13a4191 & 13n1390$^{\mathrm{G}}$ & 13n3504$^{\mathrm{G}}$ & 13n5007$^{\mathrm{G}}$ \\
12n867$^{\mathrm{G}}$ & 13a1398 & 13a2697 & 13a4198 & 13n1514$^{\mathrm{G}}$ & 13n3515$^{\mathrm{G}}$ & 13n5051$^{\mathrm{G}}$ \\
13a15 & 13a1403 & 13a2698 & 13a4200 & 13n1519$^{\mathrm{G}}$ & 13n3533$^{\mathrm{G}}$ & 13n5054$^{\mathrm{G}}$ \\
13a16 & 13a1413 & 13a2720$^{\mathrm{g}}$ & 13a4207 & 13n1592$^{\mathrm{G}}$ & 13n3593$^{\mathrm{G}}$ & 13n5059$^{\mathrm{G}}$ \\
13a48 & 13a1440 & 13a2727$^{\mathrm{g}}$ & 13a4216 & 13n1594$^{\mathrm{G}}$ & 13n3629$^{\mathrm{G}}$ &  \\
13a55 & 13a1465 & 13a2745$^{\mathrm{t}}$ & 13a4232 & 13n1712$^{\mathrm{G}}$ & 13n3666$^{\mathrm{G}}$ &  \\
13a60 & 13a1469 & 13a2761 & 13a4293 & 13n1811$^{\mathrm{G}}$ & 13n3786$^{\mathrm{G}}$ &  \\*
13a65 & 13a1472 & 13a2766 & 13a4320 & 13n1986$^{\mathrm{G}}$ & 13n3879$^{\mathrm{G}}$ &  \\*
13a82 & 13a1475 & 13a2777 & 13a4322 & 13n2001$^{\mathrm{G}}$ & 13n3880$^{\mathrm{G}}$ &  \\*
\multicolumn{7}{@{}l@{}}{\usebox{\appendixLegendBox}} \\
\end{longtable}}

\section{Knots with \texorpdfstring{$u=2$}{u=2}}\label{app:u2}
The following \nUtwo{} knots have unknotting number $2$.  Entries are ordered down the
columns, and the superscripts are explained at the foot of the table.

{\footnotesize
\setlength{\tabcolsep}{4pt}
\setlength{\LTleft}{\fill}
\setlength{\LTright}{\fill}
\sbox{\appendixLegendBox}{%
\begin{minipage}{\textwidth}
\hrule height\heavyrulewidth\vspace{3pt}
\raggedright\setlength{\parindent}{0pt}\setlength{\parskip}{0pt}
\strut Unmarked: the obstruction from the linking pairing (Section~\ref{sec:lower}).\strut\par
\strut $^{\mathrm{t}}$ torsion orders.\strut\par
\strut $^{\mathrm{c}}$ generator bound of a higher cyclic cover.\strut\par
\strut $^{\mathrm{m}}$ obstruction for Montesinos knots (Section~\ref{sec:lower}).\strut\par
\strut $^{\mathrm{G}}$ Greene's model and the surgery obstruction (Sections~\ref{sec:lower} and~\ref{sec:bj}).\strut\par
\strut $^{\mathrm{w}}$ explicit crossing-change construction (Section~\ref{sec:upper}).\strut\par
\strut $^{\mathrm{G,w}}$ lower bound from Greene's model and upper bound from our construction.\strut\par
\end{minipage}}
\begin{longtable}{*{7}{>{\raggedright\arraybackslash}p{\dimexpr\textwidth/7-2\tabcolsep\relax}}}
\toprule
\endfirsthead
\multicolumn{7}{@{}l@{}}{\emph{Knots with $u=2$, continued from the previous page.}} \\[2pt]
\toprule
\endhead
\bottomrule
\multicolumn{7}{@{}r@{}}{\emph{continued on the next page}} \\
\endfoot
\endlastfoot
11n3$^{\mathrm{m}}$ & 13n221$^{\mathrm{G,w}}$ & 13n1119$^{\mathrm{G}}$ & 13n2015 & 13n2774 & 13n3513 & 13n4382$^{\mathrm{c}}$ \\
11n17$^{\mathrm{m}}$ & 13n229$^{\mathrm{G}}$ & 13n1120$^{\mathrm{G}}$ & 13n2020 & 13n2775 & 13n3522$^{\mathrm{G}}$ & 13n4384 \\
11n51$^{\mathrm{m}}$ & 13n230 & 13n1129$^{\mathrm{G}}$ & 13n2025$^{\mathrm{G}}$ & 13n2777$^{\mathrm{G}}$ & 13n3526 & 13n4387 \\
11n54$^{\mathrm{m}}$ & 13n236$^{\mathrm{G}}$ & 13n1130$^{\mathrm{G}}$ & 13n2027 & 13n2778 & 13n3529 & 13n4388$^{\mathrm{c}}$ \\
11n60$^{\mathrm{m}}$ & 13n238$^{\mathrm{G}}$ & 13n1132$^{\mathrm{c}}$ & 13n2028 & 13n2780 & 13n3534$^{\mathrm{G}}$ & 13n4395$^{\mathrm{c}}$ \\
11n94$^{\mathrm{G}}$ & 13n239$^{\mathrm{G}}$ & 13n1135$^{\mathrm{G}}$ & 13n2029 & 13n2781$^{\mathrm{G}}$ & 13n3535 & 13n4402$^{\mathrm{G}}$ \\
11n112$^{\mathrm{G}}$ & 13n243$^{\mathrm{G}}$ & 13n1137$^{\mathrm{G}}$ & 13n2036 & 13n2782 & 13n3537 & 13n4406 \\
11n115$^{\mathrm{G}}$ & 13n244$^{\mathrm{G}}$ & 13n1140$^{\mathrm{G}}$ & 13n2037$^{\mathrm{G}}$ & 13n2787$^{\mathrm{G}}$ & 13n3540$^{\mathrm{G}}$ & 13n4410 \\
11n119$^{\mathrm{G}}$ & 13n248$^{\mathrm{G}}$ & 13n1143$^{\mathrm{G}}$ & 13n2039 & 13n2788$^{\mathrm{G}}$ & 13n3548 & 13n4412$^{\mathrm{G}}$ \\
11n120$^{\mathrm{G}}$ & 13n249$^{\mathrm{G}}$ & 13n1148$^{\mathrm{G}}$ & 13n2041$^{\mathrm{G}}$ & 13n2792 & 13n3549$^{\mathrm{G}}$ & 13n4413$^{\mathrm{G}}$ \\
11n122$^{\mathrm{m}}$ & 13n252$^{\mathrm{G}}$ & 13n1152 & 13n2042 & 13n2793$^{\mathrm{G}}$ & 13n3553 & 13n4419 \\
11n128$^{\mathrm{G}}$ & 13n253$^{\mathrm{G}}$ & 13n1156$^{\mathrm{G}}$ & 13n2044$^{\mathrm{G}}$ & 13n2795 & 13n3557 & 13n4423$^{\mathrm{G}}$ \\
11n129$^{\mathrm{G}}$ & 13n256$^{\mathrm{G}}$ & 13n1157$^{\mathrm{G}}$ & 13n2045$^{\mathrm{G}}$ & 13n2796$^{\mathrm{G}}$ & 13n3559$^{\mathrm{G}}$ & 13n4425 \\
11n138$^{\mathrm{m}}$ & 13n257$^{\mathrm{G}}$ & 13n1158$^{\mathrm{G}}$ & 13n2050$^{\mathrm{G}}$ & 13n2805 & 13n3564$^{\mathrm{G}}$ & 13n4426$^{\mathrm{G}}$ \\
11n142$^{\mathrm{G}}$ & 13n263$^{\mathrm{G}}$ & 13n1161$^{\mathrm{G}}$ & 13n2052 & 13n2806 & 13n3565 & 13n4428 \\
11n160$^{\mathrm{G}}$ & 13n264 & 13n1163$^{\mathrm{G}}$ & 13n2055$^{\mathrm{G}}$ & 13n2807 & 13n3571 & 13n4430 \\
11n161$^{\mathrm{G}}$ & 13n267$^{\mathrm{G}}$ & 13n1166$^{\mathrm{t}}$ & 13n2058 & 13n2809$^{\mathrm{w}}$ & 13n3575 & 13n4431$^{\mathrm{G}}$ \\
11n166$^{\mathrm{G}}$ & 13n275$^{\mathrm{G}}$ & 13n1168$^{\mathrm{G}}$ & 13n2060$^{\mathrm{G}}$ & 13n2812 & 13n3579 & 13n4434$^{\mathrm{G}}$ \\
11n172$^{\mathrm{G}}$ & 13n276$^{\mathrm{G}}$ & 13n1170$^{\mathrm{G}}$ & 13n2068$^{\mathrm{G}}$ & 13n2816 & 13n3586 & 13n4436 \\
11n177$^{\mathrm{G}}$ & 13n278$^{\mathrm{G}}$ & 13n1171$^{\mathrm{G}}$ & 13n2069$^{\mathrm{c}}$ & 13n2820 & 13n3589$^{\mathrm{w}}$ & 13n4439$^{\mathrm{G}}$ \\
11n179$^{\mathrm{G}}$ & 13n279$^{\mathrm{G}}$ & 13n1183$^{\mathrm{G}}$ & 13n2071$^{\mathrm{G}}$ & 13n2821$^{\mathrm{G}}$ & 13n3590$^{\mathrm{G}}$ & 13n4445 \\
11n182$^{\mathrm{G}}$ & 13n281$^{\mathrm{G}}$ & 13n1184$^{\mathrm{G}}$ & 13n2072$^{\mathrm{G}}$ & 13n2823 & 13n3592$^{\mathrm{G}}$ & 13n4447 \\
12n2$^{\mathrm{G}}$ & 13n288 & 13n1185$^{\mathrm{G}}$ & 13n2076 & 13n2825 & 13n3597$^{\mathrm{G}}$ & 13n4449$^{\mathrm{G}}$ \\
12n3$^{\mathrm{G}}$ & 13n289$^{\mathrm{G}}$ & 13n1186 & 13n2078 & 13n2827$^{\mathrm{G}}$ & 13n3598 & 13n4450$^{\mathrm{G}}$ \\
12n9$^{\mathrm{G}}$ & 13n295$^{\mathrm{G}}$ & 13n1188$^{\mathrm{G}}$ & 13n2080$^{\mathrm{G}}$ & 13n2829$^{\mathrm{G}}$ & 13n3599$^{\mathrm{G}}$ & 13n4451$^{\mathrm{G}}$ \\
12n11$^{\mathrm{m}}$ & 13n297$^{\mathrm{G}}$ & 13n1190$^{\mathrm{G}}$ & 13n2081$^{\mathrm{G}}$ & 13n2832$^{\mathrm{c}}$ & 13n3600$^{\mathrm{G}}$ & 13n4452 \\
12n12$^{\mathrm{m}}$ & 13n298$^{\mathrm{G}}$ & 13n1191$^{\mathrm{G}}$ & 13n2082$^{\mathrm{G}}$ & 13n2834$^{\mathrm{G}}$ & 13n3606 & 13n4453$^{\mathrm{G}}$ \\
12n14$^{\mathrm{G}}$ & 13n302$^{\mathrm{G}}$ & 13n1193$^{\mathrm{G}}$ & 13n2083$^{\mathrm{G}}$ & 13n2836$^{\mathrm{G}}$ & 13n3607$^{\mathrm{G}}$ & 13n4454 \\
12n18$^{\mathrm{G}}$ & 13n303$^{\mathrm{G}}$ & 13n1196$^{\mathrm{G}}$ & 13n2085$^{\mathrm{G}}$ & 13n2837$^{\mathrm{G}}$ & 13n3610 & 13n4457$^{\mathrm{G}}$ \\
12n21$^{\mathrm{G}}$ & 13n306$^{\mathrm{G}}$ & 13n1197$^{\mathrm{G}}$ & 13n2086$^{\mathrm{G}}$ & 13n2838$^{\mathrm{G}}$ & 13n3613 & 13n4458$^{\mathrm{G}}$ \\
12n22$^{\mathrm{G}}$ & 13n311$^{\mathrm{G}}$ & 13n1210$^{\mathrm{G}}$ & 13n2088 & 13n2840$^{\mathrm{G}}$ & 13n3614$^{\mathrm{c}}$ & 13n4463 \\
12n26$^{\mathrm{G}}$ & 13n313$^{\mathrm{G}}$ & 13n1212 & 13n2090 & 13n2841$^{\mathrm{G}}$ & 13n3616$^{\mathrm{G}}$ & 13n4465 \\
12n29$^{\mathrm{G}}$ & 13n314$^{\mathrm{G}}$ & 13n1213$^{\mathrm{G}}$ & 13n2091$^{\mathrm{G}}$ & 13n2842$^{\mathrm{G}}$ & 13n3617$^{\mathrm{G}}$ & 13n4467 \\
12n30$^{\mathrm{G}}$ & 13n316 & 13n1214 & 13n2096$^{\mathrm{G}}$ & 13n2843$^{\mathrm{G}}$ & 13n3620$^{\mathrm{G}}$ & 13n4471 \\
12n32$^{\mathrm{G}}$ & 13n319$^{\mathrm{G}}$ & 13n1217$^{\mathrm{G}}$ & 13n2097 & 13n2844$^{\mathrm{G}}$ & 13n3621$^{\mathrm{G}}$ & 13n4472$^{\mathrm{c}}$ \\
12n41$^{\mathrm{m}}$ & 13n320 & 13n1220$^{\mathrm{G}}$ & 13n2098$^{\mathrm{G}}$ & 13n2847$^{\mathrm{c}}$ & 13n3622$^{\mathrm{G}}$ & 13n4474$^{\mathrm{c}}$ \\
12n42$^{\mathrm{m}}$ & 13n323$^{\mathrm{G}}$ & 13n1223 & 13n2100 & 13n2850$^{\mathrm{G}}$ & 13n3623$^{\mathrm{G}}$ & 13n4476$^{\mathrm{G}}$ \\
12n43$^{\mathrm{m}}$ & 13n324$^{\mathrm{G}}$ & 13n1224$^{\mathrm{G}}$ & 13n2107$^{\mathrm{G}}$ & 13n2851 & 13n3631$^{\mathrm{G}}$ & 13n4477$^{\mathrm{G}}$ \\
12n44$^{\mathrm{m}}$ & 13n327 & 13n1230$^{\mathrm{G}}$ & 13n2108$^{\mathrm{G}}$ & 13n2852$^{\mathrm{G}}$ & 13n3636$^{\mathrm{G}}$ & 13n4478$^{\mathrm{G}}$ \\
12n45$^{\mathrm{m}}$ & 13n329 & 13n1232$^{\mathrm{G}}$ & 13n2115$^{\mathrm{G}}$ & 13n2853 & 13n3637$^{\mathrm{G}}$ & 13n4481 \\
12n47$^{\mathrm{m}}$ & 13n331$^{\mathrm{G}}$ & 13n1235$^{\mathrm{G}}$ & 13n2118$^{\mathrm{G}}$ & 13n2855$^{\mathrm{G}}$ & 13n3640$^{\mathrm{G}}$ & 13n4482$^{\mathrm{G}}$ \\
12n48$^{\mathrm{m}}$ & 13n334$^{\mathrm{G}}$ & 13n1244$^{\mathrm{G}}$ & 13n2119$^{\mathrm{G}}$ & 13n2857$^{\mathrm{G}}$ & 13n3641 & 13n4483$^{\mathrm{c}}$ \\
12n49$^{\mathrm{G}}$ & 13n336$^{\mathrm{G}}$ & 13n1247$^{\mathrm{G}}$ & 13n2124 & 13n2860 & 13n3649 & 13n4484 \\
12n50$^{\mathrm{G}}$ & 13n337$^{\mathrm{G}}$ & 13n1250$^{\mathrm{G}}$ & 13n2129$^{\mathrm{G}}$ & 13n2861 & 13n3651 & 13n4488$^{\mathrm{G}}$ \\
12n53$^{\mathrm{G}}$ & 13n340$^{\mathrm{G}}$ & 13n1253$^{\mathrm{G}}$ & 13n2138 & 13n2863$^{\mathrm{G}}$ & 13n3652$^{\mathrm{G}}$ & 13n4491$^{\mathrm{G}}$ \\
12n71$^{\mathrm{G}}$ & 13n341$^{\mathrm{G}}$ & 13n1255$^{\mathrm{G}}$ & 13n2139$^{\mathrm{G}}$ & 13n2864 & 13n3657 & 13n4494 \\
12n80$^{\mathrm{G}}$ & 13n342 & 13n1258 & 13n2143$^{\mathrm{G}}$ & 13n2865 & 13n3667 & 13n4495 \\
12n86$^{\mathrm{G}}$ & 13n343 & 13n1259$^{\mathrm{G}}$ & 13n2145$^{\mathrm{G}}$ & 13n2868 & 13n3669$^{\mathrm{G}}$ & 13n4498$^{\mathrm{G}}$ \\
12n97$^{\mathrm{G}}$ & 13n344$^{\mathrm{G}}$ & 13n1260$^{\mathrm{G}}$ & 13n2148 & 13n2869 & 13n3672$^{\mathrm{G}}$ & 13n4500 \\
12n106$^{\mathrm{G}}$ & 13n346$^{\mathrm{G}}$ & 13n1263 & 13n2155 & 13n2870$^{\mathrm{c}}$ & 13n3678$^{\mathrm{G}}$ & 13n4509$^{\mathrm{G}}$ \\
12n111$^{\mathrm{G}}$ & 13n349 & 13n1265$^{\mathrm{G}}$ & 13n2156 & 13n2871$^{\mathrm{G}}$ & 13n3679$^{\mathrm{G}}$ & 13n4510$^{\mathrm{G}}$ \\
12n122$^{\mathrm{G}}$ & 13n351 & 13n1289$^{\mathrm{G}}$ & 13n2158$^{\mathrm{G}}$ & 13n2875$^{\mathrm{G}}$ & 13n3682$^{\mathrm{G}}$ & 13n4512 \\
12n127$^{\mathrm{G}}$ & 13n353$^{\mathrm{G}}$ & 13n1303$^{\mathrm{G}}$ & 13n2160$^{\mathrm{G}}$ & 13n2878$^{\mathrm{G}}$ & 13n3684 & 13n4513$^{\mathrm{G}}$ \\
12n131$^{\mathrm{G}}$ & 13n355$^{\mathrm{G}}$ & 13n1318$^{\mathrm{G}}$ & 13n2161 & 13n2880$^{\mathrm{G}}$ & 13n3689 & 13n4514$^{\mathrm{G}}$ \\
12n154$^{\mathrm{m}}$ & 13n356 & 13n1326$^{\mathrm{G}}$ & 13n2162 & 13n2883 & 13n3694$^{\mathrm{G}}$ & 13n4517$^{\mathrm{G}}$ \\
12n158$^{\mathrm{G}}$ & 13n357$^{\mathrm{G}}$ & 13n1335 & 13n2164 & 13n2887 & 13n3699 & 13n4518$^{\mathrm{G}}$ \\
12n159$^{\mathrm{m}}$ & 13n359$^{\mathrm{G}}$ & 13n1336$^{\mathrm{G}}$ & 13n2165 & 13n2888$^{\mathrm{G}}$ & 13n3704$^{\mathrm{G}}$ & 13n4521$^{\mathrm{G}}$ \\
12n160$^{\mathrm{m}}$ & 13n371 & 13n1338 & 13n2168 & 13n2889 & 13n3705$^{\mathrm{G}}$ & 13n4523 \\
12n162$^{\mathrm{m}}$ & 13n374 & 13n1339$^{\mathrm{G}}$ & 13n2169 & 13n2890$^{\mathrm{G}}$ & 13n3708 & 13n4525$^{\mathrm{G}}$ \\
12n170$^{\mathrm{m}}$ & 13n377$^{\mathrm{G}}$ & 13n1342 & 13n2170$^{\mathrm{G}}$ & 13n2891$^{\mathrm{G}}$ & 13n3714 & 13n4527 \\
12n176$^{\mathrm{G}}$ & 13n379$^{\mathrm{G}}$ & 13n1343$^{\mathrm{G}}$ & 13n2173 & 13n2892 & 13n3715 & 13n4528$^{\mathrm{G}}$ \\
12n178$^{\mathrm{G}}$ & 13n381$^{\mathrm{G}}$ & 13n1347$^{\mathrm{G}}$ & 13n2175$^{\mathrm{G}}$ & 13n2895 & 13n3716 & 13n4530$^{\mathrm{G}}$ \\
12n180$^{\mathrm{G}}$ & 13n382 & 13n1349 & 13n2176$^{\mathrm{G}}$ & 13n2902 & 13n3717 & 13n4532$^{\mathrm{G}}$ \\
12n181$^{\mathrm{G}}$ & 13n385$^{\mathrm{G}}$ & 13n1353 & 13n2177 & 13n2905$^{\mathrm{G}}$ & 13n3719 & 13n4533$^{\mathrm{G}}$ \\
12n182$^{\mathrm{G}}$ & 13n386$^{\mathrm{G}}$ & 13n1357$^{\mathrm{c}}$ & 13n2184$^{\mathrm{G}}$ & 13n2907$^{\mathrm{w}}$ & 13n3723 & 13n4534 \\
12n189$^{\mathrm{G}}$ & 13n389$^{\mathrm{G}}$ & 13n1359 & 13n2185$^{\mathrm{G}}$ & 13n2908 & 13n3725 & 13n4535$^{\mathrm{G}}$ \\
12n197$^{\mathrm{G}}$ & 13n390$^{\mathrm{G}}$ & 13n1360 & 13n2187 & 13n2910 & 13n3726 & 13n4536$^{\mathrm{G}}$ \\
12n238$^{\mathrm{m}}$ & 13n392$^{\mathrm{G}}$ & 13n1364$^{\mathrm{G}}$ & 13n2189 & 13n2914 & 13n3727 & 13n4537$^{\mathrm{G}}$ \\
12n241$^{\mathrm{m}}$ & 13n394$^{\mathrm{G}}$ & 13n1367$^{\mathrm{G}}$ & 13n2198$^{\mathrm{c}}$ & 13n2915 & 13n3730 & 13n4540 \\
12n246$^{\mathrm{G}}$ & 13n395$^{\mathrm{G}}$ & 13n1368$^{\mathrm{G}}$ & 13n2199 & 13n2919 & 13n3733$^{\mathrm{w}}$ & 13n4542 \\
12n280$^{\mathrm{G}}$ & 13n398$^{\mathrm{G}}$ & 13n1369$^{\mathrm{G}}$ & 13n2207$^{\mathrm{G}}$ & 13n2923$^{\mathrm{c}}$ & 13n3736 & 13n4543$^{\mathrm{G}}$ \\
12n281$^{\mathrm{G}}$ & 13n399$^{\mathrm{G}}$ & 13n1370$^{\mathrm{G}}$ & 13n2208 & 13n2924 & 13n3758$^{\mathrm{G}}$ & 13n4545$^{\mathrm{G}}$ \\
12n284$^{\mathrm{G}}$ & 13n401 & 13n1373$^{\mathrm{G}}$ & 13n2209 & 13n2925 & 13n3788$^{\mathrm{G}}$ & 13n4546 \\
12n285$^{\mathrm{G}}$ & 13n403$^{\mathrm{G}}$ & 13n1374$^{\mathrm{G}}$ & 13n2210 & 13n2926 & 13n3789$^{\mathrm{G}}$ & 13n4547$^{\mathrm{G}}$ \\
12n286$^{\mathrm{m}}$ & 13n404$^{\mathrm{G}}$ & 13n1375$^{\mathrm{G}}$ & 13n2214 & 13n2927 & 13n3794 & 13n4548$^{\mathrm{G}}$ \\
12n287$^{\mathrm{G}}$ & 13n408 & 13n1377 & 13n2215 & 13n2931 & 13n3795$^{\mathrm{G}}$ & 13n4549$^{\mathrm{G}}$ \\
12n288$^{\mathrm{m}}$ & 13n410$^{\mathrm{G}}$ & 13n1378 & 13n2216$^{\mathrm{G}}$ & 13n2932$^{\mathrm{G}}$ & 13n3796 & 13n4550 \\
12n295$^{\mathrm{m}}$ & 13n411$^{\mathrm{G}}$ & 13n1379 & 13n2218$^{\mathrm{G}}$ & 13n2934 & 13n3797 & 13n4551 \\
12n298$^{\mathrm{G}}$ & 13n414 & 13n1388 & 13n2221$^{\mathrm{G}}$ & 13n2936 & 13n3798$^{\mathrm{G}}$ & 13n4553 \\
12n301$^{\mathrm{m}}$ & 13n415$^{\mathrm{G}}$ & 13n1392 & 13n2227$^{\mathrm{G}}$ & 13n2937 & 13n3802 & 13n4554 \\
12n302$^{\mathrm{m}}$ & 13n416$^{\mathrm{G}}$ & 13n1397$^{\mathrm{G}}$ & 13n2228$^{\mathrm{G}}$ & 13n2940 & 13n3803 & 13n4555$^{\mathrm{G}}$ \\
12n312$^{\mathrm{G}}$ & 13n422$^{\mathrm{G}}$ & 13n1398$^{\mathrm{G}}$ & 13n2229$^{\mathrm{G}}$ & 13n2942$^{\mathrm{G}}$ & 13n3804$^{\mathrm{G}}$ & 13n4556 \\
12n315$^{\mathrm{G}}$ & 13n423$^{\mathrm{G}}$ & 13n1400$^{\mathrm{c}}$ & 13n2231 & 13n2945$^{\mathrm{G}}$ & 13n3806 & 13n4557$^{\mathrm{G}}$ \\
12n323$^{\mathrm{G}}$ & 13n424$^{\mathrm{G}}$ & 13n1401$^{\mathrm{G}}$ & 13n2236$^{\mathrm{G}}$ & 13n2949 & 13n3807 & 13n4558 \\
12n324$^{\mathrm{G}}$ & 13n432$^{\mathrm{G}}$ & 13n1402$^{\mathrm{G}}$ & 13n2238$^{\mathrm{G}}$ & 13n2953 & 13n3813 & 13n4564$^{\mathrm{G}}$ \\
12n330$^{\mathrm{G}}$ & 13n434$^{\mathrm{G}}$ & 13n1409 & 13n2240 & 13n2954 & 13n3816 & 13n4566 \\
12n337$^{\mathrm{m}}$ & 13n435$^{\mathrm{G}}$ & 13n1410$^{\mathrm{G}}$ & 13n2241$^{\mathrm{G}}$ & 13n2955 & 13n3822 & 13n4567 \\
12n345$^{\mathrm{G}}$ & 13n436$^{\mathrm{G,w}}$ & 13n1413$^{\mathrm{G}}$ & 13n2243$^{\mathrm{G}}$ & 13n2960$^{\mathrm{G}}$ & 13n3823$^{\mathrm{G}}$ & 13n4570$^{\mathrm{c}}$ \\
12n348$^{\mathrm{G}}$ & 13n439$^{\mathrm{G}}$ & 13n1417 & 13n2246$^{\mathrm{G}}$ & 13n2964$^{\mathrm{G}}$ & 13n3824 & 13n4574$^{\mathrm{c}}$ \\
12n353$^{\mathrm{G}}$ & 13n443$^{\mathrm{G}}$ & 13n1418$^{\mathrm{G}}$ & 13n2247$^{\mathrm{G}}$ & 13n2967$^{\mathrm{c}}$ & 13n3825 & 13n4576 \\
12n359$^{\mathrm{G}}$ & 13n444$^{\mathrm{G}}$ & 13n1424 & 13n2251$^{\mathrm{G,w}}$ & 13n2980$^{\mathrm{G}}$ & 13n3830 & 13n4577$^{\mathrm{G}}$ \\
12n363$^{\mathrm{G}}$ & 13n447$^{\mathrm{G,w}}$ & 13n1428$^{\mathrm{G}}$ & 13n2254$^{\mathrm{G}}$ & 13n2981$^{\mathrm{c}}$ & 13n3834 & 13n4581$^{\mathrm{G}}$ \\
12n367$^{\mathrm{G}}$ & 13n448$^{\mathrm{G}}$ & 13n1430$^{\mathrm{G}}$ & 13n2257$^{\mathrm{G}}$ & 13n2994 & 13n3839 & 13n4583 \\
12n376$^{\mathrm{G}}$ & 13n457$^{\mathrm{G}}$ & 13n1431$^{\mathrm{G}}$ & 13n2261$^{\mathrm{c}}$ & 13n2996$^{\mathrm{G}}$ & 13n3849$^{\mathrm{G}}$ & 13n4585 \\
12n378$^{\mathrm{G}}$ & 13n463$^{\mathrm{G}}$ & 13n1432$^{\mathrm{G}}$ & 13n2262$^{\mathrm{G}}$ & 13n2997 & 13n3852 & 13n4590$^{\mathrm{G}}$ \\
12n383$^{\mathrm{G}}$ & 13n464$^{\mathrm{G}}$ & 13n1436$^{\mathrm{G}}$ & 13n2263$^{\mathrm{G}}$ & 13n3002 & 13n3860$^{\mathrm{G}}$ & 13n4593$^{\mathrm{G}}$ \\
12n385$^{\mathrm{G}}$ & 13n470$^{\mathrm{G}}$ & 13n1438$^{\mathrm{G}}$ & 13n2264$^{\mathrm{G}}$ & 13n3009 & 13n3882$^{\mathrm{G}}$ & 13n4594 \\
12n390$^{\mathrm{G}}$ & 13n483$^{\mathrm{G}}$ & 13n1439$^{\mathrm{G,w}}$ & 13n2265$^{\mathrm{G}}$ & 13n3013 & 13n3883$^{\mathrm{G}}$ & 13n4598$^{\mathrm{G}}$ \\
12n396$^{\mathrm{G}}$ & 13n485$^{\mathrm{G}}$ & 13n1445 & 13n2268$^{\mathrm{G}}$ & 13n3016$^{\mathrm{G}}$ & 13n3889 & 13n4599$^{\mathrm{G}}$ \\
12n399$^{\mathrm{G}}$ & 13n497$^{\mathrm{G}}$ & 13n1446$^{\mathrm{G}}$ & 13n2269$^{\mathrm{G}}$ & 13n3018$^{\mathrm{G}}$ & 13n3900 & 13n4600$^{\mathrm{G}}$ \\
12n408$^{\mathrm{G}}$ & 13n498$^{\mathrm{G}}$ & 13n1447$^{\mathrm{G}}$ & 13n2273$^{\mathrm{G}}$ & 13n3027$^{\mathrm{G}}$ & 13n3905 & 13n4601$^{\mathrm{G}}$ \\
12n415$^{\mathrm{G}}$ & 13n503$^{\mathrm{G}}$ & 13n1449 & 13n2274$^{\mathrm{G}}$ & 13n3031$^{\mathrm{G}}$ & 13n3907 & 13n4602$^{\mathrm{G}}$ \\
12n421$^{\mathrm{G}}$ & 13n505$^{\mathrm{G}}$ & 13n1450 & 13n2275$^{\mathrm{G}}$ & 13n3032 & 13n3908$^{\mathrm{G}}$ & 13n4605 \\
12n422$^{\mathrm{G}}$ & 13n507$^{\mathrm{G}}$ & 13n1451 & 13n2280$^{\mathrm{c}}$ & 13n3033$^{\mathrm{w}}$ & 13n3909 & 13n4610$^{\mathrm{G}}$ \\
12n431$^{\mathrm{G}}$ & 13n511 & 13n1455$^{\mathrm{G}}$ & 13n2285$^{\mathrm{G}}$ & 13n3036 & 13n3918 & 13n4611$^{\mathrm{G}}$ \\
12n444$^{\mathrm{m}}$ & 13n513 & 13n1457 & 13n2286 & 13n3045 & 13n3920 & 13n4616$^{\mathrm{G}}$ \\
12n448$^{\mathrm{G}}$ & 13n515$^{\mathrm{G}}$ & 13n1458 & 13n2290 & 13n3046 & 13n3921 & 13n4621$^{\mathrm{G}}$ \\
12n452$^{\mathrm{G}}$ & 13n518$^{\mathrm{G}}$ & 13n1459 & 13n2291$^{\mathrm{G}}$ & 13n3049$^{\mathrm{c}}$ & 13n3926$^{\mathrm{G}}$ & 13n4627$^{\mathrm{G}}$ \\
12n457$^{\mathrm{m}}$ & 13n519 & 13n1460 & 13n2293 & 13n3050$^{\mathrm{G}}$ & 13n3927$^{\mathrm{G}}$ & 13n4628 \\
12n458$^{\mathrm{G}}$ & 13n520$^{\mathrm{G}}$ & 13n1461 & 13n2301$^{\mathrm{c}}$ & 13n3051 & 13n3928 & 13n4632 \\
12n463$^{\mathrm{G}}$ & 13n522$^{\mathrm{G}}$ & 13n1462 & 13n2311 & 13n3056$^{\mathrm{G}}$ & 13n3929 & 13n4642$^{\mathrm{G}}$ \\
12n465$^{\mathrm{G}}$ & 13n525 & 13n1466$^{\mathrm{G}}$ & 13n2313 & 13n3057$^{\mathrm{G}}$ & 13n3940 & 13n4644$^{\mathrm{G}}$ \\
12n470$^{\mathrm{m}}$ & 13n527$^{\mathrm{G}}$ & 13n1468 & 13n2314 & 13n3062$^{\mathrm{G}}$ & 13n3950 & 13n4647$^{\mathrm{G}}$ \\
12n471$^{\mathrm{m}}$ & 13n529$^{\mathrm{G}}$ & 13n1470$^{\mathrm{G}}$ & 13n2315 & 13n3063$^{\mathrm{G}}$ & 13n3952$^{\mathrm{G}}$ & 13n4649$^{\mathrm{G}}$ \\
12n475$^{\mathrm{m}}$ & 13n530$^{\mathrm{G}}$ & 13n1472$^{\mathrm{G}}$ & 13n2316 & 13n3065$^{\mathrm{G}}$ & 13n3953 & 13n4651$^{\mathrm{G}}$ \\
12n478$^{\mathrm{m}}$ & 13n532$^{\mathrm{G}}$ & 13n1476$^{\mathrm{G}}$ & 13n2319$^{\mathrm{G}}$ & 13n3068 & 13n3955$^{\mathrm{G}}$ & 13n4652$^{\mathrm{G}}$ \\
12n479$^{\mathrm{m}}$ & 13n535 & 13n1480 & 13n2321$^{\mathrm{G}}$ & 13n3073$^{\mathrm{G}}$ & 13n3956$^{\mathrm{G}}$ & 13n4656 \\
12n489$^{\mathrm{G}}$ & 13n541$^{\mathrm{G}}$ & 13n1481$^{\mathrm{c}}$ & 13n2322$^{\mathrm{G}}$ & 13n3074 & 13n3957 & 13n4657 \\
12n490$^{\mathrm{G}}$ & 13n544$^{\mathrm{G}}$ & 13n1482$^{\mathrm{G}}$ & 13n2323 & 13n3075 & 13n3958$^{\mathrm{G}}$ & 13n4659 \\
12n491$^{\mathrm{G}}$ & 13n545 & 13n1483$^{\mathrm{G}}$ & 13n2324 & 13n3078$^{\mathrm{G}}$ & 13n3959$^{\mathrm{G}}$ & 13n4660 \\
12n500$^{\mathrm{m}}$ & 13n550 & 13n1485$^{\mathrm{G}}$ & 13n2326$^{\mathrm{G}}$ & 13n3083 & 13n3966$^{\mathrm{G}}$ & 13n4662$^{\mathrm{G}}$ \\
12n501$^{\mathrm{m}}$ & 13n552$^{\mathrm{G}}$ & 13n1490$^{\mathrm{G}}$ & 13n2328$^{\mathrm{G}}$ & 13n3087$^{\mathrm{G}}$ & 13n3967$^{\mathrm{G}}$ & 13n4663 \\
12n506$^{\mathrm{G}}$ & 13n558$^{\mathrm{c}}$ & 13n1492 & 13n2329$^{\mathrm{G}}$ & 13n3090 & 13n3968 & 13n4664$^{\mathrm{G}}$ \\
12n507$^{\mathrm{G}}$ & 13n559$^{\mathrm{G}}$ & 13n1494 & 13n2330$^{\mathrm{G}}$ & 13n3093$^{\mathrm{G}}$ & 13n3969$^{\mathrm{G}}$ & 13n4665$^{\mathrm{G}}$ \\
12n514$^{\mathrm{G}}$ & 13n562$^{\mathrm{c}}$ & 13n1497$^{\mathrm{G}}$ & 13n2331$^{\mathrm{G}}$ & 13n3095$^{\mathrm{G}}$ & 13n3970$^{\mathrm{G}}$ & 13n4666$^{\mathrm{G}}$ \\
12n517$^{\mathrm{G}}$ & 13n565$^{\mathrm{G}}$ & 13n1500 & 13n2332$^{\mathrm{G}}$ & 13n3096$^{\mathrm{c}}$ & 13n3973$^{\mathrm{G}}$ & 13n4669$^{\mathrm{G}}$ \\
12n520$^{\mathrm{G}}$ & 13n566$^{\mathrm{G}}$ & 13n1502$^{\mathrm{G}}$ & 13n2336$^{\mathrm{G}}$ & 13n3101 & 13n3974$^{\mathrm{G}}$ & 13n4670 \\
12n521$^{\mathrm{G}}$ & 13n568$^{\mathrm{G}}$ & 13n1508$^{\mathrm{c}}$ & 13n2339$^{\mathrm{G}}$ & 13n3102$^{\mathrm{G}}$ & 13n3976$^{\mathrm{G}}$ & 13n4671 \\
12n523$^{\mathrm{m}}$ & 13n569 & 13n1512 & 13n2342$^{\mathrm{G}}$ & 13n3103$^{\mathrm{G}}$ & 13n3977 & 13n4673 \\
12n525$^{\mathrm{G}}$ & 13n576$^{\mathrm{G}}$ & 13n1513 & 13n2347 & 13n3105 & 13n3985$^{\mathrm{G}}$ & 13n4675$^{\mathrm{G}}$ \\
12n529$^{\mathrm{G}}$ & 13n579 & 13n1515$^{\mathrm{G}}$ & 13n2359$^{\mathrm{G}}$ & 13n3107$^{\mathrm{G}}$ & 13n3988$^{\mathrm{G}}$ & 13n4676 \\
12n530$^{\mathrm{G}}$ & 13n580$^{\mathrm{G}}$ & 13n1518$^{\mathrm{G}}$ & 13n2360$^{\mathrm{G}}$ & 13n3108$^{\mathrm{G,w}}$ & 13n3996 & 13n4678$^{\mathrm{G}}$ \\
12n532$^{\mathrm{G}}$ & 13n581$^{\mathrm{G}}$ & 13n1523$^{\mathrm{G}}$ & 13n2361$^{\mathrm{G}}$ & 13n3110$^{\mathrm{G}}$ & 13n3999$^{\mathrm{G}}$ & 13n4682$^{\mathrm{G}}$ \\
12n537$^{\mathrm{G}}$ & 13n582$^{\mathrm{G}}$ & 13n1525 & 13n2363$^{\mathrm{G}}$ & 13n3113 & 13n4001 & 13n4685 \\
12n539$^{\mathrm{G}}$ & 13n590$^{\mathrm{G}}$ & 13n1527$^{\mathrm{G}}$ & 13n2366$^{\mathrm{G}}$ & 13n3116$^{\mathrm{G}}$ & 13n4003$^{\mathrm{G}}$ & 13n4688 \\
12n541$^{\mathrm{G}}$ & 13n591$^{\mathrm{G}}$ & 13n1531$^{\mathrm{G}}$ & 13n2367$^{\mathrm{G}}$ & 13n3117$^{\mathrm{G}}$ & 13n4006$^{\mathrm{G}}$ & 13n4689 \\
12n542$^{\mathrm{G}}$ & 13n595$^{\mathrm{G}}$ & 13n1536 & 13n2370 & 13n3120$^{\mathrm{G}}$ & 13n4008 & 13n4692$^{\mathrm{G}}$ \\
12n544$^{\mathrm{G}}$ & 13n598$^{\mathrm{G}}$ & 13n1538$^{\mathrm{G}}$ & 13n2372 & 13n3123$^{\mathrm{G}}$ & 13n4012 & 13n4693 \\
12n548$^{\mathrm{G}}$ & 13n599$^{\mathrm{G}}$ & 13n1540$^{\mathrm{G}}$ & 13n2373$^{\mathrm{G}}$ & 13n3125$^{\mathrm{G}}$ & 13n4014$^{\mathrm{G}}$ & 13n4697$^{\mathrm{G}}$ \\
12n550$^{\mathrm{m}}$ & 13n600$^{\mathrm{G}}$ & 13n1542$^{\mathrm{G}}$ & 13n2377$^{\mathrm{G}}$ & 13n3126$^{\mathrm{G}}$ & 13n4015$^{\mathrm{G}}$ & 13n4699$^{\mathrm{G}}$ \\
12n563$^{\mathrm{G}}$ & 13n609 & 13n1546 & 13n2379$^{\mathrm{w}}$ & 13n3130$^{\mathrm{G}}$ & 13n4017$^{\mathrm{G}}$ & 13n4701$^{\mathrm{G}}$ \\
12n564$^{\mathrm{m}}$ & 13n610 & 13n1547$^{\mathrm{G}}$ & 13n2380$^{\mathrm{G}}$ & 13n3131$^{\mathrm{G}}$ & 13n4019$^{\mathrm{G}}$ & 13n4702$^{\mathrm{G}}$ \\
12n568$^{\mathrm{G}}$ & 13n613 & 13n1548$^{\mathrm{G}}$ & 13n2384 & 13n3134$^{\mathrm{G}}$ & 13n4020 & 13n4703$^{\mathrm{G}}$ \\
12n577$^{\mathrm{m}}$ & 13n615$^{\mathrm{G}}$ & 13n1549$^{\mathrm{G}}$ & 13n2386$^{\mathrm{G}}$ & 13n3135 & 13n4023$^{\mathrm{G}}$ & 13n4704 \\
12n578$^{\mathrm{m}}$ & 13n620 & 13n1553$^{\mathrm{G}}$ & 13n2387$^{\mathrm{G}}$ & 13n3142$^{\mathrm{G}}$ & 13n4025$^{\mathrm{G,w}}$ & 13n4706$^{\mathrm{G}}$ \\
12n586$^{\mathrm{G}}$ & 13n625$^{\mathrm{G}}$ & 13n1554$^{\mathrm{G}}$ & 13n2388$^{\mathrm{G}}$ & 13n3146$^{\mathrm{G}}$ & 13n4026 & 13n4709 \\
12n587$^{\mathrm{G}}$ & 13n626$^{\mathrm{G}}$ & 13n1556$^{\mathrm{G}}$ & 13n2389$^{\mathrm{G}}$ & 13n3147$^{\mathrm{G}}$ & 13n4034 & 13n4710 \\
12n595$^{\mathrm{G}}$ & 13n630$^{\mathrm{G}}$ & 13n1557 & 13n2396$^{\mathrm{G}}$ & 13n3151$^{\mathrm{G}}$ & 13n4035 & 13n4713 \\
12n597$^{\mathrm{G}}$ & 13n632$^{\mathrm{G}}$ & 13n1561$^{\mathrm{G}}$ & 13n2398 & 13n3152$^{\mathrm{G}}$ & 13n4037$^{\mathrm{G}}$ & 13n4719$^{\mathrm{G}}$ \\
12n607$^{\mathrm{G}}$ & 13n633$^{\mathrm{G}}$ & 13n1562 & 13n2400 & 13n3153$^{\mathrm{G}}$ & 13n4040$^{\mathrm{G}}$ & 13n4720 \\
12n608$^{\mathrm{G}}$ & 13n634$^{\mathrm{G}}$ & 13n1563$^{\mathrm{G}}$ & 13n2401 & 13n3155$^{\mathrm{G}}$ & 13n4041 & 13n4724 \\
12n618$^{\mathrm{G}}$ & 13n635$^{\mathrm{G}}$ & 13n1564$^{\mathrm{G}}$ & 13n2416 & 13n3156 & 13n4042$^{\mathrm{G}}$ & 13n4726$^{\mathrm{c}}$ \\
12n625$^{\mathrm{G}}$ & 13n636$^{\mathrm{G,w}}$ & 13n1565$^{\mathrm{G}}$ & 13n2421 & 13n3157$^{\mathrm{G}}$ & 13n4043$^{\mathrm{G}}$ & 13n4727 \\
12n627$^{\mathrm{G}}$ & 13n638$^{\mathrm{G}}$ & 13n1567 & 13n2424$^{\mathrm{G}}$ & 13n3158$^{\mathrm{c}}$ & 13n4046$^{\mathrm{c}}$ & 13n4729 \\
12n632$^{\mathrm{G}}$ & 13n640$^{\mathrm{G}}$ & 13n1568$^{\mathrm{G}}$ & 13n2426 & 13n3159$^{\mathrm{G}}$ & 13n4048$^{\mathrm{c}}$ & 13n4730 \\
12n635$^{\mathrm{G}}$ & 13n644$^{\mathrm{G}}$ & 13n1570$^{\mathrm{G}}$ & 13n2427 & 13n3165 & 13n4049 & 13n4732 \\
12n657$^{\mathrm{G}}$ & 13n648$^{\mathrm{G}}$ & 13n1572$^{\mathrm{G}}$ & 13n2429$^{\mathrm{G}}$ & 13n3170$^{\mathrm{G}}$ & 13n4054$^{\mathrm{G}}$ & 13n4733 \\
12n662$^{\mathrm{G}}$ & 13n650$^{\mathrm{G}}$ & 13n1574 & 13n2430$^{\mathrm{G}}$ & 13n3171 & 13n4056 & 13n4736 \\
12n663$^{\mathrm{G}}$ & 13n653 & 13n1575 & 13n2431$^{\mathrm{G}}$ & 13n3175 & 13n4057$^{\mathrm{G}}$ & 13n4739$^{\mathrm{G}}$ \\
12n664$^{\mathrm{G}}$ & 13n655$^{\mathrm{G}}$ & 13n1578$^{\mathrm{G}}$ & 13n2432 & 13n3176 & 13n4061$^{\mathrm{G}}$ & 13n4743 \\
12n678$^{\mathrm{G}}$ & 13n657$^{\mathrm{G}}$ & 13n1583 & 13n2435$^{\mathrm{G}}$ & 13n3180$^{\mathrm{G}}$ & 13n4062$^{\mathrm{G}}$ & 13n4746$^{\mathrm{c}}$ \\
12n695$^{\mathrm{G}}$ & 13n658$^{\mathrm{G}}$ & 13n1587$^{\mathrm{G}}$ & 13n2436$^{\mathrm{G}}$ & 13n3184 & 13n4064 & 13n4750$^{\mathrm{G}}$ \\
12n702$^{\mathrm{G}}$ & 13n659$^{\mathrm{G}}$ & 13n1596$^{\mathrm{G}}$ & 13n2440$^{\mathrm{G}}$ & 13n3185$^{\mathrm{G}}$ & 13n4068$^{\mathrm{G}}$ & 13n4753$^{\mathrm{G}}$ \\
12n708$^{\mathrm{G}}$ & 13n660$^{\mathrm{G}}$ & 13n1601 & 13n2441$^{\mathrm{G}}$ & 13n3186$^{\mathrm{G}}$ & 13n4069$^{\mathrm{G}}$ & 13n4754$^{\mathrm{G}}$ \\
12n711$^{\mathrm{G}}$ & 13n662$^{\mathrm{G}}$ & 13n1603 & 13n2444$^{\mathrm{G}}$ & 13n3188 & 13n4073 & 13n4756$^{\mathrm{G}}$ \\
12n716$^{\mathrm{G}}$ & 13n663$^{\mathrm{G}}$ & 13n1604 & 13n2446$^{\mathrm{G}}$ & 13n3189$^{\mathrm{G}}$ & 13n4075$^{\mathrm{G}}$ & 13n4764$^{\mathrm{G}}$ \\
12n723$^{\mathrm{m}}$ & 13n667$^{\mathrm{G}}$ & 13n1608 & 13n2448$^{\mathrm{G}}$ & 13n3190$^{\mathrm{c}}$ & 13n4076$^{\mathrm{G}}$ & 13n4767 \\
12n728$^{\mathrm{G}}$ & 13n671$^{\mathrm{G}}$ & 13n1613 & 13n2451 & 13n3193$^{\mathrm{G}}$ & 13n4077$^{\mathrm{c}}$ & 13n4768 \\
12n730$^{\mathrm{G}}$ & 13n672$^{\mathrm{G}}$ & 13n1614 & 13n2455 & 13n3194$^{\mathrm{G}}$ & 13n4078$^{\mathrm{G}}$ & 13n4769$^{\mathrm{G}}$ \\
12n731$^{\mathrm{G}}$ & 13n673$^{\mathrm{G}}$ & 13n1615$^{\mathrm{G}}$ & 13n2457$^{\mathrm{G}}$ & 13n3195$^{\mathrm{G}}$ & 13n4080 & 13n4770 \\
12n733$^{\mathrm{G}}$ & 13n674$^{\mathrm{G}}$ & 13n1619$^{\mathrm{G}}$ & 13n2459$^{\mathrm{G}}$ & 13n3197 & 13n4081$^{\mathrm{G}}$ & 13n4772$^{\mathrm{G}}$ \\
12n735$^{\mathrm{G}}$ & 13n679 & 13n1621 & 13n2463 & 13n3198 & 13n4083$^{\mathrm{G}}$ & 13n4775$^{\mathrm{G}}$ \\
12n741$^{\mathrm{G}}$ & 13n681$^{\mathrm{G}}$ & 13n1623$^{\mathrm{c}}$ & 13n2464$^{\mathrm{G}}$ & 13n3202$^{\mathrm{G}}$ & 13n4086 & 13n4777$^{\mathrm{G}}$ \\
12n751$^{\mathrm{G}}$ & 13n685$^{\mathrm{G}}$ & 13n1624 & 13n2466$^{\mathrm{c}}$ & 13n3207$^{\mathrm{G}}$ & 13n4089 & 13n4778$^{\mathrm{G}}$ \\
12n768$^{\mathrm{G}}$ & 13n687 & 13n1626$^{\mathrm{G}}$ & 13n2467 & 13n3208 & 13n4090 & 13n4781$^{\mathrm{G}}$ \\
12n772$^{\mathrm{G}}$ & 13n692$^{\mathrm{G}}$ & 13n1627$^{\mathrm{G}}$ & 13n2469 & 13n3209$^{\mathrm{G}}$ & 13n4092 & 13n4784$^{\mathrm{G}}$ \\
12n774$^{\mathrm{G}}$ & 13n693$^{\mathrm{G}}$ & 13n1628 & 13n2473$^{\mathrm{G}}$ & 13n3211$^{\mathrm{G}}$ & 13n4093$^{\mathrm{G}}$ & 13n4786$^{\mathrm{G}}$ \\
12n775$^{\mathrm{G}}$ & 13n695$^{\mathrm{G}}$ & 13n1630$^{\mathrm{G}}$ & 13n2477$^{\mathrm{G}}$ & 13n3213$^{\mathrm{G}}$ & 13n4094 & 13n4787 \\
12n776$^{\mathrm{G}}$ & 13n696$^{\mathrm{G}}$ & 13n1631$^{\mathrm{G}}$ & 13n2478$^{\mathrm{G}}$ & 13n3215$^{\mathrm{G}}$ & 13n4095 & 13n4788 \\
12n777$^{\mathrm{G}}$ & 13n698$^{\mathrm{G}}$ & 13n1632$^{\mathrm{G}}$ & 13n2487 & 13n3216$^{\mathrm{G}}$ & 13n4097 & 13n4792$^{\mathrm{G}}$ \\
12n780$^{\mathrm{G}}$ & 13n699 & 13n1634 & 13n2488$^{\mathrm{G}}$ & 13n3220 & 13n4098 & 13n4796$^{\mathrm{G}}$ \\
12n782$^{\mathrm{G}}$ & 13n702$^{\mathrm{G}}$ & 13n1635$^{\mathrm{G}}$ & 13n2489$^{\mathrm{G}}$ & 13n3221$^{\mathrm{G}}$ & 13n4102$^{\mathrm{G}}$ & 13n4798$^{\mathrm{G}}$ \\
12n783$^{\mathrm{G}}$ & 13n703$^{\mathrm{G}}$ & 13n1638$^{\mathrm{G}}$ & 13n2493$^{\mathrm{G}}$ & 13n3222 & 13n4106 & 13n4799$^{\mathrm{G}}$ \\
12n784$^{\mathrm{G}}$ & 13n704$^{\mathrm{G}}$ & 13n1641$^{\mathrm{G}}$ & 13n2495$^{\mathrm{G}}$ & 13n3226 & 13n4107 & 13n4801$^{\mathrm{G}}$ \\
12n788$^{\mathrm{G}}$ & 13n705$^{\mathrm{G}}$ & 13n1644$^{\mathrm{G}}$ & 13n2497$^{\mathrm{G}}$ & 13n3228 & 13n4110$^{\mathrm{G}}$ & 13n4803 \\
12n790$^{\mathrm{G}}$ & 13n706$^{\mathrm{G}}$ & 13n1650 & 13n2499$^{\mathrm{G}}$ & 13n3229$^{\mathrm{G}}$ & 13n4112$^{\mathrm{c}}$ & 13n4809 \\
12n792$^{\mathrm{G}}$ & 13n708$^{\mathrm{G}}$ & 13n1657 & 13n2500 & 13n3230$^{\mathrm{G}}$ & 13n4113$^{\mathrm{c}}$ & 13n4813 \\
12n794$^{\mathrm{G}}$ & 13n710$^{\mathrm{G}}$ & 13n1664 & 13n2503$^{\mathrm{G}}$ & 13n3232 & 13n4118$^{\mathrm{G}}$ & 13n4814$^{\mathrm{G}}$ \\
12n795$^{\mathrm{G}}$ & 13n712 & 13n1665 & 13n2504$^{\mathrm{t}}$ & 13n3235$^{\mathrm{G}}$ & 13n4119$^{\mathrm{G}}$ & 13n4815$^{\mathrm{G}}$ \\
12n799$^{\mathrm{G}}$ & 13n713$^{\mathrm{G}}$ & 13n1666 & 13n2505$^{\mathrm{G}}$ & 13n3236$^{\mathrm{G}}$ & 13n4120 & 13n4816$^{\mathrm{G}}$ \\
12n800$^{\mathrm{G}}$ & 13n717$^{\mathrm{G}}$ & 13n1667 & 13n2506$^{\mathrm{G}}$ & 13n3237 & 13n4121 & 13n4824$^{\mathrm{G}}$ \\
12n802$^{\mathrm{G}}$ & 13n718$^{\mathrm{G}}$ & 13n1670 & 13n2507 & 13n3240 & 13n4122 & 13n4827 \\
12n819$^{\mathrm{G}}$ & 13n722$^{\mathrm{G}}$ & 13n1673$^{\mathrm{G}}$ & 13n2509 & 13n3241 & 13n4126 & 13n4829$^{\mathrm{G}}$ \\
12n821$^{\mathrm{G}}$ & 13n723$^{\mathrm{G}}$ & 13n1678 & 13n2510$^{\mathrm{G}}$ & 13n3243$^{\mathrm{G}}$ & 13n4129$^{\mathrm{G}}$ & 13n4830$^{\mathrm{G}}$ \\
12n822$^{\mathrm{G}}$ & 13n726$^{\mathrm{G}}$ & 13n1679 & 13n2511$^{\mathrm{G}}$ & 13n3246$^{\mathrm{G}}$ & 13n4131$^{\mathrm{G}}$ & 13n4833$^{\mathrm{G}}$ \\
12n824$^{\mathrm{G}}$ & 13n728$^{\mathrm{G}}$ & 13n1680$^{\mathrm{G}}$ & 13n2512$^{\mathrm{G}}$ & 13n3248 & 13n4134$^{\mathrm{G}}$ & 13n4835 \\
12n825$^{\mathrm{G}}$ & 13n736$^{\mathrm{G}}$ & 13n1686$^{\mathrm{G}}$ & 13n2515$^{\mathrm{G}}$ & 13n3249$^{\mathrm{G}}$ & 13n4136$^{\mathrm{G}}$ & 13n4839$^{\mathrm{G}}$ \\
12n826$^{\mathrm{G}}$ & 13n738$^{\mathrm{G}}$ & 13n1688 & 13n2516 & 13n3252$^{\mathrm{G}}$ & 13n4137$^{\mathrm{G}}$ & 13n4840$^{\mathrm{G}}$ \\
12n829$^{\mathrm{G}}$ & 13n740$^{\mathrm{G}}$ & 13n1689$^{\mathrm{G}}$ & 13n2517$^{\mathrm{G}}$ & 13n3255$^{\mathrm{G}}$ & 13n4139$^{\mathrm{G}}$ & 13n4842$^{\mathrm{G}}$ \\
12n841$^{\mathrm{G}}$ & 13n742$^{\mathrm{G}}$ & 13n1691 & 13n2519 & 13n3258 & 13n4140 & 13n4843$^{\mathrm{G}}$ \\
12n842$^{\mathrm{G}}$ & 13n745$^{\mathrm{G}}$ & 13n1692$^{\mathrm{G}}$ & 13n2520$^{\mathrm{G}}$ & 13n3259$^{\mathrm{G}}$ & 13n4142 & 13n4846$^{\mathrm{c}}$ \\
12n853$^{\mathrm{G}}$ & 13n747 & 13n1693$^{\mathrm{G}}$ & 13n2523$^{\mathrm{G}}$ & 13n3262$^{\mathrm{G}}$ & 13n4144$^{\mathrm{G}}$ & 13n4847$^{\mathrm{G}}$ \\
12n857$^{\mathrm{G}}$ & 13n748$^{\mathrm{G}}$ & 13n1695$^{\mathrm{G}}$ & 13n2524 & 13n3264 & 13n4145$^{\mathrm{G}}$ & 13n4848$^{\mathrm{G}}$ \\
12n861$^{\mathrm{G}}$ & 13n752$^{\mathrm{G}}$ & 13n1698$^{\mathrm{G}}$ & 13n2525$^{\mathrm{G}}$ & 13n3266$^{\mathrm{G}}$ & 13n4146$^{\mathrm{G}}$ & 13n4849 \\
12n864$^{\mathrm{G}}$ & 13n756$^{\mathrm{c}}$ & 13n1699 & 13n2530 & 13n3269$^{\mathrm{G}}$ & 13n4147 & 13n4850 \\
12n865$^{\mathrm{G}}$ & 13n757 & 13n1702$^{\mathrm{G}}$ & 13n2531$^{\mathrm{G}}$ & 13n3270 & 13n4148$^{\mathrm{G}}$ & 13n4851 \\
12n866$^{\mathrm{G}}$ & 13n759$^{\mathrm{G}}$ & 13n1704$^{\mathrm{G}}$ & 13n2534 & 13n3272 & 13n4149$^{\mathrm{G}}$ & 13n4852$^{\mathrm{G}}$ \\
12n872$^{\mathrm{G}}$ & 13n760 & 13n1705$^{\mathrm{G}}$ & 13n2536 & 13n3277$^{\mathrm{G}}$ & 13n4150$^{\mathrm{G}}$ & 13n4853 \\
12n875$^{\mathrm{G}}$ & 13n762$^{\mathrm{G}}$ & 13n1710$^{\mathrm{G}}$ & 13n2537 & 13n3278$^{\mathrm{G}}$ & 13n4151$^{\mathrm{G}}$ & 13n4854$^{\mathrm{c}}$ \\
12n880$^{\mathrm{G}}$ & 13n764$^{\mathrm{G}}$ & 13n1711$^{\mathrm{G}}$ & 13n2541$^{\mathrm{G}}$ & 13n3279 & 13n4152 & 13n4855$^{\mathrm{G}}$ \\
12n885$^{\mathrm{G}}$ & 13n767$^{\mathrm{G}}$ & 13n1721$^{\mathrm{G}}$ & 13n2549 & 13n3281 & 13n4155$^{\mathrm{G}}$ & 13n4856$^{\mathrm{G}}$ \\
13n2$^{\mathrm{G}}$ & 13n775 & 13n1722$^{\mathrm{G}}$ & 13n2551 & 13n3283 & 13n4161 & 13n4857$^{\mathrm{G}}$ \\
13n3$^{\mathrm{G}}$ & 13n777$^{\mathrm{G}}$ & 13n1723$^{\mathrm{G}}$ & 13n2552 & 13n3284$^{\mathrm{G}}$ & 13n4163 & 13n4858$^{\mathrm{G}}$ \\
13n6 & 13n779$^{\mathrm{G}}$ & 13n1724$^{\mathrm{G}}$ & 13n2553 & 13n3286$^{\mathrm{G}}$ & 13n4164 & 13n4860 \\
13n8$^{\mathrm{G}}$ & 13n781$^{\mathrm{G}}$ & 13n1726$^{\mathrm{G}}$ & 13n2554 & 13n3288$^{\mathrm{G}}$ & 13n4165$^{\mathrm{G}}$ & 13n4861 \\
13n9$^{\mathrm{G}}$ & 13n784$^{\mathrm{G}}$ & 13n1728$^{\mathrm{G}}$ & 13n2558 & 13n3289 & 13n4169 & 13n4862 \\
13n10$^{\mathrm{G}}$ & 13n785$^{\mathrm{G}}$ & 13n1731$^{\mathrm{G}}$ & 13n2564$^{\mathrm{G}}$ & 13n3290$^{\mathrm{G}}$ & 13n4171$^{\mathrm{G}}$ & 13n4864 \\
13n14$^{\mathrm{G}}$ & 13n794$^{\mathrm{G}}$ & 13n1733 & 13n2567$^{\mathrm{G}}$ & 13n3295$^{\mathrm{G}}$ & 13n4172$^{\mathrm{G}}$ & 13n4869 \\
13n17 & 13n801 & 13n1734 & 13n2568 & 13n3299 & 13n4173 & 13n4871 \\
13n20$^{\mathrm{G}}$ & 13n806$^{\mathrm{G}}$ & 13n1737$^{\mathrm{G}}$ & 13n2570 & 13n3300$^{\mathrm{G}}$ & 13n4175 & 13n4872 \\
13n21$^{\mathrm{G}}$ & 13n810$^{\mathrm{G}}$ & 13n1738$^{\mathrm{G}}$ & 13n2572 & 13n3305 & 13n4177$^{\mathrm{G}}$ & 13n4873$^{\mathrm{G}}$ \\
13n22$^{\mathrm{G}}$ & 13n811$^{\mathrm{G}}$ & 13n1751$^{\mathrm{G}}$ & 13n2573 & 13n3306$^{\mathrm{G}}$ & 13n4182$^{\mathrm{G}}$ & 13n4874 \\
13n23$^{\mathrm{G}}$ & 13n812$^{\mathrm{G}}$ & 13n1752$^{\mathrm{G}}$ & 13n2574 & 13n3308$^{\mathrm{G}}$ & 13n4184$^{\mathrm{G}}$ & 13n4878$^{\mathrm{G}}$ \\
13n25 & 13n813$^{\mathrm{G}}$ & 13n1759 & 13n2575$^{\mathrm{G}}$ & 13n3309$^{\mathrm{G}}$ & 13n4186 & 13n4880$^{\mathrm{G}}$ \\
13n27$^{\mathrm{G}}$ & 13n816 & 13n1763$^{\mathrm{G}}$ & 13n2579$^{\mathrm{G}}$ & 13n3313$^{\mathrm{G}}$ & 13n4188$^{\mathrm{G}}$ & 13n4882$^{\mathrm{G}}$ \\
13n30$^{\mathrm{w}}$ & 13n817$^{\mathrm{G}}$ & 13n1764$^{\mathrm{G}}$ & 13n2582 & 13n3316$^{\mathrm{G}}$ & 13n4189 & 13n4886$^{\mathrm{G}}$ \\
13n35$^{\mathrm{G}}$ & 13n820 & 13n1766$^{\mathrm{G}}$ & 13n2583 & 13n3318 & 13n4191$^{\mathrm{G}}$ & 13n4889$^{\mathrm{G}}$ \\
13n38$^{\mathrm{G}}$ & 13n822$^{\mathrm{G}}$ & 13n1771 & 13n2585$^{\mathrm{G}}$ & 13n3320$^{\mathrm{G}}$ & 13n4193$^{\mathrm{G}}$ & 13n4893 \\
13n39$^{\mathrm{G}}$ & 13n824 & 13n1783$^{\mathrm{G}}$ & 13n2588$^{\mathrm{G}}$ & 13n3321$^{\mathrm{G}}$ & 13n4194$^{\mathrm{G}}$ & 13n4895$^{\mathrm{G}}$ \\
13n45$^{\mathrm{w}}$ & 13n827 & 13n1787$^{\mathrm{G}}$ & 13n2591$^{\mathrm{G}}$ & 13n3322 & 13n4198 & 13n4899 \\
13n48$^{\mathrm{G}}$ & 13n828 & 13n1791 & 13n2593$^{\mathrm{G}}$ & 13n3323$^{\mathrm{G}}$ & 13n4201 & 13n4901$^{\mathrm{G}}$ \\
13n50 & 13n829$^{\mathrm{G}}$ & 13n1793 & 13n2596$^{\mathrm{G}}$ & 13n3324 & 13n4204 & 13n4912$^{\mathrm{G}}$ \\
13n51 & 13n830$^{\mathrm{G}}$ & 13n1796 & 13n2598 & 13n3327$^{\mathrm{G}}$ & 13n4208 & 13n4914 \\
13n52$^{\mathrm{G}}$ & 13n836$^{\mathrm{G}}$ & 13n1797 & 13n2602$^{\mathrm{G}}$ & 13n3328$^{\mathrm{G}}$ & 13n4209 & 13n4916 \\
13n53$^{\mathrm{G}}$ & 13n843$^{\mathrm{G}}$ & 13n1798$^{\mathrm{G}}$ & 13n2603 & 13n3330 & 13n4210$^{\mathrm{G}}$ & 13n4919 \\
13n57$^{\mathrm{G}}$ & 13n844$^{\mathrm{G}}$ & 13n1800$^{\mathrm{G}}$ & 13n2614$^{\mathrm{c}}$ & 13n3333 & 13n4214$^{\mathrm{G}}$ & 13n4920$^{\mathrm{G}}$ \\
13n63$^{\mathrm{G}}$ & 13n846$^{\mathrm{G}}$ & 13n1803 & 13n2616$^{\mathrm{G}}$ & 13n3334$^{\mathrm{c}}$ & 13n4216$^{\mathrm{G}}$ & 13n4922$^{\mathrm{G}}$ \\
13n64$^{\mathrm{G}}$ & 13n847$^{\mathrm{G}}$ & 13n1808 & 13n2619 & 13n3336$^{\mathrm{G}}$ & 13n4217$^{\mathrm{G}}$ & 13n4924 \\
13n66$^{\mathrm{G}}$ & 13n854$^{\mathrm{G}}$ & 13n1809$^{\mathrm{G}}$ & 13n2620 & 13n3338 & 13n4218 & 13n4925$^{\mathrm{G}}$ \\
13n67$^{\mathrm{G}}$ & 13n855$^{\mathrm{G}}$ & 13n1813$^{\mathrm{G}}$ & 13n2623$^{\mathrm{G}}$ & 13n3340$^{\mathrm{G}}$ & 13n4221 & 13n4926 \\
13n69$^{\mathrm{G}}$ & 13n859 & 13n1817$^{\mathrm{G}}$ & 13n2624$^{\mathrm{c}}$ & 13n3342$^{\mathrm{G}}$ & 13n4224$^{\mathrm{G}}$ & 13n4933 \\
13n70$^{\mathrm{G}}$ & 13n864$^{\mathrm{G}}$ & 13n1819 & 13n2625 & 13n3344$^{\mathrm{G}}$ & 13n4225 & 13n4935$^{\mathrm{G}}$ \\
13n72$^{\mathrm{G}}$ & 13n867$^{\mathrm{G}}$ & 13n1820$^{\mathrm{G}}$ & 13n2626$^{\mathrm{G}}$ & 13n3352 & 13n4227 & 13n4936 \\
13n73$^{\mathrm{G}}$ & 13n873$^{\mathrm{G}}$ & 13n1823$^{\mathrm{G}}$ & 13n2630 & 13n3353 & 13n4232 & 13n4937$^{\mathrm{c}}$ \\
13n75 & 13n875$^{\mathrm{G}}$ & 13n1824$^{\mathrm{G}}$ & 13n2631$^{\mathrm{G}}$ & 13n3357 & 13n4236$^{\mathrm{G}}$ & 13n4938 \\
13n78$^{\mathrm{G}}$ & 13n881 & 13n1828$^{\mathrm{G}}$ & 13n2633$^{\mathrm{G}}$ & 13n3358 & 13n4237$^{\mathrm{G,w}}$ & 13n4942$^{\mathrm{G}}$ \\
13n80$^{\mathrm{w}}$ & 13n887$^{\mathrm{G}}$ & 13n1829$^{\mathrm{G}}$ & 13n2634$^{\mathrm{G}}$ & 13n3360$^{\mathrm{G}}$ & 13n4240$^{\mathrm{G}}$ & 13n4948 \\
13n85 & 13n893$^{\mathrm{G}}$ & 13n1831$^{\mathrm{G}}$ & 13n2637$^{\mathrm{G}}$ & 13n3361 & 13n4247$^{\mathrm{G}}$ & 13n4953$^{\mathrm{G}}$ \\
13n86 & 13n895 & 13n1834$^{\mathrm{G}}$ & 13n2639$^{\mathrm{G,w}}$ & 13n3362$^{\mathrm{G}}$ & 13n4252$^{\mathrm{G}}$ & 13n4955 \\
13n87$^{\mathrm{G}}$ & 13n896$^{\mathrm{G}}$ & 13n1835$^{\mathrm{G}}$ & 13n2641 & 13n3368 & 13n4255$^{\mathrm{G}}$ & 13n4956$^{\mathrm{G}}$ \\
13n88$^{\mathrm{G}}$ & 13n899$^{\mathrm{G}}$ & 13n1836 & 13n2642$^{\mathrm{G}}$ & 13n3370$^{\mathrm{G}}$ & 13n4257 & 13n4957 \\
13n95$^{\mathrm{G}}$ & 13n903$^{\mathrm{G}}$ & 13n1837 & 13n2646$^{\mathrm{G}}$ & 13n3374 & 13n4258 & 13n4958$^{\mathrm{c}}$ \\
13n96 & 13n904 & 13n1839$^{\mathrm{G}}$ & 13n2651 & 13n3376$^{\mathrm{G}}$ & 13n4259$^{\mathrm{G}}$ & 13n4959$^{\mathrm{G}}$ \\
13n97$^{\mathrm{G}}$ & 13n908$^{\mathrm{G}}$ & 13n1841 & 13n2652$^{\mathrm{G}}$ & 13n3378 & 13n4260$^{\mathrm{G}}$ & 13n4962 \\
13n100$^{\mathrm{G}}$ & 13n910$^{\mathrm{G}}$ & 13n1842$^{\mathrm{G}}$ & 13n2654 & 13n3384$^{\mathrm{G}}$ & 13n4261$^{\mathrm{G}}$ & 13n4963$^{\mathrm{G}}$ \\
13n101$^{\mathrm{G}}$ & 13n913$^{\mathrm{G}}$ & 13n1843$^{\mathrm{G}}$ & 13n2656$^{\mathrm{c}}$ & 13n3388$^{\mathrm{G}}$ & 13n4263 & 13n4966 \\
13n102$^{\mathrm{G}}$ & 13n914$^{\mathrm{G}}$ & 13n1845$^{\mathrm{G}}$ & 13n2658$^{\mathrm{G}}$ & 13n3389 & 13n4264 & 13n4976$^{\mathrm{G}}$ \\
13n104$^{\mathrm{G}}$ & 13n918$^{\mathrm{G}}$ & 13n1847$^{\mathrm{G}}$ & 13n2659$^{\mathrm{G}}$ & 13n3392$^{\mathrm{G}}$ & 13n4267$^{\mathrm{G}}$ & 13n4978 \\
13n110 & 13n921$^{\mathrm{G}}$ & 13n1852 & 13n2661 & 13n3395$^{\mathrm{c}}$ & 13n4269$^{\mathrm{G}}$ & 13n4979$^{\mathrm{G}}$ \\
13n112 & 13n926$^{\mathrm{G}}$ & 13n1853$^{\mathrm{G}}$ & 13n2662$^{\mathrm{G}}$ & 13n3396$^{\mathrm{c}}$ & 13n4270 & 13n4980$^{\mathrm{G}}$ \\
13n113 & 13n928$^{\mathrm{G}}$ & 13n1854$^{\mathrm{G}}$ & 13n2663$^{\mathrm{G}}$ & 13n3397$^{\mathrm{G}}$ & 13n4272 & 13n4982 \\
13n114$^{\mathrm{G}}$ & 13n934$^{\mathrm{G}}$ & 13n1856$^{\mathrm{G}}$ & 13n2664 & 13n3403$^{\mathrm{G}}$ & 13n4273$^{\mathrm{G}}$ & 13n4986$^{\mathrm{G}}$ \\
13n115 & 13n940 & 13n1857 & 13n2667 & 13n3407$^{\mathrm{G}}$ & 13n4274$^{\mathrm{c}}$ & 13n4987$^{\mathrm{G}}$ \\
13n116 & 13n944$^{\mathrm{G}}$ & 13n1858$^{\mathrm{G}}$ & 13n2668$^{\mathrm{c}}$ & 13n3408$^{\mathrm{G}}$ & 13n4275$^{\mathrm{c}}$ & 13n4995$^{\mathrm{G}}$ \\
13n117 & 13n946$^{\mathrm{G}}$ & 13n1859 & 13n2675 & 13n3411$^{\mathrm{G}}$ & 13n4278$^{\mathrm{G}}$ & 13n4999 \\
13n118$^{\mathrm{G}}$ & 13n950$^{\mathrm{G}}$ & 13n1866$^{\mathrm{G}}$ & 13n2676$^{\mathrm{G}}$ & 13n3412$^{\mathrm{G}}$ & 13n4279$^{\mathrm{G}}$ & 13n5001 \\
13n121$^{\mathrm{G}}$ & 13n953$^{\mathrm{G}}$ & 13n1869$^{\mathrm{G}}$ & 13n2680$^{\mathrm{G}}$ & 13n3414$^{\mathrm{G}}$ & 13n4282$^{\mathrm{G}}$ & 13n5002 \\
13n122$^{\mathrm{G}}$ & 13n956$^{\mathrm{G}}$ & 13n1877$^{\mathrm{G}}$ & 13n2682 & 13n3417$^{\mathrm{G}}$ & 13n4284$^{\mathrm{G}}$ & 13n5003 \\
13n123$^{\mathrm{G}}$ & 13n957$^{\mathrm{G}}$ & 13n1880 & 13n2684 & 13n3419 & 13n4285$^{\mathrm{G}}$ & 13n5004 \\
13n124$^{\mathrm{G}}$ & 13n960$^{\mathrm{G}}$ & 13n1882$^{\mathrm{G}}$ & 13n2686$^{\mathrm{c}}$ & 13n3420$^{\mathrm{G}}$ & 13n4286$^{\mathrm{c}}$ & 13n5005$^{\mathrm{G}}$ \\
13n129$^{\mathrm{G}}$ & 13n969$^{\mathrm{G}}$ & 13n1883$^{\mathrm{G}}$ & 13n2687 & 13n3423$^{\mathrm{c}}$ & 13n4288$^{\mathrm{c}}$ & 13n5006 \\
13n130$^{\mathrm{G}}$ & 13n970$^{\mathrm{G}}$ & 13n1884$^{\mathrm{G}}$ & 13n2690 & 13n3425$^{\mathrm{G}}$ & 13n4289 & 13n5012$^{\mathrm{G}}$ \\
13n134$^{\mathrm{G}}$ & 13n990$^{\mathrm{G}}$ & 13n1887$^{\mathrm{G}}$ & 13n2694 & 13n3427$^{\mathrm{G}}$ & 13n4290$^{\mathrm{G}}$ & 13n5021 \\
13n137$^{\mathrm{G}}$ & 13n997$^{\mathrm{G}}$ & 13n1889 & 13n2696 & 13n3433 & 13n4292$^{\mathrm{G}}$ & 13n5022 \\
13n141 & 13n998$^{\mathrm{G}}$ & 13n1890$^{\mathrm{G}}$ & 13n2700$^{\mathrm{G}}$ & 13n3434$^{\mathrm{G}}$ & 13n4293$^{\mathrm{G}}$ & 13n5026$^{\mathrm{G}}$ \\
13n143 & 13n1000$^{\mathrm{G}}$ & 13n1892 & 13n2701$^{\mathrm{G}}$ & 13n3435 & 13n4294$^{\mathrm{G}}$ & 13n5028 \\
13n144 & 13n1005$^{\mathrm{G}}$ & 13n1893$^{\mathrm{G}}$ & 13n2702$^{\mathrm{G}}$ & 13n3436$^{\mathrm{G}}$ & 13n4295$^{\mathrm{G}}$ & 13n5029$^{\mathrm{G}}$ \\
13n145 & 13n1012$^{\mathrm{G}}$ & 13n1895$^{\mathrm{G}}$ & 13n2703 & 13n3437$^{\mathrm{G}}$ & 13n4297 & 13n5030 \\
13n147$^{\mathrm{G}}$ & 13n1017$^{\mathrm{G}}$ & 13n1896 & 13n2707 & 13n3439$^{\mathrm{G}}$ & 13n4299 & 13n5031 \\
13n149 & 13n1023 & 13n1898$^{\mathrm{G}}$ & 13n2712$^{\mathrm{G}}$ & 13n3440$^{\mathrm{G}}$ & 13n4300 & 13n5033 \\
13n153 & 13n1025$^{\mathrm{G}}$ & 13n1900$^{\mathrm{G}}$ & 13n2713$^{\mathrm{G}}$ & 13n3444$^{\mathrm{G}}$ & 13n4301$^{\mathrm{G}}$ & 13n5035$^{\mathrm{G}}$ \\
13n157$^{\mathrm{G}}$ & 13n1026 & 13n1908 & 13n2715 & 13n3447$^{\mathrm{G}}$ & 13n4302 & 13n5038 \\
13n160 & 13n1027 & 13n1910$^{\mathrm{G}}$ & 13n2716$^{\mathrm{G}}$ & 13n3448$^{\mathrm{c}}$ & 13n4306 & 13n5041$^{\mathrm{G}}$ \\
13n161 & 13n1028$^{\mathrm{G}}$ & 13n1913$^{\mathrm{G}}$ & 13n2718 & 13n3450$^{\mathrm{G}}$ & 13n4307$^{\mathrm{G}}$ & 13n5042$^{\mathrm{G}}$ \\
13n163 & 13n1031 & 13n1914$^{\mathrm{G}}$ & 13n2720$^{\mathrm{G}}$ & 13n3452$^{\mathrm{G}}$ & 13n4311 & 13n5044$^{\mathrm{c}}$ \\
13n164 & 13n1032$^{\mathrm{G}}$ & 13n1919$^{\mathrm{c}}$ & 13n2722$^{\mathrm{G}}$ & 13n3455$^{\mathrm{G}}$ & 13n4312$^{\mathrm{G}}$ & 13n5047 \\
13n165 & 13n1034 & 13n1922$^{\mathrm{G}}$ & 13n2723 & 13n3456$^{\mathrm{G}}$ & 13n4315$^{\mathrm{G}}$ & 13n5053$^{\mathrm{G}}$ \\
13n166 & 13n1035$^{\mathrm{G}}$ & 13n1923 & 13n2727$^{\mathrm{G}}$ & 13n3457 & 13n4318$^{\mathrm{G}}$ & 13n5058 \\
13n167 & 13n1036$^{\mathrm{G}}$ & 13n1924 & 13n2730$^{\mathrm{G}}$ & 13n3458$^{\mathrm{G}}$ & 13n4320$^{\mathrm{G}}$ & 13n5060$^{\mathrm{G}}$ \\
13n168 & 13n1039$^{\mathrm{G}}$ & 13n1931 & 13n2739$^{\mathrm{G}}$ & 13n3459 & 13n4323$^{\mathrm{G}}$ & 13n5066$^{\mathrm{G}}$ \\
13n171$^{\mathrm{G}}$ & 13n1040 & 13n1932$^{\mathrm{G}}$ & 13n2741$^{\mathrm{G}}$ & 13n3460 & 13n4325$^{\mathrm{G}}$ & 13n5067$^{\mathrm{G}}$ \\
13n173 & 13n1042 & 13n1941 & 13n2742$^{\mathrm{c}}$ & 13n3461 & 13n4326$^{\mathrm{G}}$ & 13n5068$^{\mathrm{G}}$ \\
13n174 & 13n1044$^{\mathrm{G}}$ & 13n1942$^{\mathrm{G}}$ & 13n2744$^{\mathrm{G}}$ & 13n3465$^{\mathrm{G}}$ & 13n4328$^{\mathrm{G}}$ & 13n5070 \\
13n176 & 13n1046$^{\mathrm{G}}$ & 13n1951$^{\mathrm{G}}$ & 13n2746$^{\mathrm{G}}$ & 13n3471$^{\mathrm{G}}$ & 13n4330$^{\mathrm{G}}$ & 13n5073$^{\mathrm{G}}$ \\
13n177 & 13n1053 & 13n1956 & 13n2747$^{\mathrm{c}}$ & 13n3473$^{\mathrm{G}}$ & 13n4331$^{\mathrm{G}}$ & 13n5078$^{\mathrm{G}}$ \\
13n178 & 13n1067 & 13n1960$^{\mathrm{G}}$ & 13n2748$^{\mathrm{G}}$ & 13n3474 & 13n4334$^{\mathrm{G}}$ & 13n5079$^{\mathrm{G}}$ \\
13n179 & 13n1068 & 13n1971$^{\mathrm{c}}$ & 13n2749 & 13n3475$^{\mathrm{G}}$ & 13n4335$^{\mathrm{G}}$ & 13n5081 \\
13n180 & 13n1070$^{\mathrm{G}}$ & 13n1972 & 13n2751 & 13n3476 & 13n4336$^{\mathrm{G}}$ & 13n5084 \\
13n182 & 13n1071 & 13n1979 & 13n2753$^{\mathrm{G}}$ & 13n3481$^{\mathrm{G}}$ & 13n4339$^{\mathrm{G}}$ & 13n5086 \\
13n186$^{\mathrm{G}}$ & 13n1085$^{\mathrm{G}}$ & 13n1983$^{\mathrm{G}}$ & 13n2759$^{\mathrm{G}}$ & 13n3482$^{\mathrm{G}}$ & 13n4341 & 13n5088 \\
13n189$^{\mathrm{G}}$ & 13n1092$^{\mathrm{G}}$ & 13n1985$^{\mathrm{G}}$ & 13n2760$^{\mathrm{G}}$ & 13n3483$^{\mathrm{G}}$ & 13n4342 & 13n5090 \\
13n195 & 13n1094$^{\mathrm{G}}$ & 13n1989 & 13n2762 & 13n3486 & 13n4346$^{\mathrm{G}}$ & 13n5091 \\
13n203 & 13n1098$^{\mathrm{G}}$ & 13n1991$^{\mathrm{G}}$ & 13n2763$^{\mathrm{G}}$ & 13n3487$^{\mathrm{G}}$ & 13n4347$^{\mathrm{G}}$ & 13n5092 \\
13n209 & 13n1101$^{\mathrm{G}}$ & 13n1994$^{\mathrm{G}}$ & 13n2764$^{\mathrm{G}}$ & 13n3488 & 13n4348$^{\mathrm{G}}$ & 13n5099$^{\mathrm{G}}$ \\
13n210 & 13n1102 & 13n2000$^{\mathrm{G}}$ & 13n2765 & 13n3496 & 13n4369$^{\mathrm{G}}$ & 13n5100 \\
13n215 & 13n1108$^{\mathrm{G}}$ & 13n2005$^{\mathrm{G}}$ & 13n2768$^{\mathrm{G}}$ & 13n3497 & 13n4370$^{\mathrm{G}}$ & 13n5109$^{\mathrm{G}}$ \\
13n216$^{\mathrm{G}}$ & 13n1109$^{\mathrm{G}}$ & 13n2006$^{\mathrm{G}}$ & 13n2771$^{\mathrm{G}}$ & 13n3498$^{\mathrm{G}}$ & 13n4372$^{\mathrm{G}}$ & 13n5110 \\*
13n219$^{\mathrm{G}}$ & 13n1114 & 13n2012 & 13n2772$^{\mathrm{G}}$ & 13n3506 & 13n4376$^{\mathrm{G}}$ &  \\*
13n220 & 13n1116$^{\mathrm{G}}$ & 13n2014 & 13n2773$^{\mathrm{G}}$ & 13n3509$^{\mathrm{G}}$ & 13n4378 &  \\*
\multicolumn{7}{@{}l@{}}{\usebox{\appendixLegendBox}} \\
\end{longtable}}

\section{Improved ranges}\label{app:improve}
For the following \nImproved{} knots, our calculations improve the recorded range without
determining the exact value. Each block lists the knot, its comparison range, and the
range proved here; entries are ordered down the blocks,
and the superscripts are explained at the foot of the table.
The starting ranges are those of the archived research table specified in
Section~\ref{sec:lower}; \nLatestRangeAgreements{} of the resulting ranges are
already recorded in the \KI{} snapshot of 9 September 2026.

{\footnotesize\setlength{\tabcolsep}{6pt}
\setlength{\ctKnot}{1.9cm}%
\setlength{\ctRange}{1.55cm}%
\setlength{\ctWidth}{\dimexpr2\ctKnot+4\ctRange+8\tabcolsep+2em+\arrayrulewidth\relax}%
\setlength{\LTleft}{\fill}
\setlength{\LTright}{\fill}
\sbox{\appendixLegendBox}{%
\begin{minipage}{\ctWidth}
\hrule height\heavyrulewidth\vspace{3pt}
\raggedright\setlength{\parindent}{0pt}\setlength{\parskip}{0pt}
\strut Unmarked: the obstruction from the linking pairing.\strut\par
\strut $^{\mathrm{t}}$ torsion orders.\strut\par
\strut $^{\mathrm{g}}$ generator bound of $H_1(\dbc(K))$.\strut\par
\strut $^{\mathrm{c}}$ generator bound of a higher cyclic cover.\strut\par
\strut $^{\mathrm{o}}$ correction-term obstruction to $u=2$.\strut\par
\strut $^{\mathrm{h}}$ correction-term obstruction to $u=3$.\strut\par
\strut $^{\mathrm{k}}$ no unknotting crossing in the alternating diagram (McCoy's theorem, Section~\ref{sec:lower}).\strut\par
\strut $^{\mathrm{m}}$ obstruction for Montesinos knots (Section~\ref{sec:lower}).\strut\par
\strut $^{\mathrm{G}}$ Greene's model and the surgery or correction-term obstruction (Section~\ref{sec:sweep}).\strut\par
\end{minipage}}
\begin{longtable}{@{}>{\raggedright\arraybackslash}p{\ctKnot}%
>{\centering\arraybackslash}p{\ctRange}%
>{\centering\arraybackslash}p{\ctRange}@{\quad}|@{\quad}>{\raggedright\arraybackslash}p{\ctKnot}%
>{\centering\arraybackslash}p{\ctRange}%
>{\centering\arraybackslash}p{\ctRange}@{}}
\hline
Knot & Baseline & Result & Knot & Baseline & Result \\ \hline
\endfirsthead
\multicolumn{6}{@{}p{\ctWidth}@{}}{\emph{Improved ranges, continued from the previous page.}} \\[2pt]
\hline
Knot & Baseline & Result & Knot & Baseline & Result \\ \hline
\endhead
\bottomrule
\multicolumn{6}{@{}r@{}}{\emph{continued on the next page}} \\
\endfoot
\endlastfoot
11a354$^{\mathrm{o}}$ & $[2,4]$ & $[3,4]$ & 13n1425$^{\mathrm{G}}$ & $[1,3]$ & $[2,3]$ \\
12a156$^{\mathrm{o}}$ & $[2,4]$ & $[3,4]$ & 13n1429 & $[1,3]$ & $[2,3]$ \\
12a392$^{\mathrm{o}}$ & $[2,4]$ & $[3,4]$ & 13n1454$^{\mathrm{G}}$ & $[1,3]$ & $[2,3]$ \\
12a1037$^{\mathrm{o}}$ & $[2,4]$ & $[3,4]$ & 13n1475 & $[1,3]$ & $[2,3]$ \\
12a1097$^{\mathrm{o}}$ & $[2,4]$ & $[3,4]$ & 13n1551$^{\mathrm{G}}$ & $[1,3]$ & $[2,3]$ \\
12a1113$^{\mathrm{o}}$ & $[2,4]$ & $[3,4]$ & 13n1582$^{\mathrm{G}}$ & $[1,3]$ & $[2,3]$ \\
12n75$^{\mathrm{G}}$ & $[1,3]$ & $[2,3]$ & 13n1586$^{\mathrm{G}}$ & $[1,3]$ & $[2,3]$ \\
12n84$^{\mathrm{G}}$ & $[1,3]$ & $[2,3]$ & 13n1651 & $[1,4]$ & $[2,4]$ \\
12n92$^{\mathrm{G}}$ & $[1,3]$ & $[2,3]$ & 13n1653 & $[1,3]$ & $[2,3]$ \\
12n101$^{\mathrm{G}}$ & $[1,3]$ & $[2,3]$ & 13n1656$^{\mathrm{G}}$ & $[1,3]$ & $[2,3]$ \\
12n137$^{\mathrm{G}}$ & $[1,3]$ & $[2,3]$ & 13n1659$^{\mathrm{G}}$ & $[1,3]$ & $[2,3]$ \\
12n140$^{\mathrm{G}}$ & $[1,3]$ & $[2,3]$ & 13n1672 & $[1,3]$ & $[2,3]$ \\
12n167$^{\mathrm{m}}$ & $[1,3]$ & $[2,3]$ & 13n1683 & $[1,3]$ & $[2,3]$ \\
12n216$^{\mathrm{G}}$ & $[1,3]$ & $[2,3]$ & 13n1744$^{\mathrm{G}}$ & $[1,3]$ & $[2,3]$ \\
12n291$^{\mathrm{G}}$ & $[1,3]$ & $[2,3]$ & 13n1749$^{\mathrm{G}}$ & $[1,3]$ & $[2,3]$ \\
12n304$^{\mathrm{m}}$ & $[1,3]$ & $[2,3]$ & 13n1775$^{\mathrm{G}}$ & $[1,3]$ & $[2,3]$ \\
12n307$^{\mathrm{m}}$ & $[1,3]$ & $[2,3]$ & 13n1804$^{\mathrm{G}}$ & $[1,3]$ & $[2,3]$ \\
12n443$^{\mathrm{G}}$ & $[1,3]$ & $[2,3]$ & 13n1901$^{\mathrm{G}}$ & $[1,3]$ & $[2,3]$ \\
12n454$^{\mathrm{G}}$ & $[1,3]$ & $[2,3]$ & 13n1917$^{\mathrm{G}}$ & $[1,3]$ & $[2,3]$ \\
12n522$^{\mathrm{m}}$ & $[1,3]$ & $[2,3]$ & 13n1955$^{\mathrm{G}}$ & $[1,3]$ & $[2,3]$ \\
12n524$^{\mathrm{G}}$ & $[1,3]$ & $[2,3]$ & 13n1962$^{\mathrm{G}}$ & $[1,3]$ & $[2,3]$ \\
12n531$^{\mathrm{G}}$ & $[1,3]$ & $[2,3]$ & 13n1974$^{\mathrm{G}}$ & $[1,3]$ & $[2,3]$ \\
12n569$^{\mathrm{G}}$ & $[1,3]$ & $[2,3]$ & 13n1988$^{\mathrm{G}}$ & $[1,3]$ & $[2,3]$ \\
12n631$^{\mathrm{G}}$ & $[1,3]$ & $[2,3]$ & 13n1992 & $[1,3]$ & $[2,3]$ \\
12n675$^{\mathrm{G}}$ & $[1,3]$ & $[2,3]$ & 13n2008 & $[1,3]$ & $[2,3]$ \\
12n721$^{\mathrm{m}}$ & $[1,3]$ & $[2,3]$ & 13n2034 & $[1,3]$ & $[2,3]$ \\
12n804$^{\mathrm{G}}$ & $[1,3]$ & $[2,3]$ & 13n2064$^{\mathrm{G}}$ & $[1,3]$ & $[2,3]$ \\
12n811$^{\mathrm{G}}$ & $[1,3]$ & $[2,3]$ & 13n2105$^{\mathrm{G}}$ & $[1,3]$ & $[2,3]$ \\
12n818$^{\mathrm{G}}$ & $[1,3]$ & $[2,3]$ & 13n2153$^{\mathrm{G}}$ & $[1,3]$ & $[2,3]$ \\
12n833$^{\mathrm{G}}$ & $[1,3]$ & $[2,3]$ & 13n2202$^{\mathrm{G}}$ & $[1,3]$ & $[2,3]$ \\
12n854$^{\mathrm{G}}$ & $[1,3]$ & $[2,3]$ & 13n2245 & $[1,3]$ & $[2,3]$ \\
12n855$^{\mathrm{G}}$ & $[1,3]$ & $[2,3]$ & 13n2297$^{\mathrm{G}}$ & $[1,3]$ & $[2,3]$ \\
12n859$^{\mathrm{G}}$ & $[1,3]$ & $[2,3]$ & 13n2303$^{\mathrm{G}}$ & $[1,3]$ & $[2,3]$ \\
12n873$^{\mathrm{c}}$ & $[1,3]$ & $[2,3]$ & 13n2310 & $[1,3]$ & $[2,3]$ \\
13a12$^{\mathrm{k}}$ & $[1,3]$ & $[2,3]$ & 13n2378 & $[1,3]$ & $[2,3]$ \\
13a19$^{\mathrm{k}}$ & $[1,3]$ & $[2,3]$ & 13n2382 & $[1,3]$ & $[2,3]$ \\
13a120$^{\mathrm{k}}$ & $[1,3]$ & $[2,3]$ & 13n2383$^{\mathrm{c}}$ & $[1,3]$ & $[2,3]$ \\
13a121$^{\mathrm{k}}$ & $[1,3]$ & $[2,3]$ & 13n2472$^{\mathrm{G}}$ & $[1,3]$ & $[2,3]$ \\
13a133$^{\mathrm{k}}$ & $[1,4]$ & $[2,4]$ & 13n2482$^{\mathrm{G}}$ & $[1,3]$ & $[2,3]$ \\
13a195$^{\mathrm{k}}$ & $[1,4]$ & $[2,4]$ & 13n2486$^{\mathrm{G}}$ & $[1,3]$ & $[2,3]$ \\
13a236$^{\mathrm{k}}$ & $[1,4]$ & $[2,4]$ & 13n2494$^{\mathrm{G}}$ & $[1,3]$ & $[2,3]$ \\
13a271$^{\mathrm{k}}$ & $[1,4]$ & $[2,4]$ & 13n2496$^{\mathrm{G}}$ & $[1,3]$ & $[2,3]$ \\
13a275$^{\mathrm{k}}$ & $[1,4]$ & $[2,4]$ & 13n2498$^{\mathrm{G}}$ & $[1,3]$ & $[2,3]$ \\
13a314$^{\mathrm{k}}$ & $[1,4]$ & $[2,4]$ & 13n2502$^{\mathrm{G}}$ & $[1,3]$ & $[2,3]$ \\
13a339$^{\mathrm{k}}$ & $[1,4]$ & $[2,4]$ & 13n2542 & $[1,3]$ & $[2,3]$ \\
13a523$^{\mathrm{k}}$ & $[1,3]$ & $[2,3]$ & 13n2546$^{\mathrm{G}}$ & $[1,3]$ & $[2,3]$ \\
13a579$^{\mathrm{k}}$ & $[1,3]$ & $[2,3]$ & 13n2548$^{\mathrm{G}}$ & $[1,3]$ & $[2,3]$ \\
13a650$^{\mathrm{o}}$ & $[2,4]$ & $[3,4]$ & 13n2562 & $[1,3]$ & $[2,3]$ \\
13a656$^{\mathrm{k}}$ & $[1,3]$ & $[2,3]$ & 13n2589 & $[1,3]$ & $[2,3]$ \\
13a863$^{\mathrm{k}}$ & $[1,3]$ & $[2,3]$ & 13n2594$^{\mathrm{G}}$ & $[1,3]$ & $[2,3]$ \\
13a1069$^{\mathrm{k}}$ & $[1,4]$ & $[2,4]$ & 13n2613 & $[1,3]$ & $[2,3]$ \\
13a1549$^{\mathrm{k}}$ & $[1,3]$ & $[2,3]$ & 13n2647 & $[1,3]$ & $[2,3]$ \\
13a1656$^{\mathrm{k}}$ & $[1,3]$ & $[2,3]$ & 13n2669$^{\mathrm{G}}$ & $[1,3]$ & $[2,3]$ \\
13a1698$^{\mathrm{k}}$ & $[1,4]$ & $[2,4]$ & 13n2688 & $[1,3]$ & $[2,3]$ \\
13a1712$^{\mathrm{k}}$ & $[1,3]$ & $[2,3]$ & 13n2745$^{\mathrm{G}}$ & $[1,3]$ & $[2,3]$ \\
13a2129$^{\mathrm{o}}$ & $[2,4]$ & $[3,4]$ & 13n2758$^{\mathrm{G}}$ & $[1,3]$ & $[2,3]$ \\
13a2767$^{\mathrm{o}}$ & $[2,4]$ & $[3,4]$ & 13n2767$^{\mathrm{G}}$ & $[1,3]$ & $[2,3]$ \\
13a2879$^{\mathrm{o}}$ & $[2,4]$ & $[3,4]$ & 13n2791 & $[1,3]$ & $[2,3]$ \\
13a3044$^{\mathrm{k}}$ & $[1,4]$ & $[2,4]$ & 13n2797$^{\mathrm{G}}$ & $[1,3]$ & $[2,3]$ \\
13a3107$^{\mathrm{o}}$ & $[2,4]$ & $[3,4]$ & 13n2808$^{\mathrm{G}}$ & $[1,3]$ & $[2,3]$ \\
13a3147$^{\mathrm{k}}$ & $[1,3]$ & $[2,3]$ & 13n2928$^{\mathrm{G}}$ & $[1,3]$ & $[2,3]$ \\
13a3149$^{\mathrm{k}}$ & $[1,3]$ & $[2,3]$ & 13n2929 & $[1,3]$ & $[2,3]$ \\
13a3176$^{\mathrm{o}}$ & $[2,4]$ & $[3,4]$ & 13n2933 & $[1,3]$ & $[2,3]$ \\
13a3177$^{\mathrm{k}}$ & $[1,3]$ & $[2,3]$ & 13n2948$^{\mathrm{G}}$ & $[1,3]$ & $[2,3]$ \\
13a3184$^{\mathrm{k}}$ & $[1,3]$ & $[2,3]$ & 13n2951 & $[1,3]$ & $[2,3]$ \\
13a3236$^{\mathrm{k}}$ & $[1,3]$ & $[2,3]$ & 13n2956 & $[1,3]$ & $[2,3]$ \\
13a3268$^{\mathrm{k}}$ & $[1,3]$ & $[2,3]$ & 13n2965$^{\mathrm{G}}$ & $[1,3]$ & $[2,3]$ \\
13a3595$^{\mathrm{o}}$ & $[2,4]$ & $[3,4]$ & 13n2966 & $[1,4]$ & $[2,4]$ \\
13a3693$^{\mathrm{o}}$ & $[2,4]$ & $[3,4]$ & 13n2974 & $[1,4]$ & $[2,4]$ \\
13a4088$^{\mathrm{o}}$ & $[2,4]$ & $[3,4]$ & 13n2975 & $[1,3]$ & $[2,3]$ \\
13a4196$^{\mathrm{o}}$ & $[2,4]$ & $[3,4]$ & 13n2976 & $[1,3]$ & $[2,3]$ \\
13a4214$^{\mathrm{o}}$ & $[2,4]$ & $[3,4]$ & 13n2989 & $[1,3]$ & $[2,3]$ \\
13a4233$^{\mathrm{o}}$ & $[2,4]$ & $[3,4]$ & 13n2991$^{\mathrm{G}}$ & $[1,3]$ & $[2,3]$ \\
13a4304$^{\mathrm{o}}$ & $[2,4]$ & $[3,4]$ & 13n3039$^{\mathrm{G}}$ & $[1,3]$ & $[2,3]$ \\
13a4454$^{\mathrm{o}}$ & $[2,4]$ & $[3,4]$ & 13n3081$^{\mathrm{G}}$ & $[1,3]$ & $[2,3]$ \\
13a4493$^{\mathrm{o}}$ & $[2,4]$ & $[3,4]$ & 13n3084 & $[1,3]$ & $[2,3]$ \\
13a4554$^{\mathrm{o}}$ & $[2,4]$ & $[3,4]$ & 13n3089$^{\mathrm{G}}$ & $[1,3]$ & $[2,3]$ \\
13a4757$^{\mathrm{h}}$ & $[3,5]$ & $[4,5]$ & 13n3104$^{\mathrm{G}}$ & $[1,3]$ & $[2,3]$ \\
13a4771$^{\mathrm{h}}$ & $[3,5]$ & $[4,5]$ & 13n3127 & $[1,3]$ & $[2,3]$ \\
13a4773$^{\mathrm{o}}$ & $[2,4]$ & $[3,4]$ & 13n3144$^{\mathrm{G}}$ & $[1,3]$ & $[2,3]$ \\
13a4780$^{\mathrm{h}}$ & $[3,5]$ & $[4,5]$ & 13n3160$^{\mathrm{G}}$ & $[1,3]$ & $[2,3]$ \\
13a4781$^{\mathrm{o}}$ & $[2,4]$ & $[3,4]$ & 13n3161$^{\mathrm{G}}$ & $[1,3]$ & $[2,3]$ \\
13a4790$^{\mathrm{o}}$ & $[2,4]$ & $[3,4]$ & 13n3225$^{\mathrm{G}}$ & $[1,3]$ & $[2,3]$ \\
13a4812$^{\mathrm{h}}$ & $[3,5]$ & $[4,5]$ & 13n3227$^{\mathrm{G}}$ & $[1,3]$ & $[2,3]$ \\
13a4813$^{\mathrm{o}}$ & $[2,4]$ & $[3,4]$ & 13n3250$^{\mathrm{G}}$ & $[1,3]$ & $[2,3]$ \\
13a4821$^{\mathrm{h}}$ & $[3,5]$ & $[4,5]$ & 13n3282$^{\mathrm{G}}$ & $[1,3]$ & $[2,3]$ \\
13a4839$^{\mathrm{h}}$ & $[3,5]$ & $[4,5]$ & 13n3298$^{\mathrm{G}}$ & $[1,3]$ & $[2,3]$ \\
13a4840$^{\mathrm{o}}$ & $[2,4]$ & $[3,4]$ & 13n3312$^{\mathrm{G}}$ & $[1,3]$ & $[2,3]$ \\
13a4841$^{\mathrm{h}}$ & $[3,5]$ & $[4,5]$ & 13n3317$^{\mathrm{G}}$ & $[1,3]$ & $[2,3]$ \\
13a4842$^{\mathrm{o}}$ & $[2,4]$ & $[3,4]$ & 13n3332$^{\mathrm{G}}$ & $[1,3]$ & $[2,3]$ \\
13a4851$^{\mathrm{o}}$ & $[2,4]$ & $[3,4]$ & 13n3347$^{\mathrm{G}}$ & $[1,3]$ & $[2,3]$ \\
13a4853$^{\mathrm{o}}$ & $[2,4]$ & $[3,4]$ & 13n3377$^{\mathrm{G}}$ & $[1,3]$ & $[2,3]$ \\
13a4855$^{\mathrm{o}}$ & $[2,4]$ & $[3,4]$ & 13n3421 & $[1,3]$ & $[2,3]$ \\
13a4859$^{\mathrm{h}}$ & $[3,5]$ & $[4,5]$ & 13n3449$^{\mathrm{G}}$ & $[1,3]$ & $[2,3]$ \\
13a4860$^{\mathrm{h}}$ & $[3,5]$ & $[4,5]$ & 13n3470 & $[1,3]$ & $[2,3]$ \\
13a4861$^{\mathrm{o}}$ & $[2,4]$ & $[3,4]$ & 13n3503 & $[1,3]$ & $[2,3]$ \\
13a4865$^{\mathrm{h}}$ & $[3,5]$ & $[4,5]$ & 13n3505 & $[1,3]$ & $[2,3]$ \\
13a4869$^{\mathrm{o}}$ & $[2,4]$ & $[3,4]$ & 13n3507 & $[1,3]$ & $[2,3]$ \\
13a4870$^{\mathrm{h}}$ & $[3,5]$ & $[4,5]$ & 13n3508$^{\mathrm{G}}$ & $[2,4]$ & $[3,4]$ \\
13a4871$^{\mathrm{o}}$ & $[2,4]$ & $[3,4]$ & 13n3512$^{\mathrm{G}}$ & $[1,3]$ & $[2,3]$ \\
13a4872$^{\mathrm{o}}$ & $[2,4]$ & $[3,4]$ & 13n3514$^{\mathrm{G}}$ & $[1,3]$ & $[2,3]$ \\
13a4876$^{\mathrm{h}}$ & $[3,5]$ & $[4,5]$ & 13n3516$^{\mathrm{G}}$ & $[1,3]$ & $[2,3]$ \\
13a4877$^{\mathrm{g}}$ & $[2,5]$ & $[3,5]$ & 13n3524 & $[1,3]$ & $[2,3]$ \\
13n31 & $[1,3]$ & $[2,3]$ & 13n3539 & $[1,3]$ & $[2,3]$ \\
13n41 & $[1,3]$ & $[2,3]$ & 13n3550 & $[1,3]$ & $[2,3]$ \\
13n81 & $[1,3]$ & $[2,3]$ & 13n3554 & $[1,3]$ & $[2,3]$ \\
13n84$^{\mathrm{G}}$ & $[1,3]$ & $[2,3]$ & 13n3558$^{\mathrm{G}}$ & $[1,3]$ & $[2,3]$ \\
13n94 & $[1,3]$ & $[2,3]$ & 13n3562$^{\mathrm{G}}$ & $[1,3]$ & $[2,3]$ \\
13n138$^{\mathrm{G}}$ & $[1,3]$ & $[2,3]$ & 13n3583$^{\mathrm{G}}$ & $[1,3]$ & $[2,3]$ \\
13n197 & $[1,3]$ & $[2,3]$ & 13n3595$^{\mathrm{G}}$ & $[1,3]$ & $[2,3]$ \\
13n199$^{\mathrm{G}}$ & $[1,3]$ & $[2,3]$ & 13n3688$^{\mathrm{G}}$ & $[1,3]$ & $[2,3]$ \\
13n231$^{\mathrm{G}}$ & $[1,3]$ & $[2,3]$ & 13n3700 & $[1,3]$ & $[2,3]$ \\
13n240$^{\mathrm{G}}$ & $[1,4]$ & $[2,4]$ & 13n3703$^{\mathrm{G}}$ & $[1,3]$ & $[2,3]$ \\
13n254$^{\mathrm{G}}$ & $[1,3]$ & $[2,3]$ & 13n3721 & $[1,3]$ & $[2,3]$ \\
13n290$^{\mathrm{G}}$ & $[1,3]$ & $[2,3]$ & 13n3724$^{\mathrm{G}}$ & $[1,3]$ & $[2,3]$ \\
13n299$^{\mathrm{G}}$ & $[1,4]$ & $[2,4]$ & 13n3728 & $[1,3]$ & $[2,3]$ \\
13n361$^{\mathrm{G}}$ & $[1,3]$ & $[2,3]$ & 13n3755$^{\mathrm{G}}$ & $[1,3]$ & $[2,3]$ \\
13n367 & $[1,3]$ & $[2,3]$ & 13n3787 & $[1,3]$ & $[2,3]$ \\
13n400 & $[1,3]$ & $[2,3]$ & 13n3801$^{\mathrm{G}}$ & $[1,3]$ & $[2,3]$ \\
13n406 & $[1,3]$ & $[2,3]$ & 13n3836$^{\mathrm{G}}$ & $[1,3]$ & $[2,3]$ \\
13n418 & $[1,3]$ & $[2,3]$ & 13n3868$^{\mathrm{G}}$ & $[1,3]$ & $[2,3]$ \\
13n426$^{\mathrm{G}}$ & $[1,3]$ & $[2,3]$ & 13n3884$^{\mathrm{G}}$ & $[1,3]$ & $[2,3]$ \\
13n427$^{\mathrm{G}}$ & $[1,3]$ & $[2,3]$ & 13n3902 & $[1,3]$ & $[2,3]$ \\
13n452$^{\mathrm{G}}$ & $[1,3]$ & $[2,3]$ & 13n3923 & $[1,3]$ & $[2,3]$ \\
13n454$^{\mathrm{G}}$ & $[1,3]$ & $[2,3]$ & 13n3944 & $[1,3]$ & $[2,3]$ \\
13n460$^{\mathrm{G}}$ & $[1,3]$ & $[2,3]$ & 13n3949$^{\mathrm{G}}$ & $[1,3]$ & $[2,3]$ \\
13n465$^{\mathrm{G}}$ & $[1,3]$ & $[2,3]$ & 13n4067$^{\mathrm{G}}$ & $[1,3]$ & $[2,3]$ \\
13n474$^{\mathrm{G}}$ & $[1,3]$ & $[2,3]$ & 13n4074$^{\mathrm{G}}$ & $[1,3]$ & $[2,3]$ \\
13n477$^{\mathrm{G}}$ & $[1,3]$ & $[2,3]$ & 13n4096$^{\mathrm{G}}$ & $[1,3]$ & $[2,3]$ \\
13n480$^{\mathrm{G}}$ & $[1,3]$ & $[2,3]$ & 13n4117$^{\mathrm{G}}$ & $[1,3]$ & $[2,3]$ \\
13n488$^{\mathrm{G}}$ & $[1,3]$ & $[2,3]$ & 13n4162$^{\mathrm{G}}$ & $[1,3]$ & $[2,3]$ \\
13n494 & $[1,3]$ & $[2,3]$ & 13n4181$^{\mathrm{G}}$ & $[1,3]$ & $[2,3]$ \\
13n538$^{\mathrm{G}}$ & $[1,3]$ & $[2,3]$ & 13n4226 & $[1,3]$ & $[2,3]$ \\
13n556$^{\mathrm{G}}$ & $[1,3]$ & $[2,3]$ & 13n4231$^{\mathrm{G}}$ & $[1,3]$ & $[2,3]$ \\
13n606$^{\mathrm{G}}$ & $[1,3]$ & $[2,3]$ & 13n4233$^{\mathrm{G}}$ & $[1,3]$ & $[2,3]$ \\
13n607$^{\mathrm{G}}$ & $[1,3]$ & $[2,3]$ & 13n4262 & $[1,3]$ & $[2,3]$ \\
13n608$^{\mathrm{G}}$ & $[1,3]$ & $[2,3]$ & 13n4309$^{\mathrm{G}}$ & $[1,3]$ & $[2,3]$ \\
13n618$^{\mathrm{G}}$ & $[1,3]$ & $[2,3]$ & 13n4343 & $[1,3]$ & $[2,3]$ \\
13n646$^{\mathrm{G}}$ & $[1,3]$ & $[2,3]$ & 13n4345$^{\mathrm{G}}$ & $[1,3]$ & $[2,3]$ \\
13n683$^{\mathrm{G}}$ & $[1,3]$ & $[2,3]$ & 13n4349$^{\mathrm{G}}$ & $[1,3]$ & $[2,3]$ \\
13n689$^{\mathrm{t}}$ & $[1,3]$ & $[2,3]$ & 13n4354$^{\mathrm{G}}$ & $[1,3]$ & $[2,3]$ \\
13n714 & $[1,3]$ & $[2,3]$ & 13n4366 & $[1,3]$ & $[2,3]$ \\
13n720$^{\mathrm{G}}$ & $[1,3]$ & $[2,3]$ & 13n4367$^{\mathrm{G}}$ & $[1,3]$ & $[2,3]$ \\
13n732$^{\mathrm{G}}$ & $[1,3]$ & $[2,3]$ & 13n4371 & $[1,3]$ & $[2,3]$ \\
13n733$^{\mathrm{G}}$ & $[1,3]$ & $[2,3]$ & 13n4461$^{\mathrm{G}}$ & $[1,3]$ & $[2,3]$ \\
13n750 & $[1,3]$ & $[2,3]$ & 13n4489$^{\mathrm{G}}$ & $[1,3]$ & $[2,3]$ \\
13n754$^{\mathrm{G}}$ & $[1,3]$ & $[2,3]$ & 13n4499$^{\mathrm{G}}$ & $[1,3]$ & $[2,3]$ \\
13n769$^{\mathrm{G}}$ & $[1,3]$ & $[2,3]$ & 13n4515$^{\mathrm{G}}$ & $[1,3]$ & $[2,3]$ \\
13n772$^{\mathrm{G}}$ & $[1,3]$ & $[2,3]$ & 13n4531 & $[1,3]$ & $[2,3]$ \\
13n788$^{\mathrm{G}}$ & $[1,3]$ & $[2,3]$ & 13n4541$^{\mathrm{G}}$ & $[1,3]$ & $[2,3]$ \\
13n798$^{\mathrm{G}}$ & $[1,3]$ & $[2,3]$ & 13n4552 & $[1,3]$ & $[2,3]$ \\
13n819$^{\mathrm{G}}$ & $[1,3]$ & $[2,3]$ & 13n4588 & $[1,4]$ & $[2,4]$ \\
13n837$^{\mathrm{G}}$ & $[1,3]$ & $[2,3]$ & 13n4626$^{\mathrm{G}}$ & $[1,3]$ & $[2,3]$ \\
13n861$^{\mathrm{G}}$ & $[1,3]$ & $[2,3]$ & 13n4645$^{\mathrm{G}}$ & $[1,3]$ & $[2,3]$ \\
13n869$^{\mathrm{G}}$ & $[1,3]$ & $[2,3]$ & 13n4646 & $[1,3]$ & $[2,3]$ \\
13n879$^{\mathrm{G}}$ & $[1,3]$ & $[2,3]$ & 13n4650$^{\mathrm{G}}$ & $[1,3]$ & $[2,3]$ \\
13n892$^{\mathrm{G}}$ & $[1,3]$ & $[2,3]$ & 13n4667 & $[1,3]$ & $[2,3]$ \\
13n923$^{\mathrm{G}}$ & $[1,3]$ & $[2,3]$ & 13n4674 & $[1,3]$ & $[2,3]$ \\
13n932 & $[1,3]$ & $[2,3]$ & 13n4698 & $[1,3]$ & $[2,3]$ \\
13n938$^{\mathrm{G}}$ & $[1,3]$ & $[2,3]$ & 13n4715 & $[1,3]$ & $[2,3]$ \\
13n980$^{\mathrm{G}}$ & $[1,4]$ & $[2,4]$ & 13n4717 & $[1,3]$ & $[2,3]$ \\
13n984$^{\mathrm{G}}$ & $[1,3]$ & $[2,3]$ & 13n4762$^{\mathrm{G}}$ & $[1,3]$ & $[2,3]$ \\
13n989 & $[1,3]$ & $[2,3]$ & 13n4771 & $[1,3]$ & $[2,3]$ \\
13n994 & $[1,3]$ & $[2,3]$ & 13n4795$^{\mathrm{G}}$ & $[2,4]$ & $[3,4]$ \\
13n1004 & $[1,3]$ & $[2,3]$ & 13n4831 & $[1,3]$ & $[2,3]$ \\
13n1009$^{\mathrm{G}}$ & $[1,3]$ & $[2,3]$ & 13n4834$^{\mathrm{G}}$ & $[1,3]$ & $[2,3]$ \\
13n1014$^{\mathrm{G}}$ & $[1,3]$ & $[2,3]$ & 13n4870 & $[1,3]$ & $[2,3]$ \\
13n1020$^{\mathrm{G}}$ & $[1,3]$ & $[2,3]$ & 13n4885$^{\mathrm{G}}$ & $[1,3]$ & $[2,3]$ \\
13n1050$^{\mathrm{G}}$ & $[1,3]$ & $[2,3]$ & 13n4887$^{\mathrm{G}}$ & $[1,3]$ & $[2,3]$ \\
13n1055$^{\mathrm{G}}$ & $[1,3]$ & $[2,3]$ & 13n4891 & $[1,3]$ & $[2,3]$ \\
13n1061 & $[1,3]$ & $[2,3]$ & 13n4892 & $[1,3]$ & $[2,3]$ \\
13n1089$^{\mathrm{G}}$ & $[1,3]$ & $[2,3]$ & 13n4896$^{\mathrm{G}}$ & $[1,3]$ & $[2,3]$ \\
13n1103$^{\mathrm{G}}$ & $[1,4]$ & $[2,4]$ & 13n4904$^{\mathrm{G}}$ & $[1,3]$ & $[2,3]$ \\
13n1111$^{\mathrm{G}}$ & $[1,3]$ & $[2,3]$ & 13n4905$^{\mathrm{G}}$ & $[1,3]$ & $[2,3]$ \\
13n1126$^{\mathrm{G}}$ & $[1,3]$ & $[2,3]$ & 13n4928$^{\mathrm{G}}$ & $[1,3]$ & $[2,3]$ \\
13n1194$^{\mathrm{G}}$ & $[1,3]$ & $[2,3]$ & 13n4929$^{\mathrm{G}}$ & $[2,4]$ & $[3,4]$ \\
13n1198$^{\mathrm{G}}$ & $[1,3]$ & $[2,3]$ & 13n4932$^{\mathrm{c}}$ & $[1,3]$ & $[2,3]$ \\
13n1203$^{\mathrm{G}}$ & $[1,3]$ & $[2,3]$ & 13n4939 & $[1,3]$ & $[2,3]$ \\
13n1228 & $[1,4]$ & $[2,4]$ & 13n4981$^{\mathrm{G}}$ & $[1,3]$ & $[2,3]$ \\
13n1233$^{\mathrm{G}}$ & $[1,3]$ & $[2,3]$ & 13n4984$^{\mathrm{G}}$ & $[1,3]$ & $[2,3]$ \\
13n1268$^{\mathrm{G}}$ & $[1,3]$ & $[2,3]$ & 13n4997$^{\mathrm{G}}$ & $[1,3]$ & $[2,3]$ \\
13n1272$^{\mathrm{G}}$ & $[1,3]$ & $[2,3]$ & 13n5009 & $[1,3]$ & $[2,3]$ \\
13n1277$^{\mathrm{G}}$ & $[1,3]$ & $[2,3]$ & 13n5032 & $[1,3]$ & $[2,3]$ \\
13n1284$^{\mathrm{G}}$ & $[1,3]$ & $[2,3]$ & 13n5036 & $[1,3]$ & $[2,3]$ \\
13n1290$^{\mathrm{G}}$ & $[1,3]$ & $[2,3]$ & 13n5062 & $[1,3]$ & $[2,3]$ \\
13n1311$^{\mathrm{G}}$ & $[1,3]$ & $[2,3]$ & 13n5069$^{\mathrm{G}}$ & $[1,3]$ & $[2,3]$ \\
13n1328$^{\mathrm{G}}$ & $[1,3]$ & $[2,3]$ & 13n5085 & $[1,3]$ & $[2,3]$ \\
13n1333$^{\mathrm{G}}$ & $[1,3]$ & $[2,3]$ & 13n5089$^{\mathrm{G}}$ & $[1,3]$ & $[2,3]$ \\*
13n1346$^{\mathrm{G}}$ & $[1,3]$ & $[2,3]$ & 13n5097 & $[1,3]$ & $[2,3]$ \\*
13n1361$^{\mathrm{G}}$ & $[1,3]$ & $[2,3]$ & 13n5105$^{\mathrm{G}}$ & $[2,4]$ & $[3,4]$ \\*
\multicolumn{6}{@{}l@{}}{\usebox{\appendixLegendBox}} \\
\end{longtable}}

\section{Alternating knots with range \texorpdfstring{$[2,3]$}{[2,3]} and no candidate crossing}
\label{app:dichotomy}
The following \nDichotomy{} alternating knots retain range $[2,3]$. The analysis
in Section~\ref{sec:minimal} verifies that no crossing change in any minimal
diagram gives a knot of unknotting number one. This rules out the strong
Bernhard--Jablan equality if $u=2$, but does not determine their unknotting
numbers. These entries are separate from the exact values in
Appendices~\ref{app:u5}--\ref{app:u2} and are ordered down the columns.

{\footnotesize\setlength{\tabcolsep}{5pt}
\setlength{\LTleft}{\fill}
\setlength{\LTright}{\fill}
\sbox{\appendixLegendBox}{%
\begin{minipage}{\textwidth}
\hrule height\heavyrulewidth\vspace{3pt}
\raggedright\setlength{\parindent}{0pt}\setlength{\parskip}{0pt}
\strut No crossing change in a minimal diagram gives a knot of unknotting number one (Section~\ref{sec:minimal}).\strut\par
\end{minipage}}
\begin{longtable}{*{6}{>{\raggedright\arraybackslash}p{\dimexpr\textwidth/6-2\tabcolsep\relax}}}
\toprule
\endfirsthead
\multicolumn{6}{@{}l@{}}{\emph{Knots without a candidate crossing, continued from the previous page.}} \\[2pt]
\toprule
\endhead
\bottomrule
\multicolumn{6}{@{}r@{}}{\emph{continued on the next page}} \\
\endfoot
\endlastfoot
$10_{6}$ & 12a1050 & 13a564 & 13a1541 & 13a2533 & 13a3560 \\
$10_{11}$ & 12a1066 & 13a569 & 13a1542 & 13a2535 & 13a3561 \\
$10_{47}$ & 12a1080 & 13a573 & 13a1543 & 13a2541 & 13a3562 \\
$10_{51}$ & 12a1084 & 13a574 & 13a1548 & 13a2545 & 13a3564 \\
$10_{54}$ & 12a1089 & 13a578 & 13a1553 & 13a2550 & 13a3567 \\
$10_{61}$ & 12a1099 & 13a579 & 13a1559 & 13a2557 & 13a3568 \\
$10_{76}$ & 12a1103 & 13a585 & 13a1560 & 13a2559 & 13a3577 \\
$10_{77}$ & 12a1107 & 13a587 & 13a1561 & 13a2560 & 13a3581 \\
$10_{79}$ & 12a1109 & 13a588 & 13a1567 & 13a2568 & 13a3593 \\
$10_{100}$ & 12a1118 & 13a593 & 13a1569 & 13a2573 & 13a3596 \\
11a14 & 12a1124 & 13a594 & 13a1575 & 13a2574 & 13a3597 \\
11a18 & 12a1127 & 13a596 & 13a1581 & 13a2591 & 13a3605 \\
11a20 & 12a1147 & 13a597 & 13a1583 & 13a2594 & 13a3608 \\
11a45 & 12a1148 & 13a598 & 13a1584 & 13a2595 & 13a3609 \\
11a49 & 12a1151 & 13a601 & 13a1586 & 13a2600 & 13a3613 \\
11a53 & 12a1159 & 13a608 & 13a1589 & 13a2610 & 13a3653 \\
11a60 & 12a1160 & 13a613 & 13a1601 & 13a2613 & 13a3667 \\
11a83 & 12a1162 & 13a617 & 13a1603 & 13a2614 & 13a3689 \\
11a105 & 12a1163 & 13a618 & 13a1608 & 13a2616 & 13a3690 \\
11a126 & 12a1165 & 13a619 & 13a1609 & 13a2618 & 13a3692 \\
11a137 & 12a1184 & 13a623 & 13a1619 & 13a2622 & 13a3694 \\
11a161 & 12a1202 & 13a624 & 13a1630 & 13a2623 & 13a3701 \\
11a197 & 12a1205 & 13a629 & 13a1639 & 13a2628 & 13a3706 \\
11a202 & 12a1210 & 13a630 & 13a1640 & 13a2632 & 13a3727 \\
11a293 & 12a1223 & 13a631 & 13a1645 & 13a2644 & 13a3736 \\
11a304 & 12a1240 & 13a633 & 13a1656 & 13a2663 & 13a3777 \\
11a346 & 12a1243 & 13a636 & 13a1662 & 13a2695 & 13a3779 \\
11a362 & 12a1247 & 13a638 & 13a1665 & 13a2709 & 13a3782 \\
11a363 & 12a1254 & 13a639 & 13a1675 & 13a2710 & 13a3789 \\
12a41 & 12a1255 & 13a641 & 13a1676 & 13a2711 & 13a3808 \\
12a44 & 12a1256 & 13a655 & 13a1677 & 13a2728 & 13a3813 \\
12a49 & 12a1266 & 13a656 & 13a1683 & 13a2748 & 13a3820 \\
12a50 & 12a1278 & 13a658 & 13a1688 & 13a2753 & 13a3822 \\
12a64 & 12a1279 & 13a683 & 13a1692 & 13a2772 & 13a3823 \\
12a85 & 12a1281 & 13a701 & 13a1699 & 13a2779 & 13a3830 \\
12a86 & 12a1282 & 13a702 & 13a1700 & 13a2780 & 13a3833 \\
12a89 & 12a1285 & 13a721 & 13a1703 & 13a2794 & 13a3845 \\
12a103 & 12a1286 & 13a735 & 13a1712 & 13a2795 & 13a3864 \\
12a104 & 12a1287 & 13a743 & 13a1717 & 13a2799 & 13a3867 \\
12a114 & 12a1288 & 13a746 & 13a1719 & 13a2807 & 13a3869 \\
12a117 & 13a4 & 13a798 & 13a1725 & 13a2813 & 13a3871 \\
12a125 & 13a12 & 13a800 & 13a1741 & 13a2814 & 13a3872 \\
12a127 & 13a17 & 13a811 & 13a1744 & 13a2817 & 13a3874 \\
12a153 & 13a19 & 13a813 & 13a1757 & 13a2822 & 13a3914 \\
12a161 & 13a25 & 13a819 & 13a1764 & 13a2837 & 13a3922 \\
12a168 & 13a34 & 13a826 & 13a1772 & 13a2841 & 13a3923 \\
12a178 & 13a35 & 13a828 & 13a1788 & 13a2850 & 13a3935 \\
12a181 & 13a40 & 13a835 & 13a1790 & 13a2872 & 13a3943 \\
12a183 & 13a51 & 13a861 & 13a1804 & 13a2882 & 13a3945 \\
12a186 & 13a56 & 13a862 & 13a1811 & 13a2885 & 13a3950 \\
12a193 & 13a71 & 13a863 & 13a1820 & 13a2886 & 13a3976 \\
12a195 & 13a78 & 13a895 & 13a1829 & 13a2894 & 13a3980 \\
12a196 & 13a96 & 13a908 & 13a1865 & 13a2897 & 13a3987 \\
12a199 & 13a106 & 13a910 & 13a1867 & 13a2898 & 13a3988 \\
12a211 & 13a112 & 13a922 & 13a1868 & 13a2907 & 13a4016 \\
12a212 & 13a113 & 13a929 & 13a1887 & 13a2910 & 13a4025 \\
12a238 & 13a116 & 13a953 & 13a1921 & 13a2917 & 13a4030 \\
12a240 & 13a119 & 13a959 & 13a1927 & 13a2927 & 13a4033 \\
12a244 & 13a121 & 13a960 & 13a1951 & 13a2931 & 13a4055 \\
12a294 & 13a123 & 13a965 & 13a1952 & 13a2933 & 13a4056 \\
12a297 & 13a124 & 13a971 & 13a1953 & 13a2937 & 13a4070 \\
12a302 & 13a126 & 13a975 & 13a1955 & 13a2951 & 13a4072 \\
12a304 & 13a129 & 13a976 & 13a1957 & 13a2981 & 13a4096 \\
12a315 & 13a130 & 13a984 & 13a1958 & 13a2984 & 13a4097 \\
12a329 & 13a134 & 13a989 & 13a1959 & 13a2988 & 13a4098 \\
12a350 & 13a136 & 13a991 & 13a1965 & 13a2991 & 13a4102 \\
12a370 & 13a140 & 13a995 & 13a1966 & 13a2996 & 13a4108 \\
12a371 & 13a147 & 13a996 & 13a1968 & 13a3010 & 13a4114 \\
12a372 & 13a148 & 13a997 & 13a1969 & 13a3019 & 13a4120 \\
12a375 & 13a151 & 13a998 & 13a1985 & 13a3021 & 13a4123 \\
12a376 & 13a152 & 13a1002 & 13a1986 & 13a3029 & 13a4126 \\
12a379 & 13a156 & 13a1003 & 13a1993 & 13a3030 & 13a4127 \\
12a380 & 13a158 & 13a1005 & 13a1998 & 13a3034 & 13a4134 \\
12a381 & 13a159 & 13a1006 & 13a2005 & 13a3038 & 13a4153 \\
12a382 & 13a163 & 13a1030 & 13a2015 & 13a3039 & 13a4154 \\
12a414 & 13a175 & 13a1034 & 13a2023 & 13a3046 & 13a4155 \\
12a423 & 13a177 & 13a1036 & 13a2024 & 13a3054 & 13a4166 \\
12a424 & 13a179 & 13a1060 & 13a2027 & 13a3073 & 13a4188 \\
12a434 & 13a189 & 13a1077 & 13a2030 & 13a3090 & 13a4193 \\
12a436 & 13a194 & 13a1080 & 13a2031 & 13a3109 & 13a4202 \\
12a449 & 13a197 & 13a1090 & 13a2042 & 13a3118 & 13a4203 \\
12a454 & 13a202 & 13a1102 & 13a2043 & 13a3131 & 13a4208 \\
12a461 & 13a213 & 13a1112 & 13a2047 & 13a3146 & 13a4230 \\
12a462 & 13a225 & 13a1122 & 13a2063 & 13a3147 & 13a4231 \\
12a477 & 13a232 & 13a1139 & 13a2069 & 13a3149 & 13a4234 \\
12a481 & 13a234 & 13a1148 & 13a2070 & 13a3156 & 13a4240 \\
12a482 & 13a237 & 13a1149 & 13a2079 & 13a3158 & 13a4241 \\
12a493 & 13a250 & 13a1150 & 13a2087 & 13a3159 & 13a4275 \\
12a496 & 13a255 & 13a1153 & 13a2090 & 13a3161 & 13a4278 \\
12a508 & 13a262 & 13a1156 & 13a2099 & 13a3166 & 13a4281 \\
12a533 & 13a266 & 13a1159 & 13a2100 & 13a3167 & 13a4290 \\
12a544 & 13a267 & 13a1165 & 13a2103 & 13a3168 & 13a4310 \\
12a545 & 13a269 & 13a1166 & 13a2104 & 13a3169 & 13a4336 \\
12a553 & 13a273 & 13a1169 & 13a2110 & 13a3170 & 13a4337 \\
12a556 & 13a277 & 13a1175 & 13a2114 & 13a3172 & 13a4338 \\
12a568 & 13a284 & 13a1185 & 13a2119 & 13a3174 & 13a4341 \\
12a580 & 13a285 & 13a1197 & 13a2132 & 13a3175 & 13a4345 \\
12a592 & 13a290 & 13a1198 & 13a2134 & 13a3177 & 13a4351 \\
12a597 & 13a295 & 13a1214 & 13a2142 & 13a3179 & 13a4354 \\
12a600 & 13a302 & 13a1216 & 13a2152 & 13a3182 & 13a4357 \\
12a616 & 13a312 & 13a1224 & 13a2153 & 13a3184 & 13a4360 \\
12a636 & 13a315 & 13a1226 & 13a2156 & 13a3187 & 13a4370 \\
12a639 & 13a319 & 13a1229 & 13a2158 & 13a3188 & 13a4374 \\
12a641 & 13a326 & 13a1233 & 13a2159 & 13a3189 & 13a4389 \\
12a643 & 13a327 & 13a1239 & 13a2173 & 13a3190 & 13a4401 \\
12a649 & 13a333 & 13a1246 & 13a2176 & 13a3191 & 13a4403 \\
12a651 & 13a336 & 13a1247 & 13a2181 & 13a3192 & 13a4407 \\
12a665 & 13a338 & 13a1255 & 13a2191 & 13a3198 & 13a4414 \\
12a668 & 13a351 & 13a1258 & 13a2193 & 13a3205 & 13a4415 \\
12a669 & 13a353 & 13a1259 & 13a2194 & 13a3217 & 13a4421 \\
12a680 & 13a359 & 13a1267 & 13a2195 & 13a3220 & 13a4449 \\
12a681 & 13a360 & 13a1274 & 13a2203 & 13a3222 & 13a4451 \\
12a684 & 13a365 & 13a1275 & 13a2208 & 13a3225 & 13a4468 \\
12a686 & 13a370 & 13a1287 & 13a2210 & 13a3229 & 13a4478 \\
12a689 & 13a371 & 13a1303 & 13a2228 & 13a3232 & 13a4483 \\
12a693 & 13a374 & 13a1304 & 13a2238 & 13a3233 & 13a4498 \\
12a702 & 13a377 & 13a1309 & 13a2246 & 13a3241 & 13a4499 \\
12a706 & 13a380 & 13a1311 & 13a2263 & 13a3251 & 13a4502 \\
12a730 & 13a382 & 13a1332 & 13a2264 & 13a3255 & 13a4521 \\
12a741 & 13a384 & 13a1344 & 13a2275 & 13a3256 & 13a4522 \\
12a767 & 13a387 & 13a1361 & 13a2277 & 13a3257 & 13a4537 \\
12a784 & 13a393 & 13a1363 & 13a2284 & 13a3259 & 13a4589 \\
12a789 & 13a399 & 13a1372 & 13a2301 & 13a3260 & 13a4590 \\
12a791 & 13a412 & 13a1377 & 13a2326 & 13a3265 & 13a4591 \\
12a824 & 13a413 & 13a1378 & 13a2327 & 13a3269 & 13a4593 \\
12a825 & 13a414 & 13a1379 & 13a2328 & 13a3270 & 13a4599 \\
12a826 & 13a418 & 13a1383 & 13a2336 & 13a3274 & 13a4602 \\
12a827 & 13a420 & 13a1389 & 13a2353 & 13a3277 & 13a4608 \\
12a835 & 13a421 & 13a1392 & 13a2356 & 13a3297 & 13a4609 \\
12a841 & 13a423 & 13a1393 & 13a2366 & 13a3299 & 13a4617 \\
12a845 & 13a426 & 13a1401 & 13a2367 & 13a3302 & 13a4626 \\
12a862 & 13a427 & 13a1404 & 13a2376 & 13a3308 & 13a4631 \\
12a869 & 13a440 & 13a1406 & 13a2381 & 13a3325 & 13a4633 \\
12a878 & 13a448 & 13a1407 & 13a2386 & 13a3333 & 13a4636 \\
12a879 & 13a449 & 13a1408 & 13a2391 & 13a3343 & 13a4645 \\
12a881 & 13a457 & 13a1409 & 13a2395 & 13a3364 & 13a4646 \\
12a886 & 13a459 & 13a1410 & 13a2396 & 13a3371 & 13a4647 \\
12a896 & 13a468 & 13a1414 & 13a2400 & 13a3373 & 13a4648 \\
12a898 & 13a474 & 13a1432 & 13a2405 & 13a3376 & 13a4650 \\
12a899 & 13a476 & 13a1437 & 13a2409 & 13a3407 & 13a4651 \\
12a901 & 13a480 & 13a1438 & 13a2417 & 13a3408 & 13a4654 \\
12a911 & 13a481 & 13a1439 & 13a2421 & 13a3421 & 13a4656 \\
12a912 & 13a486 & 13a1442 & 13a2422 & 13a3444 & 13a4660 \\
12a916 & 13a493 & 13a1452 & 13a2432 & 13a3451 & 13a4663 \\
12a940 & 13a496 & 13a1453 & 13a2436 & 13a3452 & 13a4699 \\
12a947 & 13a523 & 13a1461 & 13a2438 & 13a3453 & 13a4703 \\
12a957 & 13a528 & 13a1462 & 13a2441 & 13a3464 & 13a4704 \\
12a981 & 13a531 & 13a1476 & 13a2442 & 13a3476 & 13a4705 \\
12a985 & 13a535 & 13a1496 & 13a2445 & 13a3482 & 13a4726 \\
12a988 & 13a536 & 13a1501 & 13a2447 & 13a3495 & 13a4733 \\
12a989 & 13a537 & 13a1503 & 13a2449 & 13a3496 & 13a4749 \\
12a999 & 13a544 & 13a1504 & 13a2451 & 13a3504 & 13a4758 \\
12a1000 & 13a545 & 13a1506 & 13a2453 & 13a3511 & 13a4766 \\
12a1009 & 13a547 & 13a1510 & 13a2464 & 13a3512 & 13a4782 \\
12a1016 & 13a549 & 13a1528 & 13a2475 & 13a3515 & 13a4792 \\
12a1017 & 13a551 & 13a1529 & 13a2485 & 13a3529 & 13a4811 \\
12a1028 & 13a552 & 13a1532 & 13a2488 & 13a3536 & 13a4817 \\
12a1031 & 13a553 & 13a1534 & 13a2493 & 13a3549 & 13a4843 \\*
12a1039 & 13a556 & 13a1538 & 13a2507 & 13a3551 & 13a4856 \\*
12a1040 & 13a557 & 13a1539 & 13a2512 & 13a3553 &  \\*
\multicolumn{6}{@{}l@{}}{\usebox{\appendixLegendBox}} \\
\end{longtable}}

\section{Marked crossing-change diagrams}\label{app:diagrams}
The following PD codes specify the nineteen diagrams of
Section~\ref{sec:eight}. Changing the crossings represented by bold
four-tuples gives the knot in the third column, up to mirror image. This
is the unknot, denoted $0_1$, except for \kn{13n}{447}, where the single
change gives $10_{91}$ and $u(10_{91})=1$. All other crossings are unchanged.

{\small
\setlength{\tabcolsep}{4pt}
\setlength{\LTleft}{\fill}
\setlength{\LTright}{\fill}
\renewcommand{\arraystretch}{1.25}
\begin{longtable}{@{}>{\raggedright\arraybackslash}p{1.45cm}>{\raggedleft\arraybackslash}p{0.8cm}
>{\centering\arraybackslash}p{1.05cm}
>{\raggedright\arraybackslash\scriptsize}p{\dimexpr\textwidth-3.3cm-6\tabcolsep\relax}@{}}
\caption{PD codes and marked crossings for the upper-bound constructions.}\label{tab:pdcodes}\\
\toprule
Knot & $c(D)$ & Target & Planar diagram (PD) code\\
\midrule
\endfirsthead
\multicolumn{4}{@{}l@{}}{\emph{Marked crossing-change diagrams, continued.}}\\[2pt]
\toprule
Knot & $c(D)$ & Target & Planar diagram (PD) code\\
\midrule
\endhead
\bottomrule
\multicolumn{4}{@{}r@{}}{\emph{continued on the next page}}\\
\endfoot
\bottomrule
\endlastfoot
\kn{13n}{30} & 14 & $0_1$ & $[[1,9,2,8],\allowbreak [18,7,19,8],\allowbreak \mathbf{[6,21,7,22]},\allowbreak [26,5,27,6],\allowbreak [4,25,5,26],\allowbreak [24,3,25,4],\allowbreak [2,15,3,16],\allowbreak [13,28,14,1],\allowbreak \mathbf{[27,10,28,11]},\allowbreak [16,24,17,23],\allowbreak [22,18,23,17],\allowbreak [20,12,21,11],\allowbreak [12,20,13,19],\allowbreak [9,14,10,15]]$ \\[3pt]
\kn{13n}{45} & 14 & $0_1$ & $[[1,7,2,6],\allowbreak [10,6,11,5],\allowbreak \mathbf{[4,21,5,22]},\allowbreak [26,3,27,4],\allowbreak [2,16,3,15],\allowbreak [7,1,8,28],\allowbreak \mathbf{[27,16,28,17]},\allowbreak [14,25,15,26],\allowbreak [24,13,25,14],\allowbreak [12,23,13,24],\allowbreak [22,11,23,12],\allowbreak [20,18,21,17],\allowbreak [8,19,9,20],\allowbreak [18,9,19,10]]$ \\[3pt]
\kn{13n}{80} & 14 & $0_1$ & $[[1,24,2,25],\allowbreak [23,2,24,3],\allowbreak \mathbf{[15,23,16,22]},\allowbreak [8,22,9,21],\allowbreak [20,14,21,13],\allowbreak [12,20,13,19],\allowbreak [18,10,19,9],\allowbreak [17,5,18,4],\allowbreak [3,17,4,16],\allowbreak [14,7,15,8],\allowbreak [11,26,12,27],\allowbreak [27,10,28,11],\allowbreak [25,7,26,6],\allowbreak \mathbf{[5,1,6,28]}]$ \\[3pt]
\kn{13n}{221} & 13 & $0_1$ & $[[22,16,17,21],\allowbreak [16,3,4,17],\allowbreak [5,20,21,4],\allowbreak \mathbf{[1,2,11,7]},\allowbreak [3,14,15,2],\allowbreak \mathbf{[10,11,15,12]},\allowbreak [18,19,9,10],\allowbreak [6,19,20,5],\allowbreak [13,23,18,12],\allowbreak [14,22,23,13],\allowbreak [7,8,24,25],\allowbreak [6,1,25,26],\allowbreak [8,9,26,24]]$ \\[3pt]
\kn{13n}{436} & 13 & $0_1$ & $[[25,26,21,22],\allowbreak [24,25,2,3],\allowbreak [23,24,3,4],\allowbreak \mathbf{[26,23,10,11]},\allowbreak \mathbf{[1,2,14,15]},\allowbreak [18,7,8,17],\allowbreak [19,6,7,18],\allowbreak [16,8,9,15],\allowbreak [12,19,20,11],\allowbreak [22,21,13,14],\allowbreak [13,20,17,16],\allowbreak [10,5,6,12],\allowbreak [4,1,9,5]]$ \\[3pt]
\kn{13n}{447} & 13 & $10_{91}$ & $[[22,20,6,7],\allowbreak [24,25,21,22],\allowbreak [25,26,20,21],\allowbreak [24,10,11,23],\allowbreak [16,11,12,15],\allowbreak [15,12,13,14],\allowbreak [8,9,19,14],\allowbreak [7,8,13,10],\allowbreak \mathbf{[9,6,5,1]},\allowbreak [2,18,19,1],\allowbreak [3,17,18,2],\allowbreak [17,3,4,16],\allowbreak [23,4,5,26]]$ \\[3pt]
\kn{13n}{636} & 13 & $0_1$ & $[[4,23,20,3],\allowbreak [6,11,12,9],\allowbreak [7,10,11,6],\allowbreak \mathbf{[3,7,8,5]},\allowbreak \mathbf{[21,13,10,20]},\allowbreak [12,13,14,15],\allowbreak [4,5,1,2],\allowbreak [2,1,19,18],\allowbreak [8,9,15,16],\allowbreak [18,19,16,17],\allowbreak [22,23,26,24],\allowbreak [17,14,25,26],\allowbreak [21,22,24,25]]$ \\[3pt]
\kn{13n}{1439} & 13 & $0_1$ & $[[9,10,7,8],\allowbreak [21,22,6,7],\allowbreak [8,6,22,23],\allowbreak [12,9,23,24],\allowbreak [4,24,25,5],\allowbreak [5,1,2,4],\allowbreak \mathbf{[11,12,2,3]},\allowbreak [1,14,15,3],\allowbreak \mathbf{[26,13,14,25]},\allowbreak [13,17,16,15],\allowbreak [17,19,18,16],\allowbreak [10,11,20,21],\allowbreak [19,26,20,18]]$ \\[3pt]
\kn{13n}{2251} & 13 & $0_1$ & $[[21,24,25,20],\allowbreak [19,14,15,18],\allowbreak [14,1,2,13],\allowbreak [20,25,26,19],\allowbreak [15,16,10,11],\allowbreak [17,18,11,12],\allowbreak \mathbf{[17,3,4,21]},\allowbreak [13,8,9,16],\allowbreak \mathbf{[12,10,6,7]},\allowbreak [8,5,6,9],\allowbreak [3,7,5,2],\allowbreak [1,22,23,4],\allowbreak [22,26,24,23]]$ \\[3pt]
\kn{13n}{2379} & 13 & $0_1$ & $[[1,9,2,8],\allowbreak [18,8,19,7],\allowbreak [6,20,7,19],\allowbreak [20,6,21,5],\allowbreak [4,13,5,14],\allowbreak [14,3,15,4],\allowbreak \mathbf{[2,22,3,21]},\allowbreak [9,1,10,26],\allowbreak [25,11,26,10],\allowbreak [11,25,12,24],\allowbreak [16,24,17,23],\allowbreak [22,16,23,15],\allowbreak \mathbf{[12,18,13,17]}]$ \\[3pt]
\kn{13n}{2639} & 13 & $0_1$ & $[[3,9,10,2],\allowbreak [9,24,25,10],\allowbreak [18,12,13,17],\allowbreak [7,8,13,14],\allowbreak [19,11,12,18],\allowbreak [1,5,4,3],\allowbreak [20,16,11,19],\allowbreak [21,15,16,20],\allowbreak \mathbf{[2,14,15,1]},\allowbreak [5,21,22,4],\allowbreak \mathbf{[24,22,17,23]},\allowbreak [25,26,6,7],\allowbreak [8,6,26,23]]$ \\[3pt]
\kn{13n}{2809} & 13 & $0_1$ & $[\mathbf{[1,6,2,7]},\allowbreak \mathbf{[2,23,3,24]},\allowbreak [14,4,15,3],\allowbreak [4,20,5,19],\allowbreak [10,5,11,6],\allowbreak [7,16,8,17],\allowbreak [17,8,18,9],\allowbreak [9,26,10,1],\allowbreak [22,12,23,11],\allowbreak [12,22,13,21],\allowbreak [20,14,21,13],\allowbreak [24,15,25,16],\allowbreak [18,25,19,26]]$ \\[3pt]
\kn{13n}{2907} & 13 & $0_1$ & $[[1,15,2,14],\allowbreak [13,3,14,2],\allowbreak [12,10,13,9],\allowbreak [20,11,21,12],\allowbreak [10,21,11,22],\allowbreak \mathbf{[8,18,9,17]},\allowbreak [24,8,25,7],\allowbreak [6,24,7,23],\allowbreak [18,6,19,5],\allowbreak [4,20,5,19],\allowbreak \mathbf{[22,4,23,3]},\allowbreak [15,1,16,26],\allowbreak [25,17,26,16]]$ \\[3pt]
\kn{13n}{3033} & 13 & $0_1$ & $[[1,13,2,12],\allowbreak [11,21,12,20],\allowbreak [19,11,20,10],\allowbreak [9,3,10,2],\allowbreak \mathbf{[15,8,16,9]},\allowbreak [24,7,25,8],\allowbreak [6,25,7,26],\allowbreak [5,18,6,19],\allowbreak [17,4,18,5],\allowbreak [3,16,4,17],\allowbreak \mathbf{[21,26,22,1]},\allowbreak [23,14,24,15],\allowbreak [13,22,14,23]]$ \\[3pt]
\kn{13n}{3108} & 13 & $0_1$ & $[[14,13,15,16],\allowbreak [13,14,1,2],\allowbreak \mathbf{[6,12,8,5]},\allowbreak [3,1,11,12],\allowbreak [2,3,6,7],\allowbreak [24,25,17,18],\allowbreak \mathbf{[10,11,16,17]},\allowbreak [18,15,21,19],\allowbreak [7,4,20,21],\allowbreak [9,4,5,8],\allowbreak [19,20,23,22],\allowbreak [9,10,25,26],\allowbreak [22,23,26,24]]$ \\[3pt]
\kn{13n}{3589} & 13 & $0_1$ & $[[1,15,2,14],\allowbreak [26,22,1,21],\allowbreak \mathbf{[6,25,7,26]},\allowbreak [24,11,25,12],\allowbreak [10,23,11,24],\allowbreak [22,7,23,8],\allowbreak [3,20,4,21],\allowbreak [19,4,20,5],\allowbreak [5,18,6,19],\allowbreak \mathbf{[17,12,18,13]},\allowbreak [9,16,10,17],\allowbreak [15,8,16,9],\allowbreak [13,3,14,2]]$ \\[3pt]
\kn{13n}{3733} & 13 & $0_1$ & $[[14,15,12,13],\allowbreak [5,6,15,14],\allowbreak [13,12,11,9],\allowbreak [4,6,7,3],\allowbreak [8,17,18,7],\allowbreak [16,2,3,18],\allowbreak [17,24,25,16],\allowbreak [5,21,19,8],\allowbreak [4,1,10,11],\allowbreak \mathbf{[20,21,9,10]},\allowbreak [20,23,22,19],\allowbreak \mathbf{[2,25,26,1]},\allowbreak [23,26,24,22]]$ \\[3pt]
\kn{13n}{4025} & 13 & $0_1$ & $[[13,9,10,18],\allowbreak [2,14,15,1],\allowbreak \mathbf{[12,6,7,11]},\allowbreak [11,17,18,10],\allowbreak [9,3,4,12],\allowbreak [3,13,14,2],\allowbreak \mathbf{[5,6,19,20]},\allowbreak [1,22,19,4],\allowbreak [8,16,17,7],\allowbreak [22,15,16,21],\allowbreak [20,21,24,23],\allowbreak [8,5,25,26],\allowbreak [23,24,26,25]]$ \\[3pt]
\kn{13n}{4237} & 18 & $0_1$ & $[[34,35,26,27],\allowbreak \mathbf{[22,23,17,18]},\allowbreak [29,30,21,22],\allowbreak [32,33,9,10],\allowbreak [30,31,20,21],\allowbreak [35,36,25,26],\allowbreak [33,34,8,9],\allowbreak [27,28,7,8],\allowbreak [28,29,18,19],\allowbreak [36,32,23,24],\allowbreak [31,25,24,20],\allowbreak [19,17,15,16],\allowbreak \mathbf{[15,4,5,14]},\allowbreak [16,14,12,13],\allowbreak [5,6,11,12],\allowbreak [13,11,1,2],\allowbreak [10,7,3,4],\allowbreak [6,3,2,1]]$ \\[3pt]
\end{longtable}}

\section{Correction terms for the Bernhard--Jablan example}\label{app:bj-correction-terms}

This appendix records the correction-term data used in
Section~\ref{sec:bj-counterexample}. Write
$Y_{69}=\dbc(\kn{12n}{491})$ and $Y_{33}=\dbc(\kn{13n}{3370})$.
Each group $H^2(Y_D;\bbz)$ is cyclic of order $D$. Choose a generator $g_D$
and label the $\operatorname{Spin}^c$ structures by
$c_1(\mathfrak{s}_k)=k g_D$. Since $D$ is odd, this
labels every structure uniquely, and $k=0$ is the spin structure.
The generators are those in the archived Goeritz presentations; their
self-pairings, in the convention given by the inverse Goeritz matrix, are
$28/69$ and $25/33$, respectively, modulo $\bbz$.
Conjugation gives $d(Y_D,\mathfrak{s}_{-k})=d(Y_D,\mathfrak{s}_k)$,
so it suffices to list $0\leq k\leq(D-1)/2$.

Table~\ref{tab:bj-d69} gives all correction terms of $Y_{69}$.
For $Y_{33}$, Table~\ref{tab:bj-d33} gives the candidate set $C_k$
containing $d(Y_{33},\mathfrak{s}_k)$; a single entry denotes a singleton
set. Only the conjugate pairs represented by $k=7$ and $k=12$ remain
undetermined. Allowing all three choices for each pair gives a set of
nine candidate vectors containing the actual correction-term vector.
The surgery obstruction in Section~\ref{sec:bj-counterexample} shows
that none is compatible with the half-integral surgery required by
unknotting number one.

\begingroup
\small
\setlength{\tabcolsep}{9pt}
\setlength{\LTleft}{\fill}
\setlength{\LTright}{\fill}
\setlength{\LTcapwidth}{\textwidth}
\renewcommand{\arraystretch}{1.0}
\begin{longtable}{@{}rl@{\hspace{2em}}rl@{\hspace{2em}}rl@{}}
\caption{Correction terms $d_k=d(Y_{69},\mathfrak{s}_k)$ for
$\kn{12n}{491}$.}\label{tab:bj-d69}\\
\toprule
$k$ & $d_k$ & $k$ & $d_k$ & $k$ & $d_k$ \\
\midrule
\endfirsthead
\toprule
$k$ & $d_k$ & $k$ & $d_k$ & $k$ & $d_k$ \\
\midrule
\endhead
\bottomrule
\endfoot
\bottomrule
\endlastfoot
0 & $0$ & 12 & $14/23$ & 24 & $10/23$ \\
1 & $76/69$ & 13 & $10/69$ & 25 & $28/69$ \\
2 & $28/69$ & 14 & $-8/69$ & 26 & $40/69$ \\
3 & $-2/23$ & 15 & $-4/23$ & 27 & $22/23$ \\
4 & $-26/69$ & 16 & $-2/69$ & 28 & $106/69$ \\
5 & $-32/69$ & 17 & $22/69$ & 29 & $22/69$ \\
6 & $-8/23$ & 18 & $20/23$ & 30 & $30/23$ \\
7 & $-2/69$ & 19 & $-26/69$ & 31 & $34/69$ \\
8 & $34/69$ & 20 & $40/69$ & 32 & $-8/69$ \\
9 & $28/23$ & 21 & $-6/23$ & 33 & $-12/23$ \\
10 & $10/69$ & 22 & $76/69$ & 34 & $-50/69$ \\
11 & $88/69$ & 23 & $2/3$ &  &  \\

\end{longtable}

\begin{longtable}{@{}rl@{\hspace{2em}}rl@{}}
\caption{Candidate sets $C_k$ for $d(Y_{33},\mathfrak{s}_k)$ associated
with $\kn{13n}{3370}$.}\label{tab:bj-d33}\\
\toprule
$k$ & $C_k$ & $k$ & $C_k$ \\
\midrule
\endfirsthead
\toprule
$k$ & $C_k$ & $k$ & $C_k$ \\
\midrule
\endhead
\bottomrule
\endfoot
\bottomrule
\endlastfoot
0 & $0$ & 9 & $12/11$ \\
1 & $-2/33$ & 10 & $64/33$ \\
2 & $58/33$ & 11 & $2/3$ \\
3 & $16/11$ & 12 & $\{-8/11,\,3/11,\,14/11\}$ \\
4 & $34/33$ & 13 & $-8/33$ \\
5 & $16/33$ & 14 & $4/33$ \\
6 & $-2/11$ & 15 & $4/11$ \\
7 & $\{-32/33,\,1/33,\,34/33\}$ & 16 & $16/33$ \\
8 & $4/33$ &  &  \\

\end{longtable}
\endgroup

\end{document}